\documentclass[a4paper,10pt]{article}
\usepackage{geometry}
\usepackage[english]{babel}
\usepackage{hyperref}
\usepackage{amsmath,amssymb,stmaryrd,enumerate,latexsym,theorem,bbm,color}

\catcode`\<=\active \def<{
\fontencoding{T1}\selectfont\symbol{60}\fontencoding{\encodingdefault}}
\newcommand{\assign}{:=}
\newcommand{\asterisk}{\mathord{*}}
\newcommand{\mathd}{\mathrm{d}}
\newcommand{\nobracket}{}
\newcommand{\tmaffiliation}[1]{\\ #1}

\newcommand{\tmem}[1]{{\em #1\/}}
\newcommand{\tmemail}[1]{\\ \textit{Email:} \texttt{#1}}
\newcommand{\tmmathbf}[1]{\ensuremath{\boldsymbol{#1}}}
\newcommand{\tmop}[1]{\ensuremath{\operatorname{#1}}}
\newcommand{\tmsep}{, }

\newcommand{\tmtextbf}[1]{{\bfseries{#1}}}
\newcommand{\tmtextit}[1]{{\itshape{#1}}}

\newcommand{\invbreve}{\breve}
\newenvironment{proof}{\noindent\textbf{Proof\ }}{\hspace*{\fill}$\Box$\medskip}
\newenvironment{proof*}[1]{\noindent\textbf{#1\ }}{\hspace*{\fill}$\Box$\medskip}
\newtheorem{theorem}{Theorem}
\newtheorem{corollary}[theorem]{Corollary}
\newtheorem{definition}[theorem]{Definition}
\newtheorem{lemma}[theorem]{Lemma}
\newtheorem{proposition}[theorem]{Proposition}
{\theorembodyfont{\rmfamily}\newtheorem{remark}[theorem]{Remark}}
\newcommand{\tmkeywords}{\textbf{Keywords:} }
\newcommand{\tmmsc}{\textbf{A.M.S. subject classification:} }

\newcommand{\descriptionparagraphs}[1]{\begin{description}#1\end{description}}
\newcommand{\textdots}[0]{...}

\begin{document}

\title{
  The elliptic stochastic quantization\\
  of some two dimensional Euclidean QFTs
}

\author{
  Sergio Albeverio, Francesco C. De Vecchi and Massimiliano Gubinelli
  \tmaffiliation{Hausdorff Center of Mathematics \&\\
  Institute of Applied Mathematics, \\
  University of Bonn, Germany}
  \tmemail{albeverio@iam.uni-bonn.de}
  \tmemail{francesco.devecchi@uni-bonn.de}
  \tmemail{gubinelli@iam.uni-bonn.de}
}

\maketitle

\begin{abstract}
  We study a class of elliptic SPDEs with additive Gaussian noise on
  $\mathbb{R}^2 \times M$, with $M$ a $d$-dimensional manifold equipped with a
  positive Radon measure, and a real-valued non linearity given by the
  derivative of a smooth potential $V$, convex at infinity and growing at most
  exponentially. For quite general coefficients and a suitable regularity of
  the noise we obtain, via the dimensional reduction principle discussed in
  our previous paper~{\cite{Albeverio2018elliptic}}, the identity between the
  law of the solution to the SPDE evaluated at the origin with a Gibbs type
  measure on an abstract Wiener space over $M$. The results are then applied
  to the elliptic stochastic quantization equation for the scalar field with
  polynomial interaction over $\mathbb{T}^2$, and with exponential interaction
  over $\mathbb{R}^2$ (known also as H{\o}egh-Krohn or Liouville model in the
  literature). In particular for the exponential interaction case, the
  existence and uniqueness properties of solutions to the elliptic equation
  over $\mathbb{R}^{2 + 2}$ is derived as well as the dimensional reduction
  for the values of the ``charge parameter'' $\sigma = \frac{\alpha}{2
  \sqrt{\pi}} < \sqrt{4 \left( 8 - 4 \sqrt{3} \right) \pi} \simeq \sqrt{4.23
  \pi}$, for which the model has an Euclidean invariant probability measure
  (hence also permitting to get the corresponding relativistic invariant model
  on the two dimensional Minkowski space).
\end{abstract}

\tmmsc{60H15}{\tmsep}{81T08}{\tmsep}{81T40}

\tmkeywords{stochastic quantization}{\tmsep}{elliptic stochastic partial
differential equations}{\tmsep}{dimensional reduction}{\tmsep}{Euclidean
quantum field theory}{\tmsep}{exponential interaction}

{\tableofcontents}

\section{Introduction}

The Euclidean approach to quantization of relativistic non-linear wave
equations requires the construction of certain probability measures supported
on distributions satisfying a set of quite strong requirements among which
invariance under the full group of rigid motions of the Euclidean space and
reflection positivity. Such constructions have succeeded in the case where the
dimension $d$ of the Euclidean space is equal to $1, 2, 3$ (that is also the
dimension of the space-time of the original relativistic equation). The
non-linearity is quite general if $d = 1$, whereas it is restricted to the
derivatives of polynomials of even degree respectively suitable superpositions
of trigonometric or exponential functions for $d = 2$, or a cubic monomial
function for $d = 3$ (the $\varphi^4_3$ model). In these cases a number of
interesting physical and mathematical properties of the quantized relativistic
fields models have been put in evidence.

In recent years new methods coming from the study of singular stochastic
partial differential equations (SSPDEs) have been developed and have opened
the possibility of new constructions of such Euclidean measures (and other
similar measures associated with problems coming from other areas of
applications, like statistical mechanics, hydrodynamics, wave propagation in
random media {\textdots}). In these approaches the relevant Euclidean
invariant (probability) measures are obtained as invariant measures for
certain parabolic semilinear SPDEs called stochastic quantization equations
(SQEs).

Typically the measures to be constructed have the heuristic form
\[ \mu (\mathd \varphi) = Z^{- 1} e^{- S (\varphi)} \mathcal{D} \varphi, \]
where $\varphi = \varphi (x) \in \mathbb{R}$, $x \in \mathbb{R}^d$,
$\mathcal{D} \varphi = \Pi_{x \in \mathbb{R}^d} \mathd \varphi (x)$ is a
``flat measure'', $Z$ is a ``normalization constant'', $S (\varphi) = S_0
(\varphi) + \lambda S_{\tmop{int}} (\varphi)$ with $S_0 (\varphi) =
\frac{1}{2} \left( \int_{\mathbb{R}^d} | \nabla \varphi (x) |^2 \mathd x + m^2
\int_{\mathbb{R}^d} | \varphi (x) |^2 \mathd x \right)$, and where $\lambda
\geqslant 0$, $m \geqslant 0$ are parameters,
\[ S_{\tmop{int}} (\varphi) = \int_{\mathbb{R}^d} V (\varphi (x)) \mathd x, \]
for some (non-linear) measurable real valued function $V$ over $\mathbb{R}$.
Note that $\mu (\mathd \varphi)$ is heuristically invariant under Euclidean
transformations (translations and rotations and also additional reflections,
in the case where $V$ is an even function).

The associated stochastic quantization equation (first introduced
in~{\cite{parisi_perturbation_1981}}) is of the form
\[ \mathd X_{\tau} = (\Delta - m^2) X_{\tau} \mathd t - V' (X_{\tau}) \mathd
   t + \mathd W_{\tau} \]
where $\tau$ is an additional ``computer time'', $\Delta$ is the Laplacian in
$\mathbb{R}^d$, $X_{\tau} = X_{\tau} (x)$, $x \in \mathbb{R}^d$, is a random
variable taking values, for fixed $\tau$, in the space of ``(generalized)
functions'' of $x$, $\mathd W_{\tau} (x)$ is a Gaussian white noise in $\tau$
and $x$, and $V'$ in the derivative of $V$.

Following a terminology used especially when $V$ is a polynomial $P$, we call
the above SPDE the $V (\varphi)_d$--stochastic quantization equation (SQE).
Various detailed results have been obtained on $P (\varphi)_2$ SQEs, see the
introductions of~{\cite{albeverio_invariant_2017}}
and~{\cite{gubinelli_pde_2018,GH18}} for references. In particular
in~{\cite{da_prato_strong_2003}} there were obtained a strong solution and a
unique ergodic measure in the case where $\mathbb{R}^2$ is replaced by the
$2$-torus $\mathbb{T}^2$.

As for the case of the $\varphi^4_3$ SQE, a breakthrough for local in time
solutions of the SQE on the $3$-torus $\mathbb{T}^3$ came
from~{\cite{Hairer2014}}, see also~{\cite{CaCh,GuImPe}}. More recently still
for the case of the $3$-torus $\mathbb{T}^3$ a construction of invariant
solutions has been given in~{\cite{albeverio_invariant_2017}}. This can be
looked upon as a construction of the $\varphi^4_3$-measure on $\mathbb{T}^3$,
by methods different from the previous ones provided by mathematical physics
approaches, see references in~{\cite{Glimm1987}}. The methods
in~{\cite{albeverio_invariant_2017}} have been simplified and extended to the
case of $\mathbb{R}^3$ in~{\cite{gubinelli_pde_2018}}, yielding a full
alternative construction of the measures which qualify to be called
$\varphi^4_3$ Euclidean measures. A previous approach following
{\cite{Hairer2014}} is in {\cite{MoWe}}. A number of open problems has been
mentioned in~{\cite{albeverio_invariant_2017,gubinelli_pde_2018}}, the main
one concerns the uniqueness of the invariant measures.

Other approaches are however possible if one understands stochastic
quantization in the broader sense of associating to a given heuristic
``Euclidean measure'' a stochastic process having this very measure as
suitable marginal. A first approach consists in constructing a jump type
process via infinite dimensional quasi-regular jump-type Dirichlet
forms~{\cite{Yoshida,Albeverio2002,Yoshida2008,Yoshida2010}} (for such
Dirichlet forms see also~{\cite{ARu,AFeHKL}}), that has by construction the
given Euclidean measure as invariant measure. Another approach obtains the
invariant measure of an SQE in dimension $d$ by looking at solutions of a
suitable associated elliptic SPDE in dimension $d + 2$ and restricting the
solution to the $d$ dimensional Euclidean space obtained sending to zero the
two additional variables. This latter procedure is known in the physical
literature as Parisi-Sourlas dimensional reduction and has been implemented by
using algebraic methods of supersymmetry (see~{\cite{parisi_random_1979}}).
For the case of a regularized non-linear term in the original SQE and
regularized noise it has been studied from a mathematical point of view
in~{\cite{Klein1984}}.

Elliptic stochastic PDEs related to stochastic quantization of the
$\varphi_2^4$ and $\varphi_3^4$ models respectively have been discussed both
on the torus and in all space in~{\cite{GH18}}. A systematic mathematical
investigation of the mechanism of dimensional reduction has been initiated
in~{\cite{Albeverio2018elliptic}}, for the $V (\varphi)_0$ model, for both the
cases where $V$ is a convex function and also the case where $V$ is
non-convex, where an interesting phenomenon of non-uniqueness of solutions has
been put in evidence.

In fact in~{\cite{Albeverio2018elliptic}} an explicit formula for the law of
the solution of a class of elliptic SPDE in $\mathbb{R}^2$ taken at the origin
has been obtained, by means of a rigorous version of dimensional reduction.
Besides proving an instance of elliptic stochastic quantization for scalar
fields and for the case of underlying Euclidean dimension $d = 0$, it also
provided a realization of the relation between a supersymmetric quantum field
model and an elliptic SPDE in $2$ dimensions.

\

The present paper has two principal aims. The first one is to prove the
dimensional reduction principle for elliptic SPDEs with non singular noise for
equations in $d + 2$ dimensions with $d > 0$. The second is to extend the
result to elliptic singular SPDE (i.e. with space-time white noise) having
applications in stochastic quantization program and analyzing the
corresponding dimensional reduction with respect to removal of spatial
cut-offs.

We restrict our attention to singular SPDEs with $d = 2$ and analyze
polynomial and exponential interactions. In particular, in the case of
potentials of the form $\exp (\sigma \varphi)_2$, for suitable real $\sigma$
we are able to complete the dimensional reduction picture removing all the
spatial cutoffs (i.e. in both the ``fictitious'' and ``real'' spatial
variables).

\

Before passing to a more detailed description of our results, let us stress
the importance of the $\exp (\sigma \varphi)_2$ model (and related ones) in
relativistic and Euclidean quantum field theory. The $\exp (\sigma \varphi)_2$
model over $\mathbb{R}^2$ has been first introduced by
H{\o}egh-Krohn~{\cite{Krohn1971}} (see also~{\cite{Krohn1974}}
and~{\cite{Simon1974}}, p.~178~and~307-313) and constructed for an interaction
of the form
\[ U_g (\varphi) : = \int : e^{\sigma \varphi (z)} : g (z) \mathd z \mathd \nu
   (\sigma), \]
   where $\nu$ positive bounded measure, for any positive space cut-off function $g \in L^1 (\mathbb{R}^2) \cap L^2
(\mathbb{R}^2)$ and any positive finite measure $\nu$ with support in $(-
\sqrt{4 \pi}, \sqrt{4 \pi})$, in the sense that the interacting measure
\[ \mu_g (\mathd \varphi) : = \frac{e^{- U_g (\varphi)} \mu_0 (\mathd
   \varphi)}{\int e^{- U_g (\varphi)} \mu_0 (\mathd \varphi)}, \]
for $m > 0$, is absolutely continuous with respect to Nelson's free field
Gaussian measure $\mu_0$ with mean zero and covariance $(m^2 - \Delta_z)^{-
1}$, $z \in \mathbb{R}^2$ (realized as a probability measure on, for example,
$\mathcal{S}' (\mathbb{R}^2)$). See~{\cite{Krohn1974,Krohn1971}}
and~{\cite{Simon1974}} p.~178 for further properties of the measure $\mu_g$.
This model has been discussed in the infinite volume limit $g \rightarrow 1$,
that is unique, in~{\cite{Krohn1974}}, under the more restrictive assumption
that $\tmop{supp} (\nu) \subset \left[ - \frac{4}{\sqrt{\pi}},
\frac{4}{\sqrt{\pi}} \right]$ (see also {\cite{Simon1974}} pp. 307-313 and
{\cite{FroPa}}).

\

In fact in~{\cite{Krohn1974}} the authors prove that the measure of the $\exp
(\sigma \varphi)_2$ model for $| \sigma | < \sqrt{4 \pi}$ (or more generally
of suitable superpositions $\int \tmop{Ch} (\sigma \varphi)_2 \mathd \nu
(\sigma)$, where $\tmop{supp} (\nu) \subset \left[ - \frac{4}{\sqrt{\pi}},
\frac{4}{\sqrt{\pi}} \right]$, of these models) satisfies all axioms of
Euclidean quantum field theory and leads to relativistic quantum fields over
two-dimensional Minkowski space-time satisfying all Wightman axioms with
interaction and a positive, finite mass gap at the lower end of the spectrum
of the corresponding Hamiltonian.

Further properties of such models were discussed in,
e.g.,~{\cite{AHiPRoS2,AHiPRoS1,FroPa,Gie,HiKPS,Simon1974,Zegarli}}
and~{\cite{Albeverio2002,AlHKa,AFeHKL,AlLi,Liang}}. The case $| \sigma | <
\sqrt{8 \pi}$ was discussed by~S.~Kusuoka~in~{\cite{Ku}} and the relation with
independent work on multiplicative chaos by J. P. Kahane~{\cite{Ka}} (see also
below) was pointed out by H. Sato. In~{\cite{Krohn1974,AHK2}} and
also~{\cite{Simon1974}} estimates in $L^p$ (for the interaction in a bounded
region) were given for $| \sigma | < \sqrt{\frac{4 \pi}{p - 1}}$. The Gaussian
character of the model for $| \sigma |$ sufficiently large or for any $\sigma$
if the Euclidean space dimension $d$ satisfies $d \geqslant 3$ \ has been
pointed out in~{\cite{AHK2}} and~{\cite{AGaHK}}. The relevance of the $\exp
(\sigma \varphi)_2$ model for Polyakov string theory has been discussed
in~{\cite{AlKrPaSc1986,AlKrPaSc1992,AJPS}} ($| \sigma | < \sqrt{4 \pi}$
corresponding to the embedding dimension $D < 13$) and rediscovered more
recently in connection with topics like Liouville model, quantum gravity and
multiplicative chaos. Let us mention in this
connection~{\cite{AnKa,VaRhShe,RhVa}} see also
{\cite{BeShXi2014,DuSh2019,Gottschalk2013,Sha2016}} (for the literature for
other approaches to Liouville type models not directly connected with
probabilistic methods see, e.g.,~{\cite{Bahns2018}}).

Finally a connection of the $\exp (\sigma \varphi)_2$ model with the
irreducibility of a unitary representation of the group of mappings of a
manifold into a compact Lie group has been pointed out and studied
in~{\cite{AHKMT}}, see also the recent work~{\cite{AGoVe}}.

\

The $\exp (\sigma \varphi)_2$ model is also closely connected with the $\sin
(\sigma \varphi)_2$ model (``Sine-Gordon equation'') as discussed
in~{\cite{AHK3,AlHKa,FrSei}}. In the latter reference all Wightman axioms are
proved again in the region $| \sigma | < \sqrt{4 \pi}$.

Both in the $\exp (\sigma \varphi)_2$ model and the $\sin (\sigma \varphi)_2$
model a replacement of the $\exp$ (respectively $\sin$) function by a Wick
ordered version is needed (see~{\cite{AHaRu}}) but also suffices for the
existence and non-triviality of the model for $| \sigma |$ up to $\sqrt{4
\pi}$. For larger values of $| \sigma |$, in the case of the $\sin (\sigma
\varphi)_2$ model, up to $\sqrt{8 \pi}$, further renormalization by
counter-terms is required (see~{\cite{GaNi}}).

\

The study of the stochastic quantization equation associated with the latter
class of models has been initiated in~{\cite{AHaRu}} (where strong solutions
have been discussed and the necessity of renormalization has been pointed
out). In the case where $\mathbb{R}^2$ (or $\mathbb{T}^2$) is replaced by
$\mathbb{R}$ (respectively $\mathbb{T}$), i.e. for the model $\exp (\sigma
\varphi)_1$, a deeper analysis is possible and has been pursued
in~{\cite{AKaRo}}, where the existence of solutions and the strong uniqueness
of the invariant measure are proven (for the corresponding stochastic
quantization of $P (\varphi)_1$ see~{\cite{Iwata1987}}).

The case of the SQE in $d = 2$ and with a regularized noise on the torus (and
corresponding changes in the coefficients of the stochastic quantization
equation needed to keep the same invariant measure) has been discussed
in~{\cite{Mi}} (see also~{\cite{AKaMiRo}} for further developments), where
existence of solutions was proven using essentially the properties of the
$\exp (\sigma \varphi)_2$ model established in~{\cite{Krohn1974,FroPa}} and
methods of~{\cite{da_prato_strong_2003}}. Uniqueness problems are also
discussed in~{\cite{AKaMiRo}} in conjunction with an approach using Dirichlet
forms, for all $\sigma^2 < 4 \pi$, in a setting similar to~{\cite{Mi}} (see
the introduction of {\cite{AKaMiRo}} for further references).

Hairer and Shen~{\cite{HaiShe}} introduced powerful methods to handle the
dynamical $\sin (\sigma \varphi)_2$ model on $\mathbb{T}^2$, and via
regularity structures local existence is shown up to $\sigma^2 < \frac{16}{3}
\pi$. More recently these results has been extended in~{\cite{ChaHaiShe}} to
cover all the subcritical regime $\sigma^2 < 8 \pi$.

\

Concerning the exponential interaction, the work of
Garban~{\cite{Garban2018}} appeared recently in which the author studies the
SQE on the torus $\mathbb{T}^2$ and on the sphere $\mathbb{S}^2$. After
subtracting the solution to the linear equation Garban obtains an SPDE driven
by a multi-fractal and intermitted multiplicative chaos. When $| \sigma | <
\left( 4 - 2 \sqrt{3} \right) \sqrt{\pi} \left( < \sqrt{4 \pi} \right)$ (a
regime correspondingly called Da Prato-Debussche phase) he shows the existence
of a strong solution and the convergence of the solution to the equation with
regularized noise to the singular one both on the torus $\mathbb{T}^2$ and on
the sphere $\mathbb{S}^2$ (see Theorem~1.7 and Theorem~1.9
in~{\cite{Garban2018}}). The method used is based on the Besov regularity of
the Gaussian multiplicative chaos (related to the theory of~{\cite{Ka}}). The
paper~{\cite{Garban2018}} also refers to other interesting relations to
quantum gravity, conformal fields theory, strings theory, multiplicative chaos
and exciting new results on random measures. A result on existence and
uniqueness for solutions of SQE (without the proof of the convergence of the
solution to the equation with regularized noise to the singular one) is also
proved (Theorem 1.11 in~{\cite{Garban2018}}) for $| \sigma | < \left( 4 - 2
\sqrt{2} \right) \sqrt{\pi} \left( < \sqrt{4 \pi} \right)$. Moreover a
comparison with the SQE for the Sine-Gordon model is provided (cfr. Section 7
in~{\cite{Garban2018}}).\footnote{After the submission of this article for
publication we have become aware of the papers
{\cite{HoKaKu,HoKaKu2020,OhTriYu2020,OhTriYu}} which address the stochastic
quantization of the exponential model.}

\

Let us briefly review and comment the main results obtained in the present
paper. In Section~\ref{section_discrete} we study the case of the following
elliptic SPDE (equation~{\eqref{eq:main1}} in Section~\ref{section_discrete})
\[ (- \Delta_x + m^2 + \mathfrak{L}) (\phi) (x, z) + \mathcal{A}^2 [f g
   \partial V (\phi)] (x, z) = \xi^A (x, z) \]
on $\mathbb{R}^2 \times M$, with $M$ a $d$-dimensional manifold with or
without boundary, equipped with a positive Radon measure $\mathd z$ and a
real-valued non-linearity given by the derivative $\partial V$ of a positive smooth
function $V$ defined on the real line growing at most exponentially at
infinity. We use the shorthand $f g V'(\phi)$ to mean the function $f g \partial V(\phi): \mathbb{R}^2\times M  \longrightarrow  \mathbb{R}$ such that $ (x, z) \longmapsto f (x) g (z) \partial V
(\phi (x, z)) $ (where $x\in \mathbb{R}^2$ and $z\in M$).
%
 Additional assumptions are as follows. The function $V$ is
taken to the convex or ``quasi convex'' (see Hypothesis~QC below).
$\mathfrak{L}$ and $\mathcal{A}$ are two positive self-adjoint operators
acting on $L^2 (M)$, $\mathcal{A}$ is bounded, $\mathfrak{L}$ is
not-necessarily bounded but has a dense domain in $L^2 (M)$, $\mathcal{A}$ and
$\mathfrak{L}$ commute and both have completely discrete spectrum with some
additional hypotheses on the eigenfunctions and the eigenvalues of
$\mathcal{A}$. The noise $\xi^A$ is supposed to be Gaussian with mean zero and
covariance corresponding to a regularization given by $\mathcal{A}$. Finally
$f$ is a $C^2$ real-valued cut-off function on $\mathbb{R}^2$, decaying
exponentially at infinity and $g$ is a smooth real-valued cut-off with compact
support on $M$.

We introduce finite dimensional projections relative of~{\eqref{eq:main1}}
and prove the existence of a weak solution $\nu$ satisfying a form of the
dimensional reduction principle described in our previous
paper~{\cite{Albeverio2018elliptic}}, extended to the $d$ dimensional case
with the presence of a non-trivial operator $\mathfrak{L}$. See
Theorem~\ref{theorem_main} for a precise statement.

Next we prove extensions of these results in two directions: in
Section~\ref{subsection_extension1} we consider the case $M =\mathbb{R}^d$,
$\mathfrak{L} = - \Delta_z$ (cfr. Theorem~\ref{theorem_extension1}), in
Section~\ref{subsection_extension2} we consider the problem of the removal of
the cutoffs $f, g$ in the case where $V$ is a convex function (cfr.
Theorem~\ref{theorem_extension2}).

\

The first part of the present paper contains the first example of rigorous
dimensional reduction for elliptic SPDE in $d + 2$ dimension with $d > 0$.
Indeed, the only paper which to our knowledge study dimensional reduction from
the rigorous point of view before our work was~{\cite{Klein1984}} whose main
result is a theorem about the dimensional reduction of an integral in a space
of functions in $d + 2$ variables to a Gibbs type measure in a space of
functions in $d$ variables. In particular the authors do not attack the
problem of the relation with the elliptic SPDE. Furthermore
in~{\cite{Klein1984}} only polynomial potentials $V$ are considered and the
regularization of the noise is chosen to be compactly supported in Fourier
space. Our Hypothesis~QC on the potential $V$ and Hypotheses H$\mathcal{A}$
and H$\mathcal{A} 1$ for the regularization of the noise considered here are
quite more general. Let us stress that the results obtained under these more
general hypotheses use in essential way the SPDE formulation of dimensional
reduction.

\

In Section~\ref{section_exponential} we extend our results to singular SPDEs.
In particular the elliptic stochastic quantization of the exponential
interaction is proven, where $V$ is a suitable renormalized exponential
function $\exp (\alpha \phi - \infty)$, for $| \alpha | < \alpha_{\text{max}}$, where 
\begin{equation}\label{eq:alphamax}
 \alpha_{\text{max}}= 4 \sqrt{8 - 4
\sqrt{3}} \pi,
\end{equation}
(corresponding to the reduced index $| \sigma | <\frac{\alpha_{\text{max}}}{2\sqrt{\pi}}= \sqrt{4
\left( 8 - 4 \sqrt{3} \right) \pi}$), and the white noise is unregularized.
Existence and uniqueness is first proven for a suitable regularized model
using the results of the previous sections (cfr.
Theorem~\ref{theorem_regularexponential}). Subsequently the removal of the
regularization in the noise is achieved in the Besov space $B^s_{p, p, \ell}
(\mathbb{R}^4)$, where $- 1 < s < 0$ and $1 < p \leqslant 2$ are suitable
constants depending only on $\alpha$. Using the existence, uniqueness and
convergence results for the elliptic SPDEs over $\mathbb{R}^4$ proven in
Section~\ref{subsection_extension2}, the existence and uniqueness results for
the elliptic stochastic quantization equation for exponential interaction
follows in Theorem~\ref{theorem_exponentialmain}. The dimensional reduction
result is then obtained (cfr. Theorem~\ref{theorem_dimensionalreduction1})
with uniqueness for all $| \alpha | < \alpha_{\text{max}}$. Our last
result in this section, Theorem~\ref{theorem_dimensionalreduction2}, shows the
convergence and uniqueness of the equation when the spatial cut-off $g$ is
removed. As an easy corollary one obtains the full Euclidean invariance of the
law of the solutions.

Finally in Section~\ref{section_power} we prove the existence for singular
elliptic SPDE with Wick power non linearity in $d + 2 = 4$ dimension.
Furthermore we prove a dimensional reduction principle for some weak solutions
to SPDEs of this form providing the first example of elliptic stochastic
quantization for the polynomial interaction model $P (\varphi)_2$. Let us
stress that for polynomial interactions the problem of uniqueness of weak
solutions to the elliptic SPDE is open and this prevents a more detailed
analysis of the dimensional reduction.

\

The results contained in this second part of the paper gives the first
example of dimensional reduction for singular elliptic SPDE as originally
formulated by Parisi and Sourlas in~{\cite{parisi_random_1979}}. In particular
with respect to~{\cite{GH18}}, where the existence for singular elliptic SPDEs
with polynomial type non-linearity is proven for both $d + 2 = 4$ and, with
cubic type non-linearity, $d + 2 = 5$ dimension, in
Section~\ref{section_power} we prove, at least in the $d + 2 = 4$ dimensional
case, that they can be used as stochastic quantization equation for quantum
field theory. However many problems still remain open, for the polynomial
models, with respect to the removal of the spatial cutoffs.

\

Let us stress that, in Section~\ref{section_exponential}, we prove existence
and uniqueness of solutions for elliptic SPDE and the relation with the
Liouville measure in the regime $| \sigma | < \frac{\alpha_{\text{max}}}{2\sqrt{\pi}}$. This result is achieved using in an essential way two
properties of the exponential model: the fact that the Wick exponential of a
distribution is a positive measure (this fact is already exploited
in~{\cite{Garban2018}}) and the multifractality of Wick exponentials (see
Lemma~\ref{lemma_garban} for a precise formulation of this property). In
particular we prove the existence and uniqueness for the elliptic SPDE with
Wick exponential non-linearity using the space $B^s_{p, p, \ell}
(\mathbb{R}^4)$, instead of $B^{s'}_{\infty, \infty, \ell} (\mathbb{R}^4)$,
and by only the Da Prato-Debussche trick. This is possible since for any
exponent $| \sigma | <\frac{\alpha_{\text{max}}}{2\sqrt{\pi}}= \sqrt{4 \left( 8 - 4 \sqrt{3} \right) \pi}$ the noise
lives, beside that, inside $B^s_{p, p, \ell} (\mathbb{R}^4)$ for some $- 1 < s
< 0$ and $1 < p \leqslant 2$, instead for $| \sigma | \geqslant \left( 4 - 2
\sqrt{2} \right) \sqrt{\pi} \left( < \sqrt{4 \pi} < \frac{\alpha_{\text{max}}}{2\sqrt{\pi}} \right)$ the Wick power of the noise is in $B^s_{\infty,
\infty, \ell} (\mathbb{R}^4)$ with $s < - 2$. The other novelty is the fact
that we are able to solve the equation in the full space and we are then
easily able to prove the Euclidean invariance of the law.

The other best result, to our knowledge and before our paper, concerning
stochastic quantization of exponential interaction, is the
paper~{\cite{Garban2018}}, which studied the stochastic quantization in the
parabolic setting for the charge parameter $| \sigma | < \left( 4 - 2 \sqrt{2}
\right) \sqrt{\pi} \left( < \sqrt{4 \pi} \right)$. We think that the analytic
methods, developed here in Section \ref{subsection_analytic}, joined with the
probabilistic results, typical of the parabolic setting obtained
in~{\cite{Garban2018}}, can also be applied to the parabolic case for the
regime $| \sigma | < \frac{\alpha_{\text{max}}}{2\sqrt{\pi}}$.

After the submission of this article for publication we became aware of
{\cite{HoKaKu,HoKaKu2020}} that address the stochastic quantization of the
exponential model in the parabolic setting on the two dimensional torus
$\mathbb{T}^2$ for $\sigma^2 < 4 \pi$ and $\sigma^2 < 8 \pi$ respectively. Let
us mention also the papers, which have appeared after the submission of the
present article, {\cite{OhTriYu}} where the stochastic quantization of the
exponential model is addressed in the parabolic and hyperbolic setting on the
two dimensional torus $\mathbb{T}^2$ and for $\sigma^2 < 4 \pi$, and
{\cite{OhTriYu2020}} where the parabolic stochastic quantization of the model
related to Liouville quantum gravity on a general compact Riemannian manifold
is treated in the $L^2$ regime. Our paper seems to be the only one studying
the elliptic case, or more generally the exponential model on the whole
$\mathbb{R}^2$. \ \

\

\tmtextbf{Acknowledgments.} The authors are grateful to an anonymous referee
of the paper, for the careful reading of the article and for the constructive
comments that have improved very much the presentation of the results. The
first named author is grateful to HCM of the University of Bonn, to Stefania
Ugolini (Universit{\`a} degli Studi di Milano) and to the Department of
Mathematics of Universit{\`a} degli Studi di Milano for financial support.
This research has been supported by the German Research Foundation (DFG) via
CRC 1060.

\section{Dimensional reduction with regularized noise}\label{section_discrete}

\subsection{Discrete spectrum}\label{section_discrete1}

Consider the following elliptic SPDE
\begin{equation}\label{eq:main1}
  (- \Delta_x + m^2 + \mathfrak{L}) (\phi) (x, z) + \mathcal{A}^2 [f g
 \partial V (\phi)] (x, z) = \xi^{\mathcal{A}} (x, z) 
\end{equation}

where $x \in \mathbb{R}^2$ and $z \in M$ (where $M$ is a $d$ dimensional
manifold with or without boundary, equipped with a Radon positive measure
$\mathd z$) and $V$ is a smooth function on $\mathbb{R}$ growing at most
exponentially at infinity, $\partial V$ its gradient, $\phi : \mathbb{R}^2
\times M \rightarrow \mathbb{R}$ a scalar random field, and the function $f g
 \partial V (\phi)$ is defined as follows
\begin{equation}
\begin{array}{cccl}
f g \partial V(\phi):& \mathbb{R}^2\times M & \longrightarrow & \mathbb{R}\\
& (x, z) &\longmapsto &f (x) g (z) \partial V
(\phi (x, z)) 
\end{array}\label{eq:def}\end{equation}

We now list the hypotheses and assumptions on the various elements of
equation~{\eqref{eq:main1}} and of its generalization to the vector case,
where $V$ is replaced by $V_n$ being defined on $\mathbb{R}^n$ (when
$\mathbb{R}^n =\mathbb{R}$ hereafter instead of writing $V_1$ we simply write
$V$).

{\descriptionparagraphs{
  \item[Hypothesis~C.] The potential $V_n : \mathbb{R}^n \rightarrow
  \mathbb{R}$ is a positive smooth function such that
  \[ y \in \mathbb{R}^n \mapsto V_n (y), \]
  is strictly convex and it and its first and second partial derivatives grow
  at most exponentially at infinity.
  
  \item[Hypothesis~QC.] The potential $V_n : \mathbb{R}^n \rightarrow
  \mathbb{R}$ is a positive smooth function, such that it and its first and
  second partial derivatives grow at most exponentially at infinity and
  additionally such that there exists a function $\mathfrak{H} : \mathbb{R}
  \rightarrow \mathbb{R}$ with exponential growth at infinity such that we
  have
  \[ - \langle \hat{n}, \partial V_n (y + r \hat{n}) \rangle \leqslant
     \mathfrak{H} (y), \qquad \text{$\hat{n} \in \mathbb{S}^{n - 1}, y \in
     \mathbb{R}^n$ and $r \in \mathbb{R}_+$}, \]
  with $\mathbb{S}^{n - 1}$ is the $n - 1$ dimensional sphere.
  
  \item[Hypothesis~H$f$.] The non-negative function $f : \mathbb{R}^2
  \rightarrow \mathbb{R}$, is invariant with respect to rotations (i.e. there
  exists $\tilde{f} : \mathbb{R}_+ \rightarrow \mathbb{R}_+$ such that $f (x)
  = \tilde{f} (| x |^2)$), it has at least $C^2$ smoothness and in addition
  satisfies $f' (x) = \tilde{f}' (| x |^2) \leqslant 0$, it decays
  exponentially at infinity and fulfills $\Delta (f) \leqslant b^2 f$ for $b^2
  \ll m^2$ (some examples of such functions are given in~{\cite{Klein1984}}).
  
  \item[Hypothesis~H$g$.] The non-negative function $g : M \rightarrow
  \mathbb{R}$ is smooth and with compact support.
  
  \item[Hypothesis~H$\mathfrak{L}$.] The operator $\mathfrak{L}$ is closed
  (possibly unbounded) and defined on the vector space $D_{\mathfrak{L}}
  \subset L^2 (M)$ (where $L^2 (M)$ is the space of measurable $L^2$ functions
  defined on $M$ with respect to the measure $\mathd z$ on $M$), which is
  dense in $L^2 (M)$. The range of $\mathfrak{L}$ is a subset of $L^2 (M)$.
  The operator $\mathfrak{L}$ is positive and self-adjoint and has a
  completely discrete spectrum. Furthermore we suppose that there exists at
  least an orthonormal basis $H_1, \ldots, H_k, \ldots$ in $L^2 (M)$ composed
  by eigenfunctions of $\mathfrak{L}$ which are $C^0_{\mathfrak{w}} (M)$
  functions (where $C^0_{\mathfrak{w}} (M)$ is the space of continuous
  functions $h$ defined on $M$ such that $\| h (z) \mathfrak{w} (z)
  \|_{\infty} < + \infty$ where $\mathfrak{w}$ is a positive continuous
  function uniformly bounded from above). Finally we denote by $\lambda_1
  \leqslant \lambda_2 \leqslant \cdots \leqslant \lambda_k \leqslant \cdots$
  the eigenvalues of $\mathfrak{L}$ corresponding to the eigenfunctions $H_1,
  H_2, \ldots, H_k, \ldots$.
  
  \item[Hypothesis~H$\mathcal{A}$.] The operator $\mathcal{A} : L^2 (M)
  \rightarrow L^2 (M)$ is a linear, injective, continuous, positive
  self-adjoint operator commuting with $\mathfrak{L}$. If we denote by
  $\sigma_1, \ldots, \sigma_k, \ldots$ the eigenvalues of $\mathcal{A}$
  corresponding to the basis $H_1, \ldots, H_k, \ldots$ we suppose that
  \[ \sum_{k = 1}^{\infty} \sigma_k < + \infty \hspace{3em} \sum_{k =
     1}^{\infty} \sigma_k \| H_k \|_{\infty, \mathfrak{w}}^2 < + \infty, \]
  where $\| \cdot \|_{\infty, \mathfrak{w}}$ denotes the $L^{\infty}$-norm in
  the weighted space $C^0_{\mathfrak{w}} (M)$.
  
  \item[Hypothesis~H$\xi$.] The noise $\xi^{\mathcal{A}}$ is Gaussian with
  mean zero and a covariance corresponding to a regularization provided by
  $\mathcal{A}$, namely such that if $h_1, h_2 \in C^{\infty}_0 (\mathbb{R}^2
  \times M)$ we have
  \[ \mathbb{E} [\langle h_1, \xi^{\mathcal{A}} \rangle \langle h_2,
     \xi^{\mathcal{A}} \rangle] = \int_{\mathbb{R}^2 \times M} (I \otimes
     \mathcal{A}) (h_1) (x, z) (I \otimes \mathcal{A}) (h_2) (x, z) \mathd x
     \mathd z. \]
}}     
     
Hereafter when $V$ is a scalar function (i.e. $V:\mathbb{R} \rightarrow \mathbb{R}$) we use the notation $V':=\partial V$.
     
\begin{remark}
  When $V_1 = V : \mathbb{R} \rightarrow \mathbb{R}$, Hypothesis~QC can be
  reformulated as
  \[ \mp V' (y \pm r) \leqslant \mathfrak{H} (y), \]
  where $r \geqslant 0$.
\end{remark}

In the rest of this Subsection~\ref{section_discrete1} we will always assume
Hypothesis~H$f$, H$g$, H$\mathfrak{L}$, H$\mathcal{A}$ and we specify the use
of Hypothesis~QC or C case by case. Moreover, hereafter we systematically use
the following abuse of notation: we denote by $\mathfrak{L}$ and $\mathcal{A}$
both the operators with the same name defined on the space $L^2 (M)$ and also
the operators $I \otimes \mathfrak{L}$ and $I \otimes \mathcal{A}$ defined on
$L^2 (\mathbb{R}^2 \times M)$.

\

We describe now the setting for equation~{\eqref{eq:main1}}. We assume that
$\xi^{\mathcal{A}}$ is defined on an abstract Wiener space $(\mathcal{W},
\mathcal{H}, \mu^{\mathcal{A}})$ where the Cameron-Martin space is
\[ \mathcal{H} = L^2 (\mathbb{R}^2) \otimes_H \mathcal{A} (L^2 (M)), \]
(here $\otimes_H$ is the natural tensor product of Hilbert spaces) equipped
with the following scalar product
\[ \langle h_1, h_2 \rangle = \int_{\mathbb{R}^2 \times M} \mathcal{A}^{- 1}
   (h_1) (x, z) \mathcal{A}^{- 1} (h_2) (x, z) \mathd x \mathd z, \qquad h_1,
   h_2 \in \mathcal{H} . \]
The Wiener space $\mathcal{W}$ is given by
\begin{eqnarray}
  \mathcal{W} & = & (- \Delta_x + 1) (C^0_{\ell} (\mathbb{R}^2) \cap W^{1 -,
  p}_{\ell} (\mathbb{R}^2)) \otimes_{\epsilon} \mathcal{A}^{1 / 2} (L^2 (M)),
  \nonumber
\end{eqnarray}
where $\ell \geqslant 0$, \ $C^0_{\ell} (\mathbb{R}^2)$ is the space of
continuous functions on $\mathbb{R}^2$ such that
\[ \| k \|_{\infty, \ell} = \sup_{x \in \mathbb{R}^2} (| k (x) | (1 + | x
   |^2)^{- \ell / 2}), \]
$W^{1 -, p}_{\ell} (\mathbb{R}^2)$ denotes the Sobolev space of regularity $1
-$ and weight
\begin{equation}
  r_{\ell} (x) = (1 + | x |^2)^{- \ell / 2}, \label{eq:rl}
\end{equation}
and $\otimes_{\epsilon}$ denotes the injective tensor product whose norm is
given by
\begin{eqnarray}
  \| w \|_{\mathcal{W}} & = & \sup \left\{ \left| \sum_{i = 1}^n \langle
  \mathfrak{b}, k_i \rangle_{\mathcal{O}_{\ell}} \langle \mathfrak{c}, h_i
  \rangle_{\mathcal{A}^{1 / 2} (L^2 (M))} \right|, \mathfrak{b} \in
  \mathcal{B}_{\mathcal{O}_{\ell}}^{\asterisk} \quad \mathfrak{c} \in
  \mathcal{B}_{\mathcal{A}^{1 / 2} (L^2 (M))}^{\asterisk} \right\} \nonumber\\
  & = & \sup \left\{ \left\| \sum_{i = 1}^n \langle \mathfrak{b}, k_i
  \rangle_{\mathcal{O}_{\ell}} h_i \right\|_{\mathcal{A}^{1 / 2} (L^2 (M))},
  \mathfrak{b} \in \mathcal{B}_{\mathcal{O}_{\ell}}^{\asterisk} \right\}
  \nonumber\\
  & = & \sup \left\{ \left\| \sum_{i = 1}^n \langle \mathfrak{c}, h_i
  \rangle_{\mathcal{A}^{1 / 2} (L^2 (M))} k_i \right\|_{\mathcal{O}_{\ell}},
  \mathfrak{c} \in \mathcal{B}_{\mathcal{A}^{1 / 2} (L^2 (M))}^{\asterisk}
  \right\} \nonumber
\end{eqnarray}
where we suppose that $w = \sum_{i = 1}^n k_i \otimes h_i \in \mathcal{W}$,
$\mathfrak{\mathcal{O}}_{\ell} = (- \Delta_x + 1) (C^0_{\ell} (\mathbb{R}^2)
\cap W^{1 -, p}_{\ell} (\mathbb{R}^2))$, $\langle \cdot, \cdot
\rangle_{\mathcal{O}_{\ell}}$ and $\langle \cdot \mathfrak{,} \cdot
\rangle_{ \mathcal{A}^{1 / 2} (L^2 (M))}$ are the natural duality of
the Banach spaces $\mathcal{O}_{\ell}$ and $\mathcal{A}^{1 / 2} (L^2 (M))$
respectively and $\mathcal{B}_{\mathcal{O}_{\ell}}^{\asterisk}$ and
$\mathcal{B}_{\mathcal{A}^{1 / 2} (L^2 (M))}^{\asterisk}$ are the unit balls
of $\mathcal{O}_{\ell}^{\asterisk}$ and $(\mathcal{A}^{1 / 2} (L^2
(M)))^{\asterisk}$ respectively (i.e., the natural dual Banach spaces of
$\mathcal{O}_{\ell}$ and $\mathcal{A}^{1 / 2} (L^2 (M))$;
see~{\cite{Ryan2002}} for more details). In particular if $w = \sum_{i =
1}^{\infty} k_i \otimes h_i$ (where the series is supposed to converges in
$\mathcal{W}$ with respect to its natural strong topology) we have
\begin{equation}
  \| w \|_{\mathcal{W}} \leqslant \sup_{\tau \in \mathcal{B}_{\mathcal{A}^{- 1
  / 2} (L^2 (M))}} \sqrt{\left( \sum_{i = 1}^{\infty} \left( \int_M \tau (z)
  h_i (z) \mathd z \right)^2 \right)} \sqrt{\sum^{\infty}_{i = 1} \| k_i
  \|^2_{\mathcal{O}_{\ell}}}, \label{eq:tensorinequality}
\end{equation}
where $\mathcal{B}_{\mathcal{A}^{- 1 / 2} (L^2 (M))}$ is the unit ball of
$\mathcal{A}^{- 1 / 2} (L^2 (M))$ (here we exploit the fact that
$\mathcal{A}^{- 1 / 2} (L^2 (M))$ is the dual space of $\mathcal{A}^{1 / 2}
(L^2 (M))$ with respect the standard scalar product of $L^2 (M)$), we use
Cauchy-Schwarz inequality and the fact that $\mathcal{A}^{- 1 / 2} (L^2 (M))$
is the dual of $\mathcal{A}^{1 / 2} (L^2 (M))$ with respect to the scalar
product of $L^2 (M)$.

\

The measure $\mu^{\mathcal{A}}$ is the centered Gaussian measure with
Cameron-Martin space $\mathcal{H}$.

In order to prove that $(\mathcal{W}, \mathcal{H}, \mu^{\mathcal{A}})$ is
actually an abstract Wiener space it is sufficient to prove that $\mathbb{E}
[\| \xi^{\mathcal{A}} \|_{\mathcal{W}}] < + \infty$ where $\| \cdot
\|_{\mathcal{W}}$ denotes the norm of $\mathcal{W}$. We have that
\begin{equation}
  \xi^{\mathcal{A}} (x, z) = \sum_{k = 1}^{\infty} \sigma_k \xi^k (x) H_k (z),
  \label{eq:serie1}
\end{equation}
where $x \in \mathbb{R}^2$, $z \in M$, $\xi^k$ are a sequence of independent
Gaussian white noises defined on $\mathbb{R}^2$ and $\sigma_k$ are defined as
in Hypothesis~H$\mathcal{A}$.

Using inequality {\eqref{eq:tensorinequality}} we have that
\begin{eqnarray}
  \| \xi^{\mathcal{A}} \|_{\mathcal{W}} & \leqslant & \sup_{\tau \in
  \mathcal{B}_{\mathcal{A}^{- 1 / 2} (L^2 (M))}} \sqrt{\left( \sum_{k =
  1}^{\infty} \sigma_k \left( \int_M \tau (z) H_k (z) \mathd z \right)^2
  \right)} \nonumber\\
  &  & \cdot \sqrt{\sum^{\infty}_{k = 1} \sigma_k \| (- \Delta_x + 1)^{- 1} (\xi^k)
  \|^2_{C^0_{\ell} (\mathbb{R}^2) \cap W^{1 -, p}_{\ell} (\mathbb{R}^2)}}
  \nonumber\\
  & \leq & \sqrt{\sum^{\infty}_{k = 1} \sigma_k \| (- \Delta_x + 1)^{- 1}
  (\xi^k) \|^2_{C^0_{\ell} (\mathbb{R}^2) \cap W^{1 -, p}_{\ell}
  (\mathbb{R}^2)}}, 
\end{eqnarray}
where in the second inequality we use the fact that $\sqrt{\left( \sum_{k =
1}^{\infty} \sigma_k \left( \int_M \tau (z) H_k (z) \mathd z \right)^2
\right)} = 1$ whenever \ $\tau \in \mathcal{B}_{\mathcal{A}^{- 1 / 2} (L^2
(M))}$.

\

On the other hand, by Hypothesis~H$\mathcal{A}$, we obtain
\[ \mathbb{E} \left[ \sum^{\infty}_{k = 1} \sigma_k \| (- \Delta_x + 1)^{- 1}
   (\xi^k) \|^2_{C^0_{\ell} (\mathbb{R}^2) \cap W^{1 -, p}_{\ell}
   (\mathbb{R}^2)} \right] \lesssim \sum_{k = 1}^{\infty} \sigma_k < + \infty,
\]
(where $\lesssim$ stands for $\leqslant$ modulo a multiplicative positive
constant), and thus $\mathbb{E} [\| \xi^{\mathcal{A}} \|_{\mathcal{W}}] < +
\infty$. This proves that $(\mathcal{W}, \mathcal{H}, \mu^{\mathcal{A}})$ is
an abstract Wiener space when $\xi^{\mathcal{A}} : \mathcal{W} \rightarrow
\mathcal{W}$, defined as $\xi^{\mathcal{A}} (w) = w$, has exactly
$\mu^{\mathcal{A}}$ as probability distribution.

For later use, it is important to note that $\mathcal{W}$ is continuously
embedded in
\[ \mathcal{W}' = (- \Delta_x + 1) (C^0_{\ell} (\mathbb{R}^2) \cap W^{1 -,
   p}_{\ell} (\mathbb{R}^2)) \otimes_{\epsilon} C^0_{\mathfrak{w}} (M), \]
since $\mathcal{A}^{1 / 2} (L^2 (M))$ is continuously embedded in
$C^0_{\mathfrak{w}} (M)$ by Hypothesis~H$\mathcal{A}$. Furthermore by
definition of injective tensor product we have $\mathcal{W}' \subset
C^0_{r_{\ell} (x) \mathfrak{w} (z)} (\mathbb{R}^2 \times M)$, with $r_{\ell}
(x)$ is defined in {\eqref{eq:rl}}, indeed
\begin{eqnarray}
  \| w \|_{C^0_{r_{\ell} (x) \mathfrak{w} (z)}} & = & \sup_{x_0 \in
  \mathbb{R}^2, z_0 \in M} | \langle r_{\ell} (x_0) \mathfrak{w} (z_0)
  \delta_{x_0} \otimes \delta_{z_0}, w \rangle | \nonumber\\
  & \leqslant & \sup \left\{ | \langle \mathfrak{b} \otimes \mathfrak{c}, w
  \rangle |, \mathfrak{b} \in \mathcal{B}_{\mathcal{O}_{\ell}}^{\asterisk}
  \quad \mathfrak{c} \in \mathcal{B}_{C^0_{\mathfrak{w}} (M)}^{\asterisk}
  \right\} = \| w \|_{\mathcal{W}'}, \nonumber
\end{eqnarray}
where $\delta_{x_0}$ and $\delta_{z_0}$ are Dirac deltas in $x_0 \in
\mathbb{R}^2$ and $z_0 \in M$ respectively (see~{\cite{Ryan2002}} Section 3.2
for the details of the proof).

Equation~{\eqref{eq:main1}} can be written as a problem defined on the
abstract Wiener space $(\mathcal{W}, \mathcal{H}, \mu^{\mathcal{A}})$. In
particular we define the map $U : \mathcal{W} \rightarrow \mathcal{H}$ by
\[ U (w) = \mathcal{A}^2 [f (x) g (z) V' (\mathcal{I} w)], \]
where $w \in \mathcal{W}$, $x \in \mathbb{R}^2$, $z \in M$, $f$ and $g$ as in
equation~{\eqref{eq:main1}}, and $\mathcal{I} : \mathcal{W} \rightarrow (-
\Delta_x + 1)^{- 1} (\mathcal{W})$ is the linear operator given by
$\mathcal{I} = (- \Delta_x + m^2 + \mathfrak{L})^{- 1}$. We define the map $T
: \mathcal{W} \rightarrow \mathcal{W}$ as $T (w) = w + U (w)$.

\begin{definition}
  \label{definition1}It is clear that if S is any measurable map $S :
  \mathcal{W} \rightarrow \mathcal{W}$ satisfying $T \circ S =
  \tmop{id}_{\mathcal{W}}$ $\mu^{\mathcal{A}}$-almost surely we have that
  $\phi (w, x, z) = \mathcal{I} (S (w)) (x, z)$ is a (strong) solution to
  equation~{\eqref{eq:main1}}. Furthermore if $\nu$ is a probability law on
  $\mathcal{W}$ such that $T_{\asterisk} (\nu) = \mu^{\mathcal{A}}$ then a
  random variable $\phi$ taking values on $ \mathcal{I}
  (\mathcal{W})$ and having the same law as $\invbreve{\nu} =
  \mathcal{I}_{\asterisk} (\nu)$ is a (weak) solution to
  equation~{\eqref{eq:main1}}. For this reason in the following we call a map
  $S : \mathcal{W} \rightarrow \mathcal{W}$ such that $T \circ S =
  \tmop{id}_{\mathcal{W}}$ $\mu^{\mathcal{A}}$-almost surely {\tmem{a strong
  solution to equation~{\eqref{eq:main1}}}} and we call a measure $\nu$ on
  $\mathcal{W}$, such that $T_{\asterisk} (\nu) = \mu^{\mathcal{A}}$, a
  \tmtextit{{\tmem{weak solution to equation~{\eqref{eq:main1}}}}}.
\end{definition}

\begin{remark}
  Let $\invbreve{T} : (- \Delta + 1)^{- 1} (\mathcal{W}) \rightarrow (- \Delta + 1)^{-
  1} (\mathcal{W})$ be the map defined as $\invbreve{T} (\invbreve{w}) = \invbreve{w} +
  \invbreve{U} (\invbreve{w})$ where $\invbreve{w} \in (- \Delta + 1)^{- 1}
  (\mathcal{W})$
  \[ \invbreve{U} (\invbreve{w}) = \mathcal{I} (\mathcal{A}^2 [f (x) g (z) V'
     (\invbreve{w})]) . \]
  Then we have that $\phi (x, z, w)$ is a weak solution to equation if and
  only if $\invbreve{T}_{\ast} (\nu_{\phi}) = \mathcal{I} (\mu^A)$ (where is the law of
  $\phi$). Furthermore we have that if $S$ is a strong solution to equation
  {\eqref{eq:main1}} in the sense of Definition \ref{definition1} if and only
  if
  \[ \invbreve{T} (\mathcal{I} (S ((- \Delta + m^2 + \mathfrak{L}) \invbreve{w}))) = \invbreve{w}
  \]
  for $\mathcal{I}_{*} (\mu^A)$-almost every $\invbreve{w} \in (- \Delta +
  1)^{- 1} (\mathcal{W})$ (we implicitly use the fact that $(- \Delta + m^2 + \mathfrak{L})
  (\invbreve{w}) \in \mathcal{W}$ for $\mathcal{I} (\mu^A)$-almost every $\invbreve{w} \in (-
  \Delta + 1)^{- 1} (\mathcal{W})$).
\end{remark}

In order to solve equation~{\eqref{eq:main1}} we need to introduce an
approximation. Let $P_n$ be the orthogonal projection in $L^2 (M)$ onto the
finite dimensional subspace generated by $H_1, \ldots, H_n$. The restrictions
$\nobracket P_n |_{\mathcal{A} [L^2 (M)]}$ and $\nobracket P_n
|_{\mathcal{A}^{1 / 2} [L^2 (M)]}$ of $P_n$ on $\mathcal{A} [L^2 (M)]$ and
$\mathcal{A}^{1 / 2} [L^2 (M)]$ are the orthogonal projections on the subspace
generated by $H_1, \ldots, H_n$ too. This implies also that $I \otimes P_n$
(in the following denoted also simply by $P_n$) is a continuous linear
operator on $\mathcal{W}$.

Let $\phi_n$ be the solution to the following approximated equation
\begin{equation}
  (- \Delta_x + \mathfrak{L} + m^2) (\phi_n) (x, z) + P_n \mathcal{A}^2 [f g V' (P_n (\phi_n))] (x, z) =   \xi^{\mathcal{A}} (x, z) , \label{eq:reduced1}
\end{equation}
where $f g V' (P_n (\phi_n))$ is defined by equation \eqref{eq:def}, and let $U_n (w) = P_n (U (P_n (w)))$ and $T_n (w) = w + U_n (w)$ be the maps,
analogous to $U$ and $T$, related to equation~{\eqref{eq:reduced1}}. Using
these objects we can define weak and strong solutions to
equation~{\eqref{eq:reduced1}} as in Definition~\ref{definition1}.

\

Let us now study equation~{\eqref{eq:reduced1}}. We introduce the following
function $V_n : \mathbb{R}^n \rightarrow \mathbb{R}$
\[ V_n (y) = \int_M g (z) V \left( \sum_{k = 1}^n y^k H_k (z) \right) \mathd
   z. \]
Then, since $P_n$ commutes with $\mathfrak{L}$ and $\mathcal{A}$, being the
projection on a subset of common eigenfunctions of both $\mathfrak{L}$ and
$\mathcal{A}$, we have that $\phi_n$ is of the form
\[ \phi_n (x, z) = \mathcal{I} \xi^{\mathcal{A}} (x, z) +
   \bar{\phi}_n (x, z) = \mathcal{I} \xi^{\mathcal{A}} (x, z) + \sum_{k = 1}^n
   \bar{\psi}_n^k (x) H_k (z), \]
where $\bar{\psi}^k_n (x)$ solves the set of equations
\begin{equation}
  \begin{array}{lll}
    0 & = & (- \Delta_x + m^2 + \lambda_k) (\bar{\psi}^k_n) (x) +\\
    &  & \quad + \sigma_k^2 f (x) \int_M g (z) \left. V' \middle( \sum_{j =
    1}^n \sigma_j \mathcal{I}_{\lambda_j} \xi^j (x) H_j (z) + \left. \sum_{j =
    1}^n \bar{\psi}^j_n (x)  H_j (z) \right) H_k (z) \mathd z
    \right.\\
    & = & (- \Delta_x + m^2 + \lambda_k) (\bar{\psi}^k_n) (x) + f (x)
    \sigma_k^2 [\partial_{y^k} (V_n) (\tilde{\xi}_n (x) + \bar{\psi}_n)],
  \end{array} \label{eq:reduced2}
\end{equation}
with \ $\mathcal{I}_{\lambda_j} = (- \Delta_x + \lambda_j + m^2)^{- 1}$,
$\tilde{\xi}_n (x) = (\sigma_k \mathcal{I}_{\lambda_k} (\xi^k))_{k = 1,
\ldots, n} \in \mathbb{R}^n$, $\partial_{y^k}$ is the partial derivatives with
respect the $k$-th variables $y^k = \bar{\psi}^k_n$, and we used the fact that
$\mathcal{A}$ is self-adjoint in $L^2 (M)$. If we denote by $\mathcal{A}_n$
the $n \times n$ diagonal matrix such that $\mathcal{A}^{_{i j}}_n = \sigma_j
\delta^{i j}$, $i, j = 1, \ldots, n$, and by $\mathfrak{L}_n$ the $n \times n$
matrix such that $\mathfrak{L}_n^{i j} = \lambda_i \delta^{i j}$ we can write
equation~{\eqref{eq:reduced2}} in the following way
\begin{equation}
  (- \Delta_x + m^2 + \mathfrak{L}_n) (\psi_n) + f \mathcal{A}^2_n \cdot
  \partial V_n (\psi_n) = \mathcal{A}_n \cdot \xi_n, \label{eq:reduced3}
\end{equation}
where $\nobracket \psi_n = (\bar{\psi}_n^k + \tilde{\xi}_n^k) |_{k = 1,
\ldots, n} \nobracket$ and $\nobracket \xi_n = (\xi^k) |_{k = 1, \ldots, n}
\nobracket$. Equation~{\eqref{eq:reduced3}} can be reformulated defined on the
abstract Wiener space $(\widehat{\mathcal{W}}_n, \widehat{H}_n, \hat{\mu}_n)$,
where $\widehat{\mathcal{H}}_n$ is the Cameron-Martin space
\[ \widehat{\mathcal{H}}_n \assign L^2 (\mathbb{R}^2 ; \mathbb{R}^n), \]
with the scalar product and norm given by $\langle h, g \rangle = \sum_{i =
1}^n \frac{1}{\sigma_i^2} \int_{\mathbb{R}^2} h^i (x) g^i (x) \mathd x$; the
Banach space $\widehat{\mathcal{W}}_n$ (in which $\widehat{\mathcal{H}}_n$ is
densely embedded) is defined as
\[ \widehat{\mathcal{W}}_n \assign W^{- 1 -, p}_{\ell} (\mathbb{R}^2 ;
   \mathbb{R}^n) \cap (1 - \Delta_x) (C^0_{\ell} (\mathbb{R}^2 ;
   \mathbb{R}^n)) \]
(where $p \geqslant 1$ is large enough), and $\hat{\mu}_n$ is the law of the
noise $\hat{\xi}_n = (\sigma_1 \xi^1, \ldots, \sigma_n \xi^n)$. On
$\widehat{\mathcal{W}}_n$ we can introduce the corresponding maps $\hat{U}_n
(\hat{w}_n) (x) = f (x)  \mathcal{A}^2_n \cdot \partial V_n
(\mathcal{I}_{\mathfrak{L}_n} (\hat{w}_n) (x))$ and $\hat{T}_n (\hat{w}_n) =
\hat{w}_n + \hat{U}_n (\hat{w}_n)$ where $\mathcal{I}_{\mathfrak{L}_n}$ is the
matrix valued operator defined as $\mathcal{I}_{\mathfrak{L}_n}^{i j} =
\delta^{i j} \mathcal{I}_{\lambda_j}$ and $\hat{w}_n \in \widehat{W}_n$. Using
$\hat{T}_n$ we can define the concept of strong and weak solutions to
equation~{\eqref{eq:reduced3}}.

\begin{remark}
  It is important to note that the relationship between $\psi_n$ and the
  strong solution $\hat{S}_n$ (defined using the map $\hat{T}_n$) is given by
  $\psi_n (\hat{w}_n, x) =  \mathcal{I}_{\mathfrak{L}_n} (\hat{S}_n
  (\hat{w}_n)) (x)$, $x \in \mathbb{R}$. Furthermore if $\hat{\nu}_n$ is a
  weak solution to equation~{\eqref{eq:reduced3}} (i.e. such that $\hat{T}_{n,
  \asterisk} (\hat{\nu}_n) = \hat{\mu}_n$) we have that $\hat{\nu}_n$ is the
  law of $(- \Delta_x + m^2 + \mathfrak{L}_n) (\psi_n)$.
\end{remark}

In the following we shall denote by $\hat{\kappa}^n$ the probability law on
$\mathbb{R}^n$ such that
\[ \frac{\mathd \hat{\kappa}^n}{\mathd y} = \frac{\exp \left( - 4 \pi \left(
   \frac{1}{2} | (m^2 +\mathfrak{L}_n)^{1 / 2} (\mathcal{A}_n^{- 1} \cdot
   y) |^2 + V_n (y) \right) \right)}{Z_{\hat{\kappa}^n}}, \]
where $Z_{\hat{\kappa}^n}$ is a suitable constant and $\mathd y$ is the
Lebesgue measure on $\mathbb{R}^n$. Finally we define $\tilde{\Upsilon}_{f,
n}$ as the random variable defined on $\widehat{\mathcal{W}}_n$ such that
\[ \tilde{\Upsilon}_{f, n} (\hat{w}_n) : = \exp \left( 4 \int_{\mathbb{R}^2}
   f' (x) V_n (\mathcal{I}_{\mathfrak{L}_n} (\hat{w}_n) (x)) \mathd x \right)
\]
where $\hat{w}_n \in \widehat{W}_n$. In the following we introduce the
function
\[ V^A_n (y) = V_n (\sigma_1 y^1, \ldots, \sigma_n y^n) = \int_M g (z) V
   \left( A \left( \sum_{k = 1}^n y^k H_k (z) \right) \right) \mathd z. \]
\begin{theorem}
  \label{theorem:reduced1}If $V^{\mathcal{A}}_n$ satisfies Hypothesis~QC then
  there exists at least a weak solution $\hat{\nu}^n$ to
  equation~{\eqref{eq:reduced2}} such that for any bounded measurable function
  $F : \mathbb{R}^n \rightarrow \mathbb{R}$ we have
  \begin{equation}
    \int_{\widehat{\mathcal{W}}_n} F ( \mathcal{I}_{\mathfrak{L}_n}
    \hat{w}_n (0)) \tilde{\Upsilon}_{f, n} (w) \mathd \hat{\nu}_n (w) = Z_{f,
    n} \int_{\mathbb{R}^n} F (y) \mathd \hat{\kappa}^n \label{eq:reduced4}
  \end{equation}
  where
  \begin{equation}
    Z_{f, n} : = \int_{\widehat{\mathcal{W}}_n} \tilde{\Upsilon}_{f, n}
    (\hat{w}_n) \mathd \hat{\nu}_n (\hat{w}_n) . \label{eq:Zf}
  \end{equation}
\end{theorem}

\begin{proof}
  The proof is given in~{\cite{Albeverio2018elliptic}} Theorem 1 for the case
  $\mathfrak{L}_n = 0$ and ${\mathcal{A}}_n = I_n$. The case of generic ${\mathcal{A}}_n$ can be
  obtained with a change of variables as it is described in Remark 5
  of~{\cite{Albeverio2018elliptic}}. The case with $\mathfrak{L}_n$ a generic
  positive diagonal matrix is a trivial extension.
\end{proof}

It is simple to prove that if the potential $V$ satisfies Hypothesis~QC the
potential $V_n^{\mathcal{A}}$ satisfies Hypothesis~QC too, indeed the
following proposition holds.

\begin{proposition}
  If $V : \mathbb{R} \rightarrow \mathbb{R}$ satisfies Hypothesis~QC then
  $V_n^{\mathcal{A}}$ satisfies Hypothesis~QC. Finally if $V$ is strictly
  convex (i.e. it satisfies Hypothesis~C) also $V_n^{\mathcal{A}}$ and $V_n$
  are convex.
\end{proposition}

\begin{proof}
  First of all we note that
  \[ \partial_{y^k} V^{\mathcal{A}}_n (y) = \int_M g (z) V' \left( 
     \mathcal{A} \left( \sum_{j = 1}^n y^j H_j (z) \right) \right) \mathcal{A}
     (H_k) (z) \mathd z. \]
  This means that if $\tmmathbf{n} \in \mathbb{S}^{n - 1}$ we have
  \begin{eqnarray}
    -\tmmathbf{n} \cdot \partial V_n^{\mathcal{A}} (y + t\tmmathbf{n}) & = & -
    \sum_{k = 1}^n n^k \partial_{y^k} (V_n^{\mathcal{A}}) (y + t\tmmathbf{n})
    \nonumber\\
    & = & - \sum_{k = 1}^n \int_M n^k g (z) V' \left( \mathcal{A} \sum_{j =
    1}^n y^j H_j (z) + \mathcal{A} \sum_{j = 1}^n n^j H_j (z) \right)
    \mathcal{A} (H_k) (z) \mathd z \nonumber\\
    & \lesssim & \int_M g (z) \mathfrak{H} \left( \mathcal{A} \sum_{j = 1}^n
    y^j H_j (z) \right) \mathd z \nonumber\\
    & \lesssim & \mathfrak{H} \left( | y | \sum_{k = 1}^{+ \infty} \sigma_k
    \left\| H_k \mathbbm{1}_{g (z) \not{=} 0} \right\|_{\infty} \right)
    \nonumber
  \end{eqnarray}
  where we use that $g$ has compact support, the fact that $V$ satisfies
  Hypothesis~QC and the fact that $\mathfrak{H}$ is increasing. Since
  \[ \sum_{k = 1}^{+ \infty} \sigma_k \left\| H_k \mathbbm{1}_{g (z) \not{=}
     0} \right\|_{\infty} \leqslant \left\| \frac{1}{\mathfrak{w}} \mathbbm{1}_{g (z)
     \not{=} 0} \right\|_{\infty} \sum_{k = 1}^{+ \infty} \sigma_k \| H_k
     \|_{\infty, \mathfrak{w}} < + \infty, \]
  where $\mathfrak{w} : M \rightarrow \mathbb{R}_+$ is the weight function in
  Hypothesis~H$\mathcal{A}$, and thus by Hypothesis~H$\mathcal{A}$, the thesis
  is proved.
\end{proof}

Using the weak solution $\hat{\nu}_n$ to equation~{\eqref{eq:reduced3}}, given
by Theorem~\ref{theorem:reduced1}, we are able to construct a weak solution to
equation~{\eqref{eq:reduced1}} satisfying the dimensional reduction principle.

A weak solution $\nu_n$ to equation~{\eqref{eq:reduced1}} is of the form
$\nu_n = \tilde{\nu}_n \otimes \mu_n^{\mathcal{A}}$, where $\tilde{\nu}_n$ is
the law of
\[ (- \Delta_x + m^2 + \mathfrak{L}) P_n (\phi_n) (x, z) = \sum_{k = 1}^n (-
   \Delta_x + m^2 + \lambda_k) (\psi^k_n) (x) H_k (z), \]
on the subspace $\tmop{Im} (P_n) \subset \mathcal{W}$ and
$\mu_n^{\mathcal{A}}$ is the law of
\[ Q_n (\phi_n) = (- \Delta_x + m^2 + \mathfrak{L}) (I - P_n) (\phi_n) (x, z)
   = \sum_{k = n + 1}^{+ \infty}  \sigma_k \xi^k (x) H_k (z), \]
which is the law of the Gaussian field on $\tmop{Im} (Q_n) \subset
\mathcal{W}$. Using the basis $H_1, \ldots, H_n$ we can identify the law
$\tilde{\nu}_n$ on $\tmop{Im} (P_n)$ with a probability measure $\hat{\nu}_n$
on $\widehat{\mathcal{W}}_n$. In this way it is evident that if the law
$\hat{\nu}_n$ satisfies the dimensional reduction principle on
$\widehat{\mathcal{W}}_n$ then the probability law $\nu_n = \hat{\nu}_n
\otimes \mu_n^{\mathcal{A}}$ satisfies the dimensional reduction principle on
$\mathcal{W}$, since it is the tensor product of two probability laws
satisfying the dimensional reduction principle.

More precisely, we consider the following abstract Wiener space $(\mathbb{W},
\mathbb{H}, \mu^{\mathcal{A}, \mathfrak{L}})$ where
\begin{eqnarray}
  \mathbb{W} & = & \mathcal{A}^{1 / 2} (L^2 (M)), \nonumber\\
  \mathbb{H} & = & (\mathcal{\mathfrak{L}} + m^2)^{- 1 / 2} (\mathcal{A} (L^2
  (M))), \nonumber
\end{eqnarray}
(with $\mathcal{A}$ and $\mathfrak{L}$ as in~{\eqref{eq:main1}}) and
$\mathbb{W}$ is equipped with its natural norm while on $\mathbb{H}$ we define
the following scalar product
\begin{equation}
  \langle h_1, h_2 \rangle_{\mathbb{H}} = 4 \pi \int_M (\mathfrak{L} + m^2)^{1
  / 2} (\mathcal{A}^{- 1} (h_1)) (z) (\mathfrak{L} + m^2)^{1 / 2}
  (\mathcal{A}^{- 1} (h_2)) (z) \mathd z. \label{eq:covariance1}
\end{equation}
Thus $\mu^{\mathcal{A}, \mathfrak{L}}$ is the centered Gaussian measure on
$\mathbb{W}$ with Cameron-Martin space given by $\mathbb{H}$. We now denote by
$\kappa_n$ the measure on $\mathbb{W}$ such that
\begin{equation}
  \frac{\mathd \kappa_n}{\mathd \mu^{\mathcal{A} \mathfrak{, L}}} (\omega) =
  \left. \exp \left( - 4 \pi \int_M g (z) V (P_n ( \nobracket
  \omega) \nobracket (z)) \mathd z \right) \middle/ Z_{\kappa_n} \right.
  \label{eq:kappa1}
\end{equation}
where $\omega \in \mathbb{W}$ and $Z_{\kappa_n}$ is a renormalization
constant.

\

We introduce also the following random variables
\begin{equation}
  \Upsilon_{f, n} (w) : = \exp \left( 4 \int_{\mathbb{R}^2 \times M} f' (x) g
  (z) V (P_n (\mathcal{I} w) (x, z)) \mathd x \mathd z \right)
\end{equation}
\begin{equation}
  \Upsilon_f (w) : = \exp \left( 4 \int_{\mathbb{R}^2 \times M} f' (x) g (z) V
  ( (\mathcal{I} w) (x, z)) \mathd x \mathd z \right)
  \label{eq:Upsilon}
\end{equation}
where $w \in \mathcal{W}$. If we use the previous identification of $\tmop{Im}
(P_n)$ with $\widehat{\mathcal{W}}_n$ we have that
\[ \tilde{\Upsilon}_{f, n} (P_n (w)) = \Upsilon_{f, n} (w) . \]
Using the previous observations the next proposition (expressing the
dimensional reduction principle) trivially follows.

\begin{proposition}
  \label{proposition_reduced1}If $V$ satisfies Hypothesis~QC then there exists
  a weak solution $\nu_n$ to equation~{\eqref{eq:reduced1}} such that for any
  bounded measurable function $F : \mathbb{W} \rightarrow \mathbb{R}$ we have
  that
  \begin{equation}
    \int_{\mathcal{W}} F (\mathcal{I} w (0, \cdot)) \Upsilon_{f, n} (w) \mathd
    \nu_n (w) = Z_{f, n} \int_{\mathbb{W}} F (\omega) \mathd \kappa_n
    (\omega), \label{eq:reduced8}
  \end{equation}
  with $Z_{f, n}$ as in {\eqref{eq:Zf}}.
\end{proposition}

The rest of the section concerns the generalization of
Proposition~\ref{proposition_reduced1} to solutions to
equation~{\eqref{eq:main1}}. In particular we denote by $\kappa$ the
probability measure on $\mathbb{W}$ such that
\begin{equation}
  \frac{\mathd \kappa}{\mathd \mu^{\mathfrak{\mathcal{A}, L}}} (\omega) =
  \left. \exp \left( - 4 \pi \int_M g (z) V (\omega (z)) \mathd z \right)
  \middle/ Z_{\kappa} \right. \label{eq:kappa2},
\end{equation}
where $\omega \in \mathbb{W}$.

\begin{theorem}
  \label{theorem_main}Suppose that $V$ satisfies Hypothesis~QC then there
  exists at least a weak solution $\nu$ to equation~{\eqref{eq:main1}} such
  that for any bounded Borel measurable function $F : \mathbb{W} \rightarrow
  \mathbb{R}$ we have
  \begin{equation}
    \int_{\mathcal{W}} F ( \mathcal{I} w (0, \cdot)) \Upsilon_f (w)
    \mathd \nu (w) = Z_f \int_{\mathbb{W}} F (\omega) \mathd \kappa (\omega)
    \label{eq:main2},
  \end{equation}
  where $\kappa$ is defined as in~{\eqref{eq:kappa2}}, $\Upsilon_f$ is given
  by equation {\eqref{eq:Upsilon}}, and
  \[ Z_f = \int_{\mathcal{W}} \Upsilon_f (w) \mathd \nu (w) . \]
\end{theorem}

\begin{remark}
  Since $\kappa_n$ converges weakly to $\kappa$, the right hand side of
  equation~{\eqref{eq:reduced8}} converges to the right hand side of
  equation~{\eqref{eq:main2}}, as $n \rightarrow + \infty$. 
\end{remark}

In order to prove Theorem~\ref{theorem_main} we prove the following lemmas. In
the space of functions $C^0_{\ell} (\mathbb{R}^2, \mathbb{R}^n)$ we denote by
$\| \cdot \|^{\mathcal{A}}_{\ell, k}$ the following norm
\begin{eqnarray}
  \| h \|^{\mathcal{A}}_{\ell, k} & = & \sup_{x \in \mathbb{R}^2}
  \sqrt{\sum_{j = 1}^n \sigma_j^{2 k} (h^j (x))^2 r_{\ell} (x)^2}, \nonumber
\end{eqnarray}
where $h \in C^0 (\mathbb{R}^2, \mathbb{R}^n)$ and $k, \ell \in \mathbb{R}$
($r_{\ell}$ is the weight we introduced in {\eqref{eq:rl}}).

\begin{lemma}
  \label{lemma_reduced1}Let $\bar{\psi}_n$ be a solution to the
  equation~{\eqref{eq:reduced2}} and let $\ell, \ell' \in \mathbb{R}$, then we
  have
  \begin{equation}
    \| \bar{\psi}_n \|^{\mathcal{A}}_{\ell, - 1} \lesssim \left\| \exp (\alpha
    \Xi_{\ell', n} r_{- \ell'} (x) \mathfrak{w}^{- 1} (z)) \mathbbm{1}_{g (z)
    \not{=} 0} f (x) r_{\ell} (x) \right\|_{\infty} \label{eq:inequality1}
  \end{equation}
  \begin{equation}
    \begin{array}{l}
      \| (- \Delta_x + m^2 + \mathfrak{L}_n) (\bar{\psi}_n)
      \|^{\mathcal{A}}_{\ell, - 1}\\
      \qquad \lesssim \left\| \exp (\alpha (\alpha' (\Xi_{\ell', n}
      \mathfrak{w}^{- 1} (z) + \| \bar{\psi}_n \|^{\mathcal{A}}_{\ell', - 1})
      r_{- \ell'} (x))) \mathbbm{1}_{g (z) \not{=} 0} f (x) r_{\ell} (x)
      \right\|_{\infty},
    \end{array} \label{eq:inequality2}
  \end{equation}
  where $\Xi_{\ell', n} = \| P_n \mathcal{I} \xi^{\mathcal{A}} (x, z)
  \|_{\nobracket C^0_{\ell'} (\mathbb{R}^2 \nobracket) \otimes_{\epsilon}
  C^0_{\mathfrak{w}} (M)}$, $\alpha, \alpha'$ and the implicit constants do
  not depend on $n$.
\end{lemma}

\begin{proof}
  We write $\tilde{\psi}_n = \left( \frac{\bar{\psi}_n^k}{\sigma_k} \right)_{k
  = 1, \ldots, n}$. The equation for $\tilde{\psi}_n$ reads
  \[ (- \Delta_x + m^2 + \lambda_k) (\tilde{\psi}^k_n) (x) + \sigma_k f (x)
     \int_M g (z) H_k (z) V' (P_n \mathcal{I} \xi^A (x, z) +
     \bar{\phi}_n (x, z)) \mathd z = 0 \]
  where $\bar{\phi}_n (x, z)$ is related to $\tilde{\psi}_n$ by
  \[ \bar{\phi}_n (x, z) = \sum_{k = 1}^n \sigma_k \tilde{\psi}^k_n (x) H_k
     (z) . \]
  Putting $\tilde{\Psi}_n^2 (x) = (1 + \theta | x |^2)^{- \ell} \sum_{k = 1}^n
  (\tilde{\psi}^k_n (x))^2$, writing $r_{\ell, \theta} (x) = (1 + \theta | x
  |^2)^{- \ell}$, for some $\theta > 0$, and denoting by $\bar{x}$ the maximum
  of $\tilde{\Psi}_n$ we have
  \begin{eqnarray}
    m^2 \tilde{\Psi}_n^2 (\bar{x}) & \leqslant & - \frac{1}{2} \Delta
    (\tilde{\Psi}_n^2) (\bar{x}) + m^2 \tilde{\Psi}_n^2 (\bar{x}) + r_{\ell,
    \theta} (\bar{x}) \tilde{\psi}_n (\bar{x}) \cdot \mathfrak{L}_n
    (\tilde{\psi}_n (\bar{x})) \nonumber\\
    & \leqslant & - r_{\ell, \theta} \tilde{\psi}_n \cdot \Delta
    \tilde{\psi}_n - r_{\ell, \theta} \sum_{k = 1}^n | \nabla \tilde{\psi}_n^k
    |^2 - \left( \frac{- 2 | \nabla r_{\ell, \theta} |^2 + r_{\ell, \theta}
    \Delta r_{\ell, \theta}}{2 r_{\ell, \theta}^2} \right) \tilde{\Psi}^2_n
    (\bar{x}) \nonumber\\
    &  & \quad + m^2 \tilde{\Psi}_n^2 (\bar{x}) + r_{\ell, \theta}
    \tilde{\psi}_n (\bar{x}) \cdot \mathfrak{L}_n (\tilde{\psi}_n (\bar{x}))
    \nonumber\\
    & \leqslant & - f (\bar{x}) r_{\ell, \theta} (\bar{x}) \int_M g (z)
    \bar{\phi}_n (\bar{x}, z) V' (P_n \mathcal{I} \xi^A
    (\bar{x}, z) + \bar{\phi}_n (\bar{x}, z)) \mathd z \nonumber\\
    &  & \quad - \left( \frac{- 2 | \nabla r_{\ell, \theta} |^2 + r_{\ell,
    \theta} \Delta r_{\ell, \theta}}{2 r_{\ell, \theta}^2} \right)
    \tilde{\Psi}^2_n (\bar{x}),  \label{eq:norm1}
  \end{eqnarray}
  where we used that
  \[ \nabla \left[ \sum_{k = 1}^n (\tilde{\psi}^k_n (\bar{x}))^2 \right] = -
     \frac{\nabla r_{\ell, \theta} (\bar{x})}{r_{\ell, \theta} (\bar{x})}
     \sum_{k = 1}^n (\tilde{\psi}^k_n (\bar{x}))^2 \]
  since $\bar{x}$ is a stationary point for $\tilde{\Psi}_n$. From
  equation~{\eqref{eq:norm1}} if, for any fixed $\ell$, we choose a $\theta$
  in such a way that
  \[ \left| \frac{- 2 | \nabla r_{\ell, \theta} |^2 + r_{\ell, \theta} \Delta
     r_{\ell, \theta}}{2 r_{\ell, \theta}^2} \right| < m^2, \]
  we obtain
  \[ \tilde{\Psi}_n^2 (\bar{x}) \lesssim \left\| \exp (\alpha \Xi_{\eta', n}
     r_{- \ell'} (x) \mathfrak{w}^{- 1} (z)) \mathbbm{1}_{g (z) \not{=} 0} f
     (x) (r_{\ell, \theta} (x))^{\frac{1}{2}} \right\|_{\infty} \cdot \left\|
     (r_{\ell, \theta} (x))^{\frac{1}{2}} \bar{\phi}_n (x, z) \mathbbm{1}_{g
     (z) \not{=} 0} \right\|_{\infty}, \]
  where all the constants do not depend on $n$. On the other hand
  \begin{eqnarray*}
    (r_{\ell, \theta} (x))^{\frac{1}{2}} | \bar{\phi}_n (x, z) | & \leqslant &
    (r_{\ell, \theta} (x))^{\frac{1}{2}} \sum_{k = 1}^n | \bar{\psi}_n^k (x) |
    \cdot | H_k (z) |\\
    & \lesssim & \left\| \frac{1}{\mathfrak{w} (z)} \mathbbm{1}_{g (z) \not{=} 0}
    \right\|_{\infty} \tilde{\Psi}_n (\bar{x}) \sqrt{\sum_{k = 1}^{\infty}
    \sigma_k \| H_k (z) \mathfrak{w} (z) \|_{\infty}^2},
  \end{eqnarray*}
  where $w (z)$ is the weight function in Hypothesis~H$\mathcal{A}$, for all
  $x \in \mathbb{R}^2$ and $z \in \tmop{supp} (g)$. Using
  Hypothesis~H$\mathcal{A}$ we deduce from this that
  \[ \tilde{\Psi}_n (\bar{x}) \lesssim \left\| \exp (\alpha \Xi_{\ell', n}
     r_{- \ell'} (x) \mathfrak{w}^{- 1} (z)) \mathbbm{1}_{g (z) \not{=} 0} f (x)
     (r_{\ell, \theta} (x))^{\frac{1}{2}} \right\|_{\infty} . \]
  Since $\| \bar{\psi}_n \|^{\mathcal{A}}_{\ell, - 1} \lesssim_{\theta}
  \tilde{\Psi}_n (\bar{x})$ inequality~{\eqref{eq:inequality1}} is proved.
  Inequality~{\eqref{eq:inequality2}} follows directly from
  inequality~{\eqref{eq:inequality1}} and equation~{\eqref{eq:reduced2}}.
\end{proof}

\begin{corollary}
  \label{corollary_reduced1}Under the hypotheses and notations of
  Lemma~\ref{lemma_reduced1}, if $\ell < - 1$, we have that there exists an
  increasing continuous function $K_0 : \mathbb{R}_+ \rightarrow \mathbb{R}_+$
  such that
  \[ \| (- \Delta_x + m^2 + \mathfrak{L}) (\bar{\phi}_n) \|_{\mathcal{H}}
     \leqslant K_0 (\Xi_{\ell, n}), \]
  with $\Xi_{\ell, n}$ as in Lemma~\ref{lemma_reduced1}.
\end{corollary}

\begin{proof}
  The proof consists simply in noting that
  \[ \| (- \Delta_x + m^2 + \mathfrak{L}) (\bar{\phi}_n) \|_{\mathcal{H}}
     \lesssim \| (- \Delta_x + m^2 + \mathfrak{L}_n) (\bar{\psi}_n)
     \|^{\mathcal{A}}_{\ell, - 1}, \]
  indeed
  \begin{eqnarray*}
    \| (- \Delta_x + m^2 + \mathfrak{L}) (\bar{\phi}_n) \|_{\mathcal{H}}^2 & =
    & \int_{\mathbb{R}^2 \times M} [(- \Delta_x + m^2 + \mathfrak{L})
    (\mathcal{A}^{- 1} (\bar{\phi}_n) (x, z))]^2 \mathd x \mathd z\\
    & = & \sum_{k = 1}^n \frac{1}{\sigma_k^2} \int_{\mathbb{R}^2} ((-
    \Delta_x + m^2 + \lambda_k) (\bar{\psi}_n) (x))^2 \mathd x\\
    & \leqslant & (\| (- \Delta_x + m^2 + \mathfrak{L}_n) (\bar{\psi}_n)
    \|^{\mathcal{A}}_{\ell, - 1})^2 \int_{\mathbb{R}^2} (1 + | x |^2)^{\ell}
    \mathd x
  \end{eqnarray*}
  which is finite whenever $\ell < - 1$.
\end{proof}

\begin{remark}
  \label{remark_xi}Since we can identify $\xi^{\mathcal{A}}$ with the identity map $w
  \mapsto \xi^{\mathcal{A}} (w) = w$ on $\mathcal{W}$, we can identify $\Xi_{n, \ell}$ with a random
  variable defined as $w \mapsto \Xi_{n, \ell} (w) = \| P_n (\mathcal{I}
  \xi^{\mathcal{A}} (w)) (x, z) \|_{\nobracket C^0_{\ell'} (\mathbb{R}^2
  \nobracket) \otimes_{\epsilon} C^0_{\mathfrak{w}} (M)}$.
\end{remark}

An important consequence of Corollary~\ref{corollary_reduced1} is the
following lemma.

\begin{lemma}
  \label{lemma_reduced2}Let $K \subset \mathcal{W}$ be a compact set, then
  $\mathfrak{K} = \cup_{n \in \mathbb{N}} T^{- 1}_n (K)$ is precompact in
  $\mathcal{W}$.
\end{lemma}

\begin{proof}
  We write $\mathcal{T}_n (\hat{\phi}, w) = \hat{\phi} + U_n (\hat{\phi} + w)
  .$ If $w_{n, y} \in \mathcal{W}$ is solution to the equation $T_n  (w_{n,
  y}) = y$ (where $y \in W$) then $\hat{\phi}_{n, y} = w_{n, y} - y$ will be a
  solution to the equation $\mathcal{T}_n (\hat{\phi}_{n, y}, y) = 0$. This
  will implies that $\bar{\phi}_{n, y} = \mathcal{I} (\hat{\phi}_{n, y})$ is a
  solution to equation~{\eqref{eq:reduced2}} for the realization of the white
  noise $\xi^{\mathcal{A}} (y)$. By Corollary~\ref{corollary_reduced1} and
  Remark \ref{remark_xi}, we have that $\| \hat{\phi}_{n, y} \|_{\mathcal{H}}
  \leqslant K_0 (\Xi_{\ell, n} (y))$. On the other hand we have the estimate
  \[ \sup_{n \in \mathbb{N}} (\Xi_{\ell, n} (y)) = \sup_{n \in \mathbb{N}} \|
     P_n \mathcal{I} \xi^{\mathcal{A}} (y) \|_{\nobracket C^0_{\ell}
     (\mathbb{R}^2 \nobracket) \otimes_{\epsilon} C^0_{\mathfrak{w}} (M)}
     \lesssim \sup_{n \in \mathbb{N}} \| P_n \xi^{\mathcal{A}} (y)
     \|_{\mathcal{W}} \lesssim \| \xi^{\mathcal{A}} (y) \|_{\mathcal{W}}, \]
  which implies that
  \[ \sup_{n \in \mathbb{N}, y \in K} \| \hat{\phi}_{n, y} \|_{\mathcal{H}}
     \leqslant \sup_{n \in \mathbb{N}, y \in K} K_0 (\Xi_{\ell, n} (y))
     \lesssim \sup_{y \in K} K_0 (\| \xi^{\mathcal{A}} (y) \|_{\mathcal{W}}) =
     C_K < + \infty . \]
  Letting $\tilde{K} = \{ h \in \mathcal{H}, \| h \|_{\mathcal{H}} \leqslant
  C_K \}$ we have that $\tilde{K}$ is compact in $\mathcal{W}$ since the
  inclusion map $i : \mathcal{H} \rightarrow \mathcal{W}$ is compact. This
  fact implies that $\mathfrak{K} \subset K + \tilde{K}$ is precompact since
  it is contained in the sum of two compact sets.
\end{proof}

\begin{lemma}
  \label{lemma_reduced3}The family $(\nu_n)_n$ of measures such that $T_{n,
  \asterisk} (\nu_n) = \mu^{\mathcal{A}}$ is tight.
\end{lemma}

\begin{proof}
  Let $\tilde{K}$ be a compact set such that $\mu^{\mathcal{A}} (\tilde{K})
  \geqslant 1 - \epsilon$ for a fixed $0 < \epsilon < 1$, then, by
  Lemma~\ref{lemma_reduced2}, $\mathfrak{K} \assign \overline{\cup_{n \in
  \mathbb{N}} T^{- 1}_n (\tilde{K})}$ is a compact set in $\mathcal{W}$. This
  implies
  \[ \nu_n (\mathfrak{K}) \geqslant \nu_n (\cup_i T^{- 1}_i (\tilde{K}))
     \geqslant \nu_n (T^{- 1}_n (\tilde{K})) \geqslant \mu^{\mathcal{A}}
     (\tilde{K}) \geqslant 1 - \epsilon, \]
  for any $n \in \mathbb{N}$. Thus the sequence $\nu_n$ is tight and the lemma
  is proven. 
\end{proof}

\begin{lemma}
  \label{lemma_reduced4}Suppose that $(\nu_n)_n$ is a sequence of weak
  solutions to equation~{\eqref{eq:reduced1}} weakly converging to $\nu$, then
  $\nu$ is a solution to~{\eqref{eq:main1}}, i.e. $T_{\asterisk} (\nu) =
  \mu^{\mathcal{A}}$.
\end{lemma}

\begin{proof}
  Proving the claim is equivalent to prove that for any function $h :
  \mathcal{W} \rightarrow \mathbb{R}$ which is bounded with continuous and
  bounded Fr{\'e}chet derivative we have $\int h \circ T \mathd \nu = \int h
  \mathd \mu^{\mathcal{A}}$. In order to prove this we note that
  \[ \| h \circ T (w) - h \circ T_n (w) \|_{\mathcal{W}} \leqslant \| h
     \|_{C^1 (\mathcal{W})} \cdot \| U (w) - P_n U (P_n (w)) \|_{\mathcal{W}}
     . \]
  On the other hand we have
  \[ \| U (w) - P_n U (P_n (w)) \|_{\mathcal{W}} \leqslant \sqrt{\sum_{k =
     n}^{\infty} \sigma_k} \cdot \| U (w) \|_{\mathcal{H}} + \sup_{y \in [w,
     P_n (w)]} \| \nabla U (y) \|_{\mathcal{L} (\mathcal{W})} \| P_n (w) - w
     \|_{\mathcal{W}} . \]
  Finally we observe that
  \[ \| P_n (w) - w \|_{\mathcal{W}} = \| (I - P_n) (T_n (w)) \|_{\mathcal{W}}
     . \]
  Let, for any fixed $\varepsilon > 0$, $K_{\epsilon} \subset \mathcal{W}$ be
  a compact set such that $\nu_n (\mathcal{W} \backslash K_{\epsilon})$, $\nu
  (\mathcal{W} \backslash K_{\epsilon}) < \epsilon$, then we have
  \begin{eqnarray}
    \left| \int_{\mathcal{W}} h \circ T \mathd \nu - \int_{\mathcal{W}} h
    \mathd \mu^A \right| & \leqslant & \left| \int_{\mathcal{W}} h \circ T
    \mathd \nu - \int_{\mathcal{W}} h \circ T \mathd \nu_n \right| + \left|
    \int_{\mathcal{W}} (h \circ T - h \circ T_n) \mathd \nu_n \right|
    \nonumber\\
    &  & \quad + \left| \int_{\mathcal{W}} h \circ T_n \mathd \nu_n -
    \int_{\mathcal{W}} h \mathd \mu^A \right| \nonumber\\
    & \leqslant & \left| \int_{\mathcal{W}} h \circ T \mathd \nu -
    \int_{\mathcal{W}} h \circ T \mathd \nu_n \right| + \epsilon \| h \|_{C^0
    (\mathcal{W})} \nonumber\\
    &  & \quad + \left| \int_{K_{\epsilon}} (h \circ T - h \circ T_n) \mathd
    \nu_n \right| .  \label{eq:reduced5}
  \end{eqnarray}
  But
  \begin{eqnarray}
    \left| \int_{K_{\epsilon}} (h \circ T - h \circ T_n) \mathd \nu_n \right|
    & \lesssim_h & \sqrt{\sum_{k = n}^{\infty} \sigma_k} \cdot \sup_{w \in
    K_{\epsilon}} \| U (w) \|_{\mathcal{H}} + \sup_{w \in K_{\epsilon}} \|
    \nabla U (w) \|_{\mathcal{L} (\mathcal{W})} \nonumber\\
    &  & \quad \times \int_{\mathcal{W}} \| (I - P_n) (T_n (w))
    \|_{\mathcal{W}} \mathd \nu_n .  \label{eq:reduced6}
  \end{eqnarray}
  On the other hand
  \begin{equation}
    \int_{\mathcal{W}} \| (I - P_n) (T_n (w)) \|_{\mathcal{W}} \mathd \nu_n =
    \int_{\mathcal{W}} \| (I - P_n) (w) \|_{\mathcal{W}} \mathd
    \mu^{\mathcal{A}} \lesssim \sum_{k = n}^{\infty} \sigma_k
    \label{eq:reduced7},
  \end{equation}
  where we used that $\mathbb{E} [\| (- \Delta_x + 1)^{- 1} (\xi^k)
  \|^2_{C^0_{\ell} (\mathbb{R}^2) \cap W^{1 -, p}_{\ell} (\mathbb{R}^2)}]$ is
  equal to a finite constant independent on $k$. Since $\sum \sigma_k$ is
  convergent, we have that the right hand side of~{\eqref{eq:reduced7}}
  converges to 0 as $n \rightarrow + \infty$. Since $U (w), \nabla U (w)$ are
  continuous with respect to $w$, this implies that the right hand side of
  equation~{\eqref{eq:reduced6}}, converges to 0 as $n \rightarrow + \infty$.
  Exploiting this fact, the fact that $\nu_n$ weakly converges to $\nu$ and
  equation~{\eqref{eq:reduced5}} we obtain that $\left| \int_{\mathcal{W}} h
  \circ T \mathd \nu - \int_{\mathcal{W}} h \mathd \mu^{\mathcal{A}} \right|
  \leqslant \epsilon \| h \|_{C^0 (\mathcal{W})}$. Since $\epsilon > 0$ is
  arbitrary the lemma is proved.
\end{proof}

\begin{proof*}{Proof of Theorem~\ref{theorem_main}}
  Thanks to Lemma~\ref{lemma_reduced4} we have that the sequence of weak
  solutions $\nu_n$, satisfying the dimensional reduction principle (whose
  existence is proven by Proposition~\ref{proposition_reduced1}) converges
  weakly, as $n \rightarrow + \infty$, to a solution $\nu$ of the
  equation~{\eqref{eq:main1}}. If we are able to prove that $\Upsilon_{f, n}
   \mathd\nu_n \rightarrow \Upsilon_f  \mathd\nu$ weakly, then the theorem would follow.
  
  First of all we note that $\Upsilon_{f, n} (w) = \Upsilon_f (P_n (w))$ and
  that $\Upsilon_f$ is a bounded Fr{\'e}chet $C^1 (\mathcal{W})$ function. We
  want to prove that $\nabla \Upsilon_f (w)$ is bounded when $w$ is in a
  bounded subset of $\mathcal{W}$. We have that
  \[ \nabla \Upsilon_f (w) [h] = - \Upsilon_f (w) \cdot \int_{\mathbb{R}^2
     \times M} f' (x) g (z) V' (\mathcal{I} (w) (x, z)) (\mathcal{I} h) (x, z)
     \mathd x \mathd z. \]
  In particular we have that
  \[ \| \nabla \Upsilon_f (w) \|_{\mathcal{W}^{\asterisk}} \lesssim
     \int_{\mathbb{R}^2 \times M} | f' (x) g (z) | \exp (\alpha \| \mathcal{I}
     (w) \|_{C^0_{\ell} (\mathbb{R}^2) \otimes_{\epsilon} C^0_{\mathfrak{w}}
     (M)} r_{- \ell} (x) \mathfrak{w}^{- 1} (z)) \mathd x \mathd z, \]
  for a suitable positive constant $\alpha$. On the other hand the linear map
  $\mathcal{I} : \mathcal{W} \rightarrow C^0_{\ell} (\mathbb{R}^2)
  \otimes_{\epsilon} C^0_{\mathfrak{w}} (M)$ is continuous, which implies that
  if $B$ is a bounded subset of $\mathcal{W}$ given by $\sup_{w \in B} \|
  \mathcal{I} (w) \|_{C^0_{\ell} (\mathbb{R}^2) \otimes_{\epsilon}
  C^0_{\mathfrak{w}} (M)} < + \infty$. This observation proves that $\sup_{w
  \in B} \| \nabla \Upsilon_f (w) \|_{\mathcal{W}^{\asterisk}} < + \infty$ for
  any bounded set $B \subset \mathcal{W}$.
  
  Let for any given $\epsilon > 0$, $K_{\epsilon} \subset \mathcal{W}$ be a
  compact set such that $\nu_n (\mathcal{W} \backslash K_{\epsilon}), \nu
  (\mathcal{W} \backslash K_{\epsilon}) < \epsilon$, then there exists a ball
  $B_{\epsilon} \subset \mathcal{W}$ such that $K_{\epsilon} \cup (\cup_{n \in
  \mathbb{N}} P_n (K_{\epsilon})) \subset B_{\epsilon}$. Let $F$ be a
  continuous and bounded function on $\mathcal{W}$, then
  \begin{eqnarray}
    \left| \int_{\mathcal{W}} F (w) \Upsilon_{f, n} \mathd \nu_n -
    \int_{\mathcal{W}} F (w) \Upsilon_f \mathd \nu \right| & \leqslant &
    \left| \int_{\mathcal{W}} F (w) (\Upsilon_f (P_n (w)) - \Upsilon_f (w))
    \mathd \nu_n \right| \nonumber\\
    &  & \quad + \left| \int_{\mathcal{W}} F (w) \Upsilon_f \mathd \nu_n -
    \int_{\mathcal{W}} F (w) \Upsilon_f \mathd \nu \right| \nonumber\\
    & \lesssim & \epsilon + \| \nabla \Upsilon_f (w) \|_{C^1 (B_{\epsilon})}
    \int_{\mathcal{W}} \| P_n (w) - w \|_{\mathcal{W}} \mathd \nu_n
    \nonumber\\
    &  & \quad + \left| \int_{\mathcal{W}} F (w) \Upsilon_f \mathd \nu_n -
    \int_{\mathcal{W}} F (w) \Upsilon_f \mathd \nu \right| \nonumber\\
    & \lesssim & \epsilon + \| \nabla \Upsilon_f (w) \|_{C^1 (B_{\epsilon})}
    \int_{\mathcal{W}} \| P_n (T_n (w)) - T_n (w) \|_{\mathcal{W}} \mathd
    \nu_n \nonumber\\
    &  & \quad + \left| \int_{\mathcal{W}} F (w) \Upsilon_f \mathd \nu_n -
    \int_{\mathcal{W}} F (w) \Upsilon_f \mathd \nu \right| \rightarrow
    \epsilon, \nonumber
  \end{eqnarray}
  as $n \rightarrow \infty$, where the constants implied in the symbol
  $\lesssim$ do not depend on $\epsilon$. For this reason since $\epsilon$ is
  arbitrary positive, $\int_{\mathcal{W}} F (w) \Upsilon_{f, n} \mathd \nu_n
  \rightarrow \int_{\mathcal{W}} F (w) \Upsilon_f \mathd \nu$ and the theorem
  is proven.
\end{proof*}

\subsection{Dimensional reduction in the full
space}\label{subsection_extension1}

In this section we consider the following equation
\begin{equation}
  (- \Delta_x - \Delta_z + m^2) (\phi) (x, z) + \mathcal{A}^2 [f g
  V'(\phi) ] (x, z) = \xi^{\mathcal{A}} (x, z) \label{eq:extension1},
\end{equation}
where $x \in \mathbb{R}^2$ and $z \in M$ (see equation \eqref{eq:def} for the notation $fgV'(\phi):=fg\partial V(\phi)$). We consider the following
hypotheses on $\mathfrak{L}$ and $\mathcal{A}$ (in addition to the previous
H$f$ and H$g$ on the cut-offs):

{\descriptionparagraphs{
  \item[Hypothesis~H$\mathfrak{L}$1] $\mathfrak{L} = - \Delta_z$ and $M
  =\mathbb{R}^d$ (for any fixed integer $d \geqslant 0$).
  
  \item[Hypothesis~H$\mathcal{A}$1] $\mathcal{A} (h (z')) = \mathfrak{a}
  \asterisk h$ where $\mathfrak{a} (z)$ is a $\mathcal{C}^{1 + d / 2 +}
  (\mathbb{R}^2)$ H{\"o}lder continuous function with compact support.
  
  \item[Hypothesis~H$\xi 1$] The noise $\xi^{\mathcal{A}}$ is such that if
  $h_1, h_2 \in C^{\infty}_0 (\mathbb{R}^2 \times \mathbb{R}^d)$ we have
  \[ \mathbb{E} [\langle h_1, \xi^{\mathcal{A}} \rangle \langle h_2,
     \xi^{\mathcal{A}} \rangle] = \int_{\mathbb{R}^2 \times \mathbb{R}^d} (I
     \otimes \mathcal{A}) (h_1) (x, z) (I \otimes \mathcal{A}) (h_2) (x, z)
     \mathd x \mathd z. \]
  }}
  
In this subsection we assume Hypotheses~H$\mathfrak{L}$1,~H$\mathcal{A}$1
and~H$\xi$1 instead of the corresponding assumptions
H$\mathfrak{L}$,~H$\mathcal{A}$ and~H$\xi$ respectively we had before. The
assumptions~H$f$ and~H$g$ on the space cut-offs however remain.

\

We define the abstract Wiener space $(\mathcal{W}, \mathcal{H},
\mu^{\mathcal{A}})$ relative to equation~{\eqref{eq:extension1}} as follows
\begin{eqnarray}
  \mathcal{H} & = & L^2 (\mathbb{R}^2) \otimes_H \mathcal{A} (L^2
  (\mathbb{R}^d)) \nonumber\\
  \mathcal{W} & = & (- \Delta_x + 1) (C^0_{\ell} (\mathbb{R}^2) \cap W^{1 -,
  p}_{\ell} (\mathbb{R}^2)) \otimes_{\epsilon} C^0_{\ell'} (\mathbb{R}^d)
  \nonumber
\end{eqnarray}
where we suppose $\ell, \ell' > 0$. The maps $T$ and $U$, and so the concept
of weak and strong solution to equation~{\eqref{eq:extension1}}, are defined
in the same way as in the previous section. We shall prove the following
theorem.

\begin{theorem}
  \label{theorem_extension1}Suppose that $V$ satisfies Hypothesis~QC (in
  Section \ref{section_discrete1}), then there exists a weak solution $\nu$ to
  equation~{\eqref{eq:extension1}} such that for any bounded Borel measurable
  function $F : \mathbb{W} \rightarrow \mathbb{R}$
  \[ \int_{\mathcal{W}} F (\mathcal{I} w (0)) \Upsilon_f (w) \mathd
     \nu (w) = Z_f \int_{\mathbb{W}} F (\omega) \mathd \kappa (\omega), \]
  as in~{\eqref{eq:main2}}.
\end{theorem}

It is clear that the conditions H$\mathfrak{L}$1, H$\mathcal{A}$1 on
$\mathfrak{L}, M, \mathcal{A}$ are incompatible with the previous
Hypotheses~H$\mathfrak{L}$ and~H$\mathcal{A}$. Thus
Theorem~\ref{theorem_extension1} does not follow directly from
Theorem~\ref{theorem_main}. The rest of this subsection will focus on the
proof of Theorem~\ref{theorem_extension1}.

\

In order to achieve the proof we first replace $M$ by the manifolds $M_R
=\mathbb{T}^d_R$, i.e. $M_R$ is a torus of radius $R > 0$. If $h$ is a
real-valued function on $M =\mathbb{R}^d$ with sufficient decay at infinity we
can easily define a function $h_R$ on $M_R$ in the following way
\[ h_R (z_R) = \sum_{j \in \mathbb{Z}^d} h (z_R + R j) \]
where $z_R \in [- R / 2, R / 2]^d$. There is also a standard way of defining a
function $\tilde{h}_R$ on $M_R$ given one on $M$, that is the following
\[ \tilde{h}_R (z_R) = h (z_R) . \]
If the function $h$ has support contained in $[- R / 2, R / 2]^d$ then
$\tilde{h}_R = h_R$.

On the other hand if we consider a function $k : M_R \rightarrow \mathbb{R}$
we can associated with it a function $k^p$ defined on all of $M$ in the
following way
\[ k^p (z) = k \left( z - R \left\lfloor \frac{z + \frac{R}{2}}{R}
   \right\rfloor \right), \]
$z \in M =\mathbb{R}^d$. In general we have that $\widetilde{(k^p)}_R = k$,
and so if $k$ has compact support contained in $(- R / 2, R / 2)^d$ we have $k
= (k^p)_R$. Furthermore if $\mathfrak{b}$ is a function with compact support
on $\mathbb{R}^d$ and $k$ is a function on $M_R$ we have $(\mathfrak{b}_R
\asterisk k)^p = \mathfrak{b} \asterisk k^p$. Furthermore it is important to
note that if $\mathfrak{G}_R$ is the Green function associated with the
operator $(- \Delta_x - \Delta_z + m^2)^{- 1}$ on $\mathbb{R}^2 \times M_R$
and $\mathcal{G}$ is the Green function of the same operator on $\mathbb{R}^2
\times M$ we have
\[ \mathcal{G}_R = \mathfrak{G}_R \]
We define the operator $\mathcal{A}_R$ on $L^2 (M_R)$ by $\mathcal{A}_R (k) =
\mathfrak{a}_R \asterisk k$. In this way we can define the abstract Wiener
space $(\mathcal{W}_R, \mathcal{H}_R, \mu^{\mathcal{A}_R}_R)$ as follows
\begin{eqnarray}
  \mathcal{H}_R & = & L^2 (\mathbb{R}^2) \otimes_H \mathcal{A}_R (L^2 (M_R))
  \nonumber\\
  \mathcal{W}_R & = & (- \Delta_x + 1) (C^0_{\ell} (\mathbb{R}^2) \cap W^{1 -,
  p}_{\ell} (\mathbb{R}^2)) \otimes_{\epsilon} \mathcal{A}_R^{1 / 2} (L^2
  (M_R)) \nonumber
\end{eqnarray}
and $\mu^{\mathcal{A}_R}_R$ is the law of the noise $\xi^{\mathcal{A}}_R$ on
$\mathcal{W}_R$ with covariance $\mathcal{A}_R$.

Using the map $(\cdot)^p$ we can extend the noise $\xi^{\mathcal{A}}_R$,
defined on $M_R$, to the noise $\xi^{\mathcal{A}, p}_R$ defined on all of $M$.
This means that the law $\mu^{\mathcal{A}_R}_R$ of $\xi^{\mathcal{A}}_R$ on
$\mathcal{W}_R$, thanks to the map $(\cdot)^p$, induces a Gaussian measure
$\mu^{\mathcal{A}_R, p}_R$ on $\mathcal{W}$, which is the underling
probability, measure of the noise $\xi^{\mathcal{A}, p}_R$. It is simple to
prove that $\mu^{\mathcal{A}_R, p}_R$ weakly converges to $\mu^{\mathcal{A}}$
when $R \rightarrow + \infty$ and the support of $\mathfrak{a}$ is compact.

Let $U_R : \mathcal{W}_R \rightarrow \mathcal{H}_R$ be the function defined
by
\[ U_R (w_R) : = f (x) \mathcal{A}_R (g_R (z_R) V' (\mathfrak{G}_R \asterisk
   w_R)) = f (x) \mathcal{A}_R (g_R (z_R) V' (\mathcal{G}_R \asterisk w_R)) .
\]
We set $T_R (w_R) = w_R + U_R (w_R)$. The map $U_R$ induces a map $U_R^p$
on $\mathcal{W}$ in the following way
\[ U_R^p (w) = f (x) \mathcal{A}_R (g_R^p (z_R) V' (\mathcal{G} \asterisk
   (\tilde{w}_R)^p)), \]
and $T^p_R (w) = w + U^p_R (w)$. Let $\nu_R$ be a probability law on
$\mathcal{W}_R$ such that $T_{R, \asterisk} (\nu_R) = \mu^{\mathcal{A}_R}_R$,
then it induces a probability law $\nu^p_R$ on $\mathcal{W}$ such that
$T^p_{R, \ast} (\nu_R^p) = \mu^{\mathcal{A}_R, p}_R$.

\begin{lemma}
  \label{lemma_extension1}Let $R_n \in \mathbb{R}_+$ be a sequence such that
  $R_n \rightarrow + \infty$, then the sequence $\nu_{R_n}^p$ is tight on
  $\mathcal{W}$ and if $\nu_{R_n}^p \rightarrow \nu$, as $n \rightarrow
  \infty$ and $R_n \rightarrow \infty$, then $T_{\asterisk} (\nu) =
  \mu^{\mathcal{A}}$.
\end{lemma}

\begin{proof}
  We note that the support of $\nu_{R_n}^p$ is contained in the set of $R_n$
  periodic distributions contained in $\mathcal{W}$ and furthermore the map
  $T_{R_n}^p$ sends $R_n$ periodic distributions into $R_n$ periodic
  distributions.
  
  Using the methods of Lemma~\ref{lemma_reduced1} and
  Corollary~\ref{corollary_reduced1}, it is simple to prove that if $y_n \in
  \mathcal{W}$ is an $R_n$ periodic distribution and if $w_{y_n} \in
  \mathcal{W}$ is an $R_n$ periodic distribution such that $w_{y_n} \in T^{p,
  - 1}_{R_n} (y_n)$ then
  \begin{equation}
    \| w_{y_n} - y_n \|_{\mathcal{H}} \leqslant K_{f, g} (\bar{\Xi}_{\ell,
    \ell'} (y_n)), \label{eq:extension2}
  \end{equation}
  where
  \[ \bar{\Xi}_{\ell, \ell'} (y_n) = \sup_{(x, z) \in \mathbb{R}^2 \times M}
     | \mathcal{G} \asterisk y_n (x, z) | r_{\ell} (x) r_{\ell'} (z), \]
  and $K_{f, g}$ is a positive increasing continuous function depending only
  on the functions $f$ and $g$. It is simple to prove that the map
  $\bar{\Xi}_{\ell, \ell'}$ is continuous on $\mathcal{W}$. Using the fact
  that, as $n \rightarrow \infty$ and $R_n \rightarrow \infty$,
  $\mu^{\mathcal{A}_R, p}_{R_n}$ converges to $\mu^{\mathcal{A}}$ weakly and
  so the sequence $\mu^{\mathcal{A}_R, p}_{R_n}$ is tight, we can use the
  bound~{\eqref{eq:extension2}} and the same methods of
  Lemma~\ref{lemma_reduced3} to prove that $\nu_{R_n}^p$ is tight.
  
  Suppose that $\nu_{R_n}^p$ weakly converges to $\nu$, we want to prove that
  $T_{\asterisk} (\nu) = \mu^{\mathcal{A}}$. Let $F_{R_n}$ be a function of
  the form $F_{R_n} (w) = G (\langle h_{1,} w \rangle, \ldots, \langle h_r,
  w \rangle)$, where $G : \mathbb{R}^r \rightarrow \mathbb{R}$ is \ a
  continuous and bounded function and $h_1, \ldots, h_r$ are smooth functions
  with support in $(- R / 2 + \mathfrak{r}, R / 2 - \mathfrak{r})^n$ and
  $\mathfrak{r} = \tmop{diam} (\tmop{supp} (\mathfrak{a}))$. For this kind of
  function we have that $F_{R_n} \circ T_{R_k}^p = F_{R_n} \circ T$ for $k
  \geqslant n$. From this observation we get
  \[ \int F_{R_n} \circ T \mathd \nu = \lim_k \int F_{R_n} \circ T \mathd
     \nu_{R_k}^p = \lim_k \int F_{R_n} \circ T_{R_k}^p \mathd \nu_{R_k}^p =
     \lim_k \int F_{R_n} \mathd \mu^{\mathcal{A}_{R_k}, p}_{R_k} = \int
     F_{R_n} \mathd \mu^{\mathcal{A}}, \]
  where the limit is taken for $k \rightarrow \infty$ and $R_k \rightarrow
  \infty$. Since the functions of the form $F_{R_n}$, for $n \in \mathbb{N}$,
  generate the space of all $\mathcal{W}$ Borel measurable functions the lemma
  is proved.
\end{proof}

\begin{proof*}{Proof of Theorem~\ref{theorem_extension1}}
  By Theorem~\ref{theorem_main} there exists a probability law $\nu_{R_n}$ on
  $\mathcal{W}_{R_n}$ such that $T_{R_n, \asterisk} (\nu_{R_n}) =
  \mu^{\mathcal{A}_{R_n}}_{R_n}$. On the other hand this implies that
  $T^p_{R_n} (\nu_{R_n}^p) = \mu^{\mathcal{A}_{R_n}, p}_{R_n}$ and so, by
  Lemma~\ref{lemma_extension1}, there exists at least one probability measure
  $\nu$ on $\mathcal{W}$ such that $T_{\asterisk} (\nu) = \mu^{\mathcal{A}}$
  and $\nu^p_{R_n} \rightarrow \nu$ weakly, as $n \rightarrow \infty$.
  
  Then, using the notations of the proof of Lemma~\ref{lemma_extension1}, we
  have that for any bounded continuous function $F_{R_n} : \mathbb{W}
  \rightarrow \mathbb{R}$
  \begin{equation}
    \int_{\mathcal{W}} F_{R_n} (\mathcal{G} \asterisk w (0, z)) \Upsilon_f (w)
    \mathd \nu^p_{R_k} (w) = \int_{\mathbb{W}} F_{R_n} (\omega) \mathd
    \kappa_{R_k}^p (\omega), \label{eq:extension3}
  \end{equation}
  where we used that
  \[ \int_{\mathcal{W}} F_{R_n} (\mathcal{G}_{R_k} \asterisk w_{R_k} (0, z))
     \Upsilon_f (w_{R_k}) \mathd \nu_{R_k} (w_{R_k}) = \int_{\mathcal{W}}
     F_{R_n} (\mathcal{G} \asterisk w (0, z)) \Upsilon_f (w) \mathd
     \nu^p_{R_k} (w), \]
  for $k \geqslant n$ and a similar equality for $\kappa_{R_k}$. Since
  $F_{R_k} (\mathcal{G} \asterisk \cdot)$ and $\Upsilon_f$ are continuous on
  $\mathcal{W}$, the left hand side of~{\eqref{eq:extension3}}, converges to
  $\int_{\mathcal{W}} F_{R_n} (\mathcal{G} \asterisk w (0, z)) \Upsilon_f (w)
  \mathd \nu (w)$ as $k \rightarrow + \infty$.
  
  Furthermore, since
  \[ \left. \frac{\mathd \kappa_{R_k}^p}{\mathd \mu^{\mathcal{A}_{R_k}, p}} =
     Z^{- 1}_{\kappa_{R_k}^p} \right. \left. \exp \left( - 4 \pi
     \int_{\mathbb{R}^2} g (z) V (\omega (z)) \mathd z \right) \right. \]
  and since $\mu^{\mathcal{A}_{R_k} \mathfrak{,} p}_{R_k}$ weakly converges to
  $\mu^{\mathcal{A}, \mathfrak{L}}$, we have that $\kappa^p_{R_k}$ weakly
  converges to $\kappa$, as $k, R_k \rightarrow \infty$. This proves that the
  right hand side of~{\eqref{eq:extension3}} converges to $\int_{\mathbb{W}}
  F_{R_n} (\omega) \mathd \kappa (\omega)$, as $k, R_k \rightarrow \infty$.
  Since the functions of the form $F_{R_n}$, for $n \in \mathbb{N}$, generate
  the space of $\mathbb{W}$ measurable functions, the theorem is proved.
\end{proof*}

\subsection{Cut-off removal with convex
potential}\label{subsection_extension2}

Hereafter we denote by $\omega_{\beta} (x)$ the function
\begin{equation}
  \omega_{\beta} (x) \assign \exp (- \beta \sqrt{(1 + | x |^2)}),
  \label{eq:exponentialweight}
\end{equation}
$\beta > 0$, $x \in \mathbb{R}^2$, and introduce the space
$\mathcal{W}_{\beta}$ in the following way
\[ \mathcal{W}_{\beta} : = (- \Delta + 1) C^0_{\exp \beta} (\mathbb{R}^2)
   \otimes_{\epsilon} \mathcal{A}^{1 / 2} (L^2 (M)) \]
where $C^0_{\exp \beta}$ is the space of continuous functions with respect to
the weighted $L^{\infty}$ norm
\[ \| g \|_{\infty, \exp \beta} \assign \sup_{x \in \mathbb{R}^2} |
   \omega_{\beta} (x) g (x) |, \]
and $M$ as before in Section \ref{section_discrete1}.

In this section we want to prove the following theorem.

\begin{theorem}
  \label{theorem_extension2}Suppose that $V$ is a convex function, and suppose
  that $\mathcal{A}$, and $\mathfrak{L}$ satisfy Hypotheses~H$\mathcal{A}$,
  H$\mathfrak{L}$ and H$\xi$ (or Hypotheses~H$\mathcal{A}$1, H$\mathfrak{L}$1
  and H$\xi$1) then there exists a unique strong solution $\phi (x, z)$ to
  equation~{\eqref{eq:main1}} (or to equation~{\eqref{eq:extension1}}) with $f
  \equiv 1$ taking values on $\mathcal{W}_{\beta}$ (for a $0 < \beta \leqslant
  \beta_0$ which depends only on $m^2$) such that for any $\mathbb{W}$
  measurable bounded function $F$ we have
  \begin{equation}
    \mathbb{E} [F (\phi (0, z))] = \int_{\mathbb{W}} F (\omega) \mathd \kappa
    (\omega) \label{eq:extension4}
  \end{equation}
  ($\kappa$ as in equation~{\eqref{eq:kappa2}}).
\end{theorem}

The proof is very similar to those of Theorem~\ref{theorem_main} and
Theorem~\ref{theorem_extension1}. For this reason we report here only the main
differences. First of all we need a replacement for
Proposition~\ref{proposition_reduced1}.

\begin{proposition}
  \label{proposition_extension1}Let $V$ (and so $V_n$) be a convex function,
  then under assumptions H$\mathfrak{L}$, H$\mathcal{A}$ and H$\xi$ \ there
  exists a unique strong solution $\phi_n$ to equation~{\eqref{eq:reduced1}}
  such that for all $\mathbb{W}$ measurable bounded continuous function $F$ we
  have
  \begin{equation}
    \mathbb{E} [F (\phi_n (0, z))] = \int_{\mathbb{W}} F (\omega) \mathd
    \kappa_n (\omega) \label{eq:reduced9}
  \end{equation}
  where $\kappa_n$ is given by expression~{\eqref{eq:kappa1}}.
\end{proposition}

\begin{proof}
  The proof is given in~{\cite{Albeverio2018elliptic}}, Theorem~2 in the case
  $\mathfrak{L}_n = 0$. The case considered here is a trivial extension. 
\end{proof}

In order to pass from equation~{\eqref{eq:reduced9}} to
equation~{\eqref{eq:extension4}} we need a generalization of
Lemma~\ref{lemma_reduced1}. We denote by $\| \cdot \|_{\exp \beta,
\ell}^{\mathcal{A}}$ respectively $\| \cdot \|^{\mathcal{A}}_{U, k}$ the
following norms
\[ \begin{array}{lll}
     \| h \|^{\mathcal{A}}_{\exp \beta, k} & \assign & \sup_{x \in
     \mathbb{R}^2} \sqrt{\sum_{j = 1}^n \sigma_j^{2 k} (h^j (x))^2
     \omega_{\beta} (x)^2},\\
     \| h \|^{\mathcal{A}}_{U, k} & \assign & \sup_{x \in U} \sqrt{\sum_{j =
     1}^n \sigma_j^{2 k} (h^j (x))^2} .
   \end{array} \]

\begin{lemma}
  \label{lemma_extension2}There exists a number $\beta_0 > 0$ (depending on
  $m^2$) such that for any $0 < \beta \leqslant \beta_0$, and for any open
  bounded set $U \subset \mathbb{R}^2$ and under
  Hypotheses~C,~H$g$,~H$\mathcal{A}$,~H$\mathfrak{L}$,~H$\xi$ and $f \equiv 1$
  we have
  \begin{eqnarray}
    \| \bar{\psi}_n \|^{\mathcal{A}}_{\exp \beta, - 1} & \lesssim & \left\|
    \exp (\alpha \Xi_{\ell', n} r_{- \ell'} (x) \mathfrak{w}^{- 1} (z))
    \mathbbm{1}_{g (z) \not{=} 0} \omega_{\beta} (x) \right\|_{\infty}
    \label{eq:extension5} \\
    \| (- \Delta_x + m^2 + \mathfrak{L}_n) (\bar{\psi}_n) \|^{\mathcal{A}}_{U,
    - 1} & \lesssim & \nobracket \nobracket \nobracket \| \exp (\alpha'
    (\Xi_{\ell', n} r_{- \ell'} (x) \mathfrak{w}^{- 1} (z) \nobracket
    \nobracket \nobracket \nobracket \nobracket \nobracket + \nonumber\\
    &  & + \nobracket \nobracket \nobracket \nobracket \nobracket \nobracket
    \| \bar{\psi}_n \|^{\mathcal{A}}_{\exp \beta, - 1}) \omega_{- \beta} (x))
    \| \nobracket \nobracket \nobracket_{C^0 (U_g)}  \label{eq:extension6}
  \end{eqnarray}
  where $U_g = U \times \left\{ g (z) \not{=} 0 \right\} \subset \mathbb{R}^2
  \times M$, uniformly in $n$ (where $\Xi_{\ell, n}$ is defined as in Lemma
  \ref{lemma_reduced1}).
\end{lemma}

\begin{proof}
  The proof is verbatim the same as for Lemma~\ref{lemma_reduced1} where we
  replace the function $r_{\theta, \ell} (x) = (1 + \theta | x |^2)^{- \ell}$
  by the function $\omega_{\beta}$, defined by {\eqref{eq:exponentialweight}},
  and we use the fact for $\beta$ small enough we have
  \[ \left| \frac{- 2 | \nabla \omega_{\beta} |^2 + \omega_{\beta} \Delta
     \omega_{\beta}}{2 \omega_{\beta}^2} \right| < m^2 . \]
\end{proof}

The inequality~{\eqref{eq:extension6}} implies that, for any bounded open
subset $U$ of $\mathbb{R}^2$ we have
\begin{equation}
  \| (- \Delta_x + m^2 + \mathfrak{L}) (\bar{\phi}_n) \|_{\mathcal{H}_U}
  \leqslant K_U (\Xi_{\ell, n}) \label{eq:extension7}
\end{equation}
where $\| \cdot \|_{\mathcal{H}_U}$ is the natural norm of the Hilbert space
$\mathcal{H}_U = L^2 (U) \otimes_H \mathcal{A} (L^2 (M))$, and $K_U$ is a
positive increasing continuous function depending only on $U$.
Inequality~{\eqref{eq:extension7}} guarantees us enough compactness to
generalize \ Lemma~\ref{lemma_reduced2}, Lemma~\ref{lemma_reduced3} and
Lemma~\ref{lemma_reduced4} in order to prove the existence of a weak solution
to equation~{\eqref{eq:main1}} satisfying the dimensional reduction principle
when $f \equiv 1$ and under Hypotheses~C, H$g$, H$\mathcal{A}$,
H$\mathfrak{L}$, H$\xi$. We can generalize the described result under
Hypotheses~C, H$g$, H$\mathcal{A}$1, H$\mathfrak{L}$1, H$\xi$1 using a similar
strategy.

It remains to prove the uniqueness of the solution to
equation~{\eqref{eq:main1}} (or~{\eqref{eq:extension1}}) when $V$ is convex.
This can be done using the methods of Lemma~\ref{lemma_reduced1} and
Lemma~\ref{lemma_extension2}. We give here only a sketch of the proof.

Let $\phi_1$ and $\phi_2$ be two strong solutions to
equation~{\eqref{eq:main1}}. It is simple to prove that $\phi_1 (x, \cdot) -
\phi_2 (x, \cdot) \in \mathcal{A} (L^2 (M))$, so the following function is
well defined
\[ \Psi_{\beta} (x)^2 = \omega_{2 \beta} (x) \int_M (\mathcal{A}^{- 1} (\phi_1
   (x, z') - \phi_2 (x, z')))^2 \mathd z, \]
where $\omega_{2 \beta}$ is defined as in equation
{\eqref{eq:exponentialweight}}. The function $\Psi^2_{\beta}$ belongs to $C^2
(\mathbb{R}^2)$ \ and goes to zero at infinity. This means that the maximum is
realized in at least one point $\bar{x} \in \mathbb{R}^2 .$ Making some
computations similar to the ones of Lemma~\ref{lemma_reduced1} we obtain
\[ \begin{array}{lll}
     m^2 \Psi_{\beta} (\bar{x})^2 & \leqslant & - \int_M g (z) (\phi_1
     (\bar{x}, z) - \phi_2 (\bar{x}, z))(V' (\phi_1 (\bar{x}, z)) - V'
     (\phi_2 (\bar{x}, z))) \mathd z\\
     &  & \quad - \left( \frac{- 2 | \nabla \omega_{\beta} |^2 +
     \omega_{\beta} \Delta \omega_{\beta}}{2 \omega_{\beta}^2} \right)
     \Psi_{\beta} (\bar{x})^2 .
   \end{array} \]
On the other hand there exists a point $\theta_{\bar{x}, z} \in [\phi_1
(\bar{x}, z), \phi_2 (\bar{x}, z)]$ such that
\[ V' (\phi_1 (\bar{x}, z)) - V' (\phi_2 (\bar{x}, z)) \leqslant V''
   (\theta_{\bar{x}, z}) (\phi_1 (\bar{x}, z) - \phi_2 (\bar{x}, z)) . \]
Choosing $\beta > 0$ small enough we obtain
\[ \sup_{x \in \mathbb{R}^2} \Psi_{\beta} (x)^2 \lesssim - \int_M g (z) V''
   (\theta_{\bar{x}, z}) (\phi_1 (\bar{x}, z) - \phi_2 (\bar{x}, z))^2
   \mathd z \leqslant 0, \]
since $V$ is convex. This implies that $\Psi_{\beta} (x) \equiv 0$ and so that
$\phi_1 (x, z) = \phi_2 (x, z)$.

\section{The exponential interaction on
$\mathbb{R}^2$}\label{section_exponential}

In this section we want to consider the following elliptic SPDE
\begin{equation}
  (- \Delta_x - \Delta_z + m^2) (\phi) + g (z) \alpha \exp (\alpha \phi -
  \infty) = \xi \label{eq:exponential1}
\end{equation}
where $x \in \mathbb{R}^2$, $z \in M =\mathbb{R}^2$, $\xi = \xi (x, z)$ is a
standard Gaussian white noise on $\mathbb{R}^4$, $| \alpha | < 4 \sqrt{2}
\pi$, $m > 0$, $g : \mathbb{R}^2 \rightarrow \mathbb{R}$ is a non-negative
smooth function with compact support, and where $- \infty$ means that the
equation should be properly renormalized. In order to give a meaning to the
previous equation we formally subtract from the solution $\phi$ the solution
to the linear equation (i.e. equation {\eqref{eq:exponential1}} with $g = 0$)
which means that we consider the equation for the unknown $\bar{\phi}$:
\begin{equation}
  (- \Delta_x - \Delta_z + m^2) (\bar{\phi}) + g (z) \alpha \exp (\alpha
  \bar{\phi}) \eta (x, z) = 0 \label{eq:exponential2},
\end{equation}
where $\eta (x, z) = \exp^{\diamond} (\alpha \mathcal{I} \xi)$ is a
renormalized version of the distribution $\exp (\alpha \mathcal{I} \xi -
\infty)$, where $\exp^{\diamond}$ denotes the Wick exponential of the Gaussian
distribution $\mathcal{I} \xi$. Hereafter we denote in general by $B^s_{p, q,
\ell} (\mathbb{R}^{d + 2})$ the weighted Besov space of indices $1 \leqslant p
\leqslant \infty$ and $1 \leqslant q \leqslant \infty$ and weight given by
$\bar{r}_{\ell} (x, z) = \left( \sqrt{| x |^2 + | z |^2 + 1} \right)^{-
\ell}$, where $\ell \in \mathbb{R}$ (see~{\cite{Triebel2006}}). It is well
known that $\mathcal{I} \xi \in B^{- \delta}_{p, p, \ell}$ for any $1
\leqslant p \leqslant \infty$, $\delta > 0$ and $\ell > 0$ (see
e.g.~{\cite{GH18}}).

\

In the following we shall take $d = 2$ and give a rigorous meaning to
equation~{\eqref{eq:exponential2}} (and so to
equation~{\eqref{eq:exponential1}}) when the exponent $| \alpha | < \alpha_{\text{max}}$ (see equation \eqref{eq:alphamax} for the definition $\alpha_{\text{max}}$) and, when $\alpha \bar{\phi} \leqslant 0$, prove that there exists only
one solution to equation~{\eqref{eq:exponential1}}.

Furthermore we want to prove that dimensional reduction holds for the unique
solution to equation~{\eqref{eq:exponential1}}, namely that, if we consider
the measure $\kappa_g$ given by
\begin{equation}
  \frac{\mathd \kappa_g}{\mathd \mu} (\omega) = \exp \left( - 4 \pi
  \int_{\mathbb{R}^2} g (z) \exp^{\diamond} (\alpha \omega) (\mathd z)
  \right), \label{eq:exponentiallaw1}
\end{equation}
where $\omega \in B^{- \delta}_{p, p, \ell} (\mathbb{R}^2)$ and where $\mu$ is
the law of $(- \Delta_z + m^2)^{- 1 / 2} (\xi)$ on $B^{- \delta}_{p, p, \ell}
(\mathbb{R}^2)$, we have
\[ \mathbb{E} [F (\phi (0, \cdot))] = \int F (\omega) \mathd \kappa_g (\omega)
   \label{eq:exponentialreduction1}, \]
for any bounded measurable function $F$ defined on $B^{- \delta}_{p, p, \ell}
(\mathbb{R}^2)$.

In order to prove the existence and uniqueness of the solution $\bar{\phi}$ to
equation~{\eqref{eq:exponential2}} and dimensional
reduction~{\eqref{eq:exponentialreduction1}} for the random field $\phi =
\bar{\phi} + \mathcal{I} \xi$, we need to introduce the following two
approximate equations
\begin{eqnarray}
  & (- \Delta_x - \Delta_z + m^2) (\phi_{\epsilon}) +
  \mathfrak{a}_{\epsilon}^{\asterisk 2} \asterisk [g (z) \alpha \exp (\alpha
  \phi_{\epsilon} - C_{\epsilon})] = \mathfrak{a}_{\epsilon} \asterisk \xi & 
  \label{eq:exponential3}\\
  & (- \Delta_x - \Delta_z + m^2) (\bar{\phi}_{\epsilon}) +
  \mathfrak{a}_{\epsilon}^{\asterisk 2} \asterisk [g (z) \alpha \exp (\alpha
  \bar{\phi}_{\epsilon}) \eta_{\epsilon} (x, z)] = 0 & 
  \label{eq:exponential4}
\end{eqnarray}
where $\mathfrak{a}$ is a positive function satisfying
Hypothesis~H$\mathcal{A} 1$ (see Section~\ref{subsection_extension1}) and
$C_{\epsilon} : = \frac{\alpha^2}{2} \mathbb{E} [\mathcal{I}
(\mathfrak{a}_{\epsilon} \asterisk \xi)]$, $\bar{\phi}_{\epsilon} : =
\phi_{\epsilon} - \mathcal{I} (\mathfrak{a}_{\epsilon} \asterisk \xi)$, and
$\eta_{\varepsilon}$ is the positive measure defined as
\[ \eta_{\varepsilon} (\mathd x, \mathd z) \assign \exp^{\diamond} (\alpha
   \mathcal{I} (\mathfrak{a}_{\epsilon} \asterisk \xi)) \mathd x \mathd z =
   \exp (\alpha \mathcal{I} (\mathfrak{a}_{\epsilon} \asterisk \xi) -
   C_{\varepsilon}) \mathd x \mathd z. \]
The constant $C_{\epsilon}$ is chosen in such a way that $\eta_{\epsilon}
\rightarrow \eta$ in $B^s_{p, p, \ell}$ for suitable $s < 0$, $1 < p \leqslant
2$ and $\ell > 0$ (see Section~\ref{subsection_probabilistic}).

For equations~{\eqref{eq:exponential3}} and~{\eqref{eq:exponential4}} we have
the following fundamental dimensional reduction result.

\begin{theorem}
  \label{theorem_regularexponential}Equation~{\eqref{eq:exponential4}} admits
  a unique solution in $C^0_{\ell} (\mathbb{R}^4)$ for $\ell$ big enough.
  Furthermore we have that, for any $\epsilon > 0$ and any bounded continuous
  function $F$
  \begin{equation}
    \mathbb{E} [F (\phi_{\epsilon} (0, z))] = \int_{\mathcal{C}^{0 -}_{\ell}
    (\mathbb{R}^2)} F (\omega) \mathd \kappa_{\epsilon} (\omega)
    \label{eq:exponential5},
  \end{equation}
  for all $z \in \mathbb{R}^2$, where
  \[ \frac{\mathd \kappa_{\epsilon}}{\mathd \mu_{\epsilon}} = \exp \left( - 4
     \pi \int_{\mathbb{R}^2} g (z) \exp (\alpha \omega (z) - \alpha
     C_{\epsilon}) \mathd z \right) \]
  and $\mu_{\epsilon}$ is the Gaussian measure on $\mathcal{C}^{0 -}_{\ell}$
  with covariance $\mathbb{E} [\omega (z) \omega (z')] =
  \mathfrak{a}^{\asterisk 2}_{\epsilon} \asterisk \mathcal{G}_z (z - z')$,
  where $\mathcal{G}_z$ is the Green function of the operator $(- \Delta_z +
  m^2)$. Finally the unique solution to equation~{\eqref{eq:exponential4}}
  satisfies $\alpha \bar{\phi}_{\varepsilon} \leqslant 0$.
\end{theorem}

\begin{proof}
  The proof is an application of Theorem~\ref{theorem_extension1} to
  equation~{\eqref{eq:exponential3}} using the fact that $\exp (\alpha y -
  C_{\epsilon})$ satisfies Hypothesis~C in Section~\ref{section_discrete}.
  
  The fact that $\alpha \bar{\phi}_{\varepsilon} \leqslant 0$ follows from an
  application of the maximum principle to the function $\hat{z} \mapsto
  \bar{\phi}_{\epsilon} (\hat{z})  \bar{r}_{\ell} (\lambda \hat{z})$ where
  $\hat{z} = (x, z) \in \mathbb{R}^4$ and $\lambda > 0$ small enough.
\end{proof}

In order to prove existence, uniqueness and the reduction principle for
equation~{\eqref{eq:exponential1}} we have now to study the behavior of the
regularized noises $\eta_{\epsilon}$ and their convergence to $\eta$ in the
Besov spaces of the form $B^s_{p, p, \ell}$, as $\varepsilon \rightarrow 0$.
Once we have established this convergence, we can give a meaning to
equation~{\eqref{eq:exponential2}} and we are able to prove that there exists
a subsequence of $\epsilon_n \rightarrow 0$ such that $\phi_{\epsilon_n}
\rightarrow \phi$ in probability in $B^{- \delta}_{p, p, \ell}$, where $\phi$
solves {\eqref{eq:exponential1}}. This fact will permit us to prove
equality~{\eqref{eq:exponentialreduction1}}.

\subsection{Probabilistic analysis}\label{subsection_probabilistic}

In this subsection we propose an analysis of the regularity of the noise
$\eta$ and $\eta_{\epsilon}$, for $\epsilon > 0$. First of all we note that
\begin{eqnarray}
  \eta_{\epsilon} & = & \sum_{k = 0}^{\infty} \frac{\alpha^k}{k!} (\mathcal{I}
  \xi_{\epsilon})^{\diamond k},  \label{eq:noise1}\\
  \eta & = & \sum_{k = 0}^{\infty} \frac{\alpha^k}{k!} (\mathcal{I}
  \xi)^{\diamond k},  \label{eq:noise2}
\end{eqnarray}
where $I$ is the integral operator defined in Section \ref{section_discrete},
\[ (\mathcal{I} \xi_{\epsilon})^{\diamond k} = \underbrace{\mathcal{I}
   \xi_{\epsilon} \diamond \ldots \diamond \mathcal{I} \xi_{\epsilon}}_{k
   \quad \tmop{times}} , \]
the symbol $\diamond$ denotes the Wick product, $\xi_{\epsilon} =
\mathfrak{a}_{\epsilon} \asterisk \xi$ and $\mathcal{I} \xi_{\epsilon}$ is
defined correspondingly. The previous expressions are well defined as $L^2
(\mu)$ convergent series for $| \alpha | < 4 \sqrt{2} \pi$. Indeed we have
that for any \ smooth function $g$ exponentially decaying at infinity \
\begin{equation}
  \mathbb{E} [| \langle \eta_{\epsilon}, g \rangle |^2] = \int_{\mathbb{R}^8}
  g (\hat{z}) g (\hat{z}') \exp (\alpha^2 \mathcal{G}_{\epsilon} (\hat{z} -
  \hat{z}')) \mathd \hat{z} \mathd \hat{z}' \label{eq:noise3}
\end{equation}
where, hereafter, we write $\mathcal{G}_{\epsilon} = \mathfrak{a}^{\asterisk
2}_{\epsilon} \asterisk \mathcal{G}$, where $\mathcal{G}$ is the Green
function associated with the operator $(- \Delta_{\hat{z}} + m^2)^{- 2}$,
$\hat{z} = (x, z) \in \mathbb{R}^4$. It is well known (see Proposition A.1
of~{\cite{Albeverio2002}}) that for $\hat{z} \in \mathbb{R}^4$ such that $|
\hat{z} | \leqslant 1$ there exists a constant $C_1 > 0$ for which the
following inequality holds
\begin{equation}
  \mathcal{G} (\hat{z}) \leqslant - \frac{2}{(4 \pi)^2} \log_+ (| \hat{z} |) +
  C_1 \label{eq:probabilistic1} .
\end{equation}
Furthermore for $| \hat{z} | \geqslant 1$ there exists two constants $C_2, C_3
> 0$ for which we get
\begin{equation}
  \mathcal{G} (\hat{z}) \leqslant C_2 \exp (- C_3 | \hat{z} |)
  \label{eq:probabilistic2} .
\end{equation}
An easy consequence of the inequalities~{\eqref{eq:probabilistic1}}
and~{\eqref{eq:probabilistic2}} are the following inequalities
\begin{equation}
  \mathfrak{a}_{\epsilon}^{\asterisk 2} \asterisk \mathcal{G} (\hat{z})
  \leqslant - \frac{2}{(4 \pi)^2} \log_+ (| \hat{z} |) + C_4,
  \label{eq:probabilistic3}
\end{equation}
when $4 \tmop{diam} (\tmop{supp} (\mathfrak{a})) \epsilon \leqslant | \hat{z}
| \leqslant 1$ and for some constant $C_4 > 0$; \
\begin{equation}
  \mathfrak{a}_{\epsilon}^{\asterisk 2} \asterisk \mathcal{G} (\hat{z})
  \leqslant - \frac{2}{(4 \pi)^2} \log (\epsilon) + C_5,
  \label{eq:probabilistic4}
\end{equation}
when $| \hat{z} | < 2 \tmop{diam} (\tmop{supp} (\mathfrak{a})) \epsilon$ and
for some constant $C_5 > 0$; and finally
\begin{equation}
  \mathfrak{a}_{\epsilon}^{\asterisk 2} \asterisk \mathcal{G} (\hat{z})
  \leqslant C_6 \exp (- C_7 | \hat{z} |) \label{eq:probabilistic5},
\end{equation}
when $| \hat{z} | \geqslant 1$ and for some constants $C_6, C_7 > 0$. It is
important to note that the constants $C_4, C_5, C_6, C_7$ are independent of
$\epsilon$. The previous inequalities and eq.~{\eqref{eq:noise3}} imply that
$\mathbb{E} [| \langle \eta_{\epsilon}, g \rangle |^2] < + \infty$ for any
$\epsilon \geqslant 0$ and $| \alpha | < 4 \sqrt{2} \pi$ (with $\eta_0 =
\eta$).

Let $D_k$, with $k \geqslant - 1$, be the functions related to \
Littlewood-Paley block (see Appendix \ref{appendix_besov} for the definition
of this concept). Using Remark \ref{remark_decay}, we can suppose that there
exist some constants $\gamma_{- 1}, \gamma > 0$ and $0 < \theta_{- 1}, \theta
< 1$ such that
\[ | D_{- 1} (\hat{z}) | \lesssim \exp (- \gamma_{- 1} | \hat{z} |^{\theta_{-
   1}}) \quad \tmop{and} \quad | D_k (\hat{z}) | \lesssim 2^{4 k} \exp (-
   \gamma 2^{k \theta} | \hat{z} |^{\theta}), \]
where $k > 0$. Using this decay at infinity, we get
\[ \mathbb{E} [\| \langle \eta_{\epsilon}, D_k (\hat{z} - \cdot) \rangle
   \|_{L^2_{r_{\ell}}}^2] = \int_{\mathbb{R}^4} \mathbb{E} [\langle
   \eta_{\epsilon}, D_k \rangle^2] (1 + | \hat{z} |)^{- 2 \ell} \mathd \hat{z}
   \lesssim \mathbb{E} [\langle \eta_{\epsilon}, D_k \rangle^2] < + \infty, \]
whenever $\alpha < 4 \sqrt{2} \pi$, and where we used the invariance in law of
$\eta_{\epsilon}$ with respect to translations. Since $D_k (x, z) = 2^{4 k}
D_0 (2^k x, 2^k z)$ and using~{\eqref{eq:probabilistic3}}
and~{\eqref{eq:probabilistic5}}, we obtain
\begin{eqnarray}
  \mathbb{E} [\langle \eta_{\epsilon}, D_k \rangle^2] & = &
  \int_{\mathbb{R}^8} D_k (\hat{z}) D_k (\hat{z}') \exp (\alpha^2 \cdot
  \mathfrak{a}_{\epsilon}^{\asterisk 2} \asterisk \mathcal{G} (\hat{z} -
  \hat{z}')) \mathd \hat{z} \mathd \hat{z}' \nonumber\\
  & \lesssim & \int_{\mathbb{R}^8} | D_k (\hat{z}) D_k (\hat{z}') | | \hat{z}
  - \hat{z}' |^{- \frac{2 \alpha^2}{(4 \pi)^2}} \mathd \hat{z} \mathd \hat{z}'
   \nonumber\\
  &  & + \left( 1 + \epsilon^{4 - \frac{2 \alpha^2}{(4 \pi)^2}} \right)
  \int_{\mathbb{R}^8} | D_k (\hat{z}) D_k (\hat{z}') | \mathd \hat{z} \mathd
  \hat{z}'  \nonumber\\
  & \lesssim & 2^{\frac{2 \alpha^2}{(4 \pi)^2} k} \left( \int_{\mathbb{R}^8}
  \left( 1 + | \hat{z} - \hat{z}' |^{- \frac{2 \alpha^2}{(4 \pi)^2}} \right) |
  D_0 (\hat{z}) D_0 (\hat{z}') | \mathd \hat{z} \mathd \hat{z}' \right)
  \lesssim 2^{\frac{2 \alpha^2}{(4 \pi)^2} k}  \label{eq:eta1}
\end{eqnarray}
where all the constants implied in the symbol $\lesssim$ are independent of
$\epsilon$. This means that
\begin{equation}
  \mathbb{E} [\| \eta_{\epsilon} \|^2_{B^s_{2, 2, \ell}}] \lesssim \sum^{+
  \infty}_{k = - 1} 2^{\frac{2 \alpha^2}{(4 \pi)^2} k + 2 s k},
  \label{eq:eta2}
\end{equation}
which is finite and uniformly bounded in $\epsilon$ whenever $s < -
\frac{\alpha^2}{(4 \pi)^2}$ and for $\ell > 0$ large enough. Furthermore we
have that
\begin{eqnarray}
  \mathbb{E} [\| \eta - \eta_{\epsilon} \|_{B^s_{2, 2, \ell}}^2]^{1 / 2} &
  \lesssim & \left( \sum_{k = - 1}^{+ \infty} 2^{2 s k} \mathbb{E}[| \langle
  \eta - \eta_{\epsilon}, D_k \rangle |^2] \right)^{1 / 2} \nonumber\\
  & \lesssim & \sum_{k = - 1}^{+ \infty} 2^{s k} \sum^{+ \infty}_{n = 0}
  \frac{| \alpha |^n}{n!} \mathbb{E} [| \langle \mathcal{I} \xi^{\diamond n} -
  \mathcal{I} \xi_{\epsilon}^{\diamond n}, D_k \rangle |^2]^{1 / 2} \nonumber
\end{eqnarray}
for $\ell$ big enough. On the other hand we get
\begin{eqnarray}
  \sum_{k = - 1}^{+ \infty} 2^{s k} \sum^{+ \infty}_{n = 0} \frac{| \alpha
  |^n}{n!} \mathbb{E} [| \langle \mathcal{I} \xi^{\diamond n} - \mathcal{I}
  \xi_{\epsilon}^{\diamond n}, D_k \rangle |^2]^{1 / 2} & \leqslant & \sum_{k
  = - 1}^{+ \infty} 2^{s k} \sum^{+ \infty}_{n = 0} \frac{| \alpha |^n}{n!}
  \mathbb{E} [| \langle \mathcal{I} \xi^{\diamond n}, D_k \rangle |^2]^{1 / 2}
   \nonumber\\
  &  & + \sum_{k = - 1}^{+ \infty} 2^{s k} \sum^{+ \infty}_{n = 0} \frac{|
  \alpha |^n}{n!} \mathbb{E} [| \langle \mathcal{I} \xi_{\epsilon}^{\diamond
  n}, D_k \rangle |^2]^{1 / 2} \nonumber\\
  & \leqslant & \frac{1}{2} \sum_{k = - 1}^{+ \infty} 2^{s k} \sum^{+
  \infty}_{n = 0} \frac{\beta^{2 n} | \alpha |^{2 n}}{(n!)^2} \mathbb{E} [|
  \langle \mathcal{I} \xi^{\diamond n}, D_k \rangle |^2]  \nonumber\\
  &  & + \frac{1}{2} \sum_{k = - 1}^{+ \infty} 2^{s k} \sum^{+ \infty}_{n =
  0} \frac{\beta^{2 n} | \alpha |^{2 n}}{(n!)^2} \mathbb{E} [| \langle
  \mathcal{I} \xi_{\epsilon}^{\diamond n}, D_k \rangle |^2]  \nonumber\\
  &  & + \frac{\beta^2}{\beta^2 - 1} \sum_{k = - 1}^{+ \infty} 2^{s k}
  \nonumber\\
  & \leqslant & \frac{1}{2} \sum_{k = - 1}^{+ \infty} 2^{s k} \mathbb{E} [|
  \langle : \exp (\beta \alpha \mathcal{I} \xi_{\epsilon}^{\diamond n}) :, D_k
  \rangle |^2]  \nonumber\\
  &  & + \frac{1}{2} \sum_{k = - 1}^{+ \infty} 2^{s k} \mathbb{E} [| \langle
  : \exp (\beta \alpha \mathcal{I} \xi^{\diamond n}) :, D_k \rangle |^2] 
  \nonumber\\
  &  & + \frac{\beta^2}{\beta^2 - 1} \sum_{k = - 1}^{+ \infty} 2^{s k} < C
  \nonumber
\end{eqnarray}
for some constant $C > 0$ independent of $\epsilon$, whenever $\beta > 1$ is
arbitrary small and $s < - \frac{\beta \alpha^2}{(4 \pi)^2}$, where we use an
estimate similar to {\eqref{eq:eta1}} and {\eqref{eq:eta2}} for the measure $:
\exp (\beta \alpha \mathcal{I} \xi^{\diamond n}) :$. Furthermore, since
$\mathfrak{a}_{\epsilon}$ is a regular mollifier and by the properties of Wick
product, we obtain
\[ \mathbb{E} [| \langle \mathcal{I} \xi^{\diamond n} - \mathcal{I}
   \xi_{\epsilon}^{\diamond n}, D_k \rangle |^2] \rightarrow 0, \]
as $\varepsilon \rightarrow 0$ and for any $k \geq - 1$ and any $n \in
\mathbb{N}$. The above inequality, together with the Lebesgue dominated
convergence theorem, implies that
\[ \mathbb{E} [\| \eta - \eta_{\epsilon} \|_{B^s_{2, 2, \ell}}^2]^{1 / 2}
   \rightarrow 0, \]
when $\epsilon \rightarrow 0$, for any $\ell > 0$ large enough. We have thus
proven the following lemma.

\begin{lemma}
  \label{lemma_L2}For $| \alpha | < 4 \sqrt{2} \pi$ and $s < -
  \frac{\alpha^2}{(4 \pi)^2}$ and $\ell > 0$ large enough we have that
  $\eta_{\epsilon} \rightarrow \eta$ as $\epsilon \rightarrow 0$ in $L^2
  (\mathcal{W} ; B^s_{2, 2, \ell} (\mathbb{R}^4), \mathd \mu)$ and thus in
  probability in $B^s_{2, 2, \ell} (\mathbb{R}^4)$.
\end{lemma}

We want to use the previous lemma to prove the following theorem.

\begin{theorem}
  \label{theorem_Lp}For $| \alpha | < 4 \sqrt{2} \pi$, $1 < p \leqslant 2$, $s
  < - \frac{\alpha^2 (p - 1)}{(4 \pi)^2}$ and $\ell > 0$ large enough we have
  that $\eta_{\epsilon} \rightarrow \eta$, as $\epsilon \rightarrow 0$, in
  $L^p (\mathcal{W} ; B^s_{p, p, \ell} (\mathbb{R}^4), \mathd \mu)$ and thus
  in probability in $B^s_{p, p, \ell} (\mathbb{R}^4)$.
\end{theorem}

In order to prove Theorem \ref{theorem_Lp} we introduce the following lemma.

\begin{lemma}
  \label{lemma_garban}Let $B_r (\hat{z})$ be the ball of radius $r$ and center
  in $\hat{z} \in \mathbb{R}^4$ then for any $r < R$, and for any $1 < p < 2$
  \[ \mathbb{E} \left[ \left( \int_{B_r (\hat{z})} \mathd \eta_{\epsilon}
     \right)^p \right] \lesssim r^{- \frac{\alpha^2}{(4 \pi)^2} p (p - 1) + 4
     p} \]
  where the constants depend only $R$ and are uniform on $\epsilon \rightarrow
  0$ and $| \alpha | < 4 \sqrt{2} \pi$.
\end{lemma}

\begin{proof}
  The proof can be found in Proposition 2.7 of~{\cite{Vargas2010}}.
\end{proof}

\begin{proof*}{Proof of Theorem~\ref{theorem_Lp}}
  First of all we have that, for any $\varepsilon > 0$ and $k \geqslant 0$, \
  \
  \[ \mathbb{E} [\| \langle \eta_{\epsilon}, D_k (\hat{z} - \cdot) \rangle
     \|_{L^p_{\ell}}^p] = \int_{\mathbb{R}^4} \mathbb{E} [| \langle
     \eta_{\epsilon}, D_k \rangle |^p] (1 + | \hat{z} |)^{- p \ell} \mathd
     \hat{z} \lesssim \mathbb{E} [| \langle \eta_{\epsilon}, D_k \rangle |^p]
     . \]
  Then
  \[ \mathbb{E} [| \langle \eta_{\epsilon}, D_k \rangle |^p] \lesssim
     \mathbb{E} \left[ \left( \int_{B_{r_k} (0)} | D_k (\hat{z}) | \mathd
     \eta_{\epsilon} (\hat{z}) \right)^p \right] +\mathbb{E} \left[ \left(
     \int_{\mathbb{R}^4 \setminus B_{r_k} (0)} | D_k (\hat{z}) | \mathd
     \eta_{\epsilon} (\hat{z}) \right)^2 \right]^{p / 2} . \]
  If we choose $r_k = 2^{- k} (\log (2^{\beta k}))^u$, for some $\beta, u >
  0$, from Lemma~\ref{lemma_garban} we have, for any $k \geqslant 0$, \
  \[ \mathbb{E} \left[ \left( \int_{B_{r_k} (0)} | D_k (\hat{z}) | \mathd
     \eta_{\epsilon} (\hat{z}) \right)^p \right] \lesssim 2^{4 k p}
     2^{\frac{\alpha^2}{(4 \pi)^2} p (p - 1) k - 4 p k} \log (2^{\beta k})^{u
     \left( - \frac{\alpha^2}{(4 \pi)^2} p (p - 1) + 4 p \right)} \]
  Furthermore, if we denote by $\gamma > 0$ and $0 < \theta < 1$ the real
  number such that $D_0 (\hat{z}) \lesssim \exp (- \gamma | \hat{z}
  |^{\theta})$ and choosing $u = \frac{1}{\theta'}$ where $0 < \theta' <
  \theta$, we obtain
  \[ \begin{array}{c}
       \mathbb{E} \left[ \left( \int_{\mathbb{R}^4 \setminus B_{r_k} (0)} |
       D_k (\hat{z}) | \mathd \eta_{\epsilon} (\hat{z}) \right)^2 \right]^{p /
       2} \lesssim\\
       \lesssim 2^{8 k - \gamma \beta k} \int_{| \hat{z} | > r_k, | \hat{z} |
       > r_k} e^{\gamma \beta k - \gamma 2^{k \theta} (| \hat{z} |^{\theta} +
       | \hat{z}' |^{\theta}) + \alpha^2 \mathcal{G}_{\epsilon} (\hat{z} -
       \hat{z}')} \mathd \hat{z} \mathd \hat{z}'\\
       \lesssim 2^{8 k - \gamma \beta k} \int_{| \hat{z} | > r_k, | \hat{z} |
       > r_k} e^{- \gamma 2^{k \theta} (| \hat{z} |^{\theta - \theta'} + |
       \hat{z}' |^{\theta - \theta'}) + \alpha^2 \mathcal{G}_{\epsilon}
       (\hat{z} - \hat{z}')} \mathd \hat{z} \mathd \hat{z}'\\
       \lesssim 2^{8 k - \gamma \beta k} \int_{\mathbb{R}^8} e^{- \gamma (|
       \hat{z} |^{\theta - \theta'} + | \hat{z}' |^{\theta - \theta'}) +
       \alpha^2 \mathcal{G}_{\epsilon} (\hat{z} - \hat{z}')} \mathd \hat{z}
       \mathd \hat{z}' \lesssim 2^{8 k - \gamma \beta k}
     \end{array} \]
  If we choose $\beta$ such that $8 - \gamma \beta < \frac{\alpha^2}{(4
  \pi)^2} p (p - 1)$ we obtain that
  \begin{equation}
    \mathbb{E} [| \langle \eta_{\epsilon}, D_k \rangle |^p] \lesssim
    2^{\frac{\alpha^2}{(4 \pi)^2} p (p - 1) k} \log (2^{\beta k})^{u \left( -
    \frac{\alpha^2}{(4 \pi)^2} p (p - 1) + 4 p \right)} \label{eq:Lp1}
  \end{equation}
  uniformly in $\epsilon$. This implies that $\mathbb{E} [\| \eta_{\epsilon}
  \|_{B^s_{p, p, \ell}}^p]$ is bounded for $s < - \frac{\alpha^2}{(4 \pi)^2}
  (p - 1)$ uniformly in $\epsilon > 0$.
  
  Furthermore, by Lemma~\ref{lemma_L2}, $\mathbb{E} [| \langle
  \eta_{\epsilon}, D_k \rangle |^2]$ is uniformly bounded. This means that the
  random variables $| \langle \eta_{\epsilon}, D_k \rangle |^p$ (for $p < 2$)
  are uniformly integrable. On the other hand, by Lemma~\ref{lemma_L2} and
  since $D_k \in B^2_{2, 2, \ell}$, $| \langle \eta_{\epsilon}, D_k \rangle
  |^p$ converges to $| \langle \eta, D_k \rangle |^p$ in probability, we have
  that $\mathbb{E} [| \langle \eta_{\epsilon}, D_k \rangle |^p] \rightarrow
  \mathbb{E} [| \langle \eta, D_k \rangle |^p]$, as $\varepsilon \rightarrow
  0$. Thus by the bound~{\eqref{eq:Lp1}} we can use the Lebesgue dominated
  converge theorem for computing the limit of $\mathbb{E} [\| \eta_{\epsilon}
  \|_{B^s_{p, p, \ell}}^p]$, obtaining
  \[ \mathbb{E} [\| \eta_{\epsilon} \|_{B^s_{p, p, \ell}}^p] \rightarrow
     \mathbb{E} [\| \eta \|_{B^s_{p, p, \ell}}^p], \]
  as $\varepsilon \rightarrow 0$. But, since $\mathbb{E} [\| \eta_{\epsilon}
  \|_{B^s_{p, p, \ell}}^p]$ is uniformly bounded and since, by
  Lemma~\ref{lemma_L2}, as $\varepsilon \rightarrow 0$, $\eta_{\epsilon}
  \rightarrow \eta$ in probability in $B^{- 2}_{2, 2, \ell}$, we have that
  $\eta_{\epsilon}$ converges weakly to $\eta$ in $L^p (\mathcal{W}, B^s_{p,
  p, \ell}, \mathd \mu)$. Since $L^p (\mathcal{W}, B^s_{p, p, \ell}, \mathd
  \mu)$ is a uniformly convex space, being the space of $L^p$ functions taking
  values in the uniformly convex space $B^s_{p, p, \ell}$
  (see~{\cite{Boyko2009}}), by the weak convergence $\eta_{\epsilon}
  \rightarrow \eta$ and the convergence of the $L^p$ norm $\mathbb{E} [\|
  \eta_{\epsilon} \|_{B^s_{p, p, \ell}}^p] \rightarrow \mathbb{E} [\| \eta
  \|_{B^s_{p, p, \ell}}^p]$, we obtain that $\eta_{\epsilon}$ converges
  strongly to $\eta$ in $B^s_{p, p, \ell} (\mathbb{R}^4)$, as $\varepsilon
  \rightarrow 0$. 
\end{proof*}

\subsection{Analysis of the elliptic SPDE}\label{subsection_analytic}

In this section we want to prove the following theorem.

\begin{theorem}
  \label{theorem_exponentialmain}For any $| \alpha | < \alpha_{\text{max}}$ (see equation \eqref{eq:alphamax}), there are some $p, s, \delta \in \mathbb{R}$ such that $1 < p
  \leqslant 2$ and $p < \frac{2 (4 \pi)^2}{\alpha^2}$, $- 1 < s <
  \frac{\alpha^2 (p - 1)}{(4 \pi)^2}$ and $0 < \delta < s + 1$, for which
  equation~{\eqref{eq:exponential2}} admits a unique solution in $B^{s +
  2}_{p, p, \ell}$ for $\ell > 0$ large enough and such that $\alpha
  \bar{\phi} \leqslant 0$. \ Furthermore we have that there exists a
  subsequence $\varepsilon_n \rightarrow 0$ such that \
  $\bar{\phi}_{\epsilon_n} \rightarrow \bar{\phi}$ in $B^{s + 2 - \delta}_{p,
  p, \ell + \delta'}$ (for some $\delta' > 0$ small enough), as $n \rightarrow
  \infty$, almost surely.
\end{theorem}

Hereafter, for $t > 0$, we write $\mathbb{B}^t_{p, p, \ell} (\mathbb{R}^4) =
B^t_{p, p, \ell} (\mathbb{R}^4) \cap L^{\infty} (\mathbb{R}^4)$. The space
$\mathbb{B}^t_{p, p, \ell}$ has a natural norm given by the sum of the norms
of $B^t_{p, p, \ell} (\mathbb{R}^4)$ and $L^{\infty} (\mathbb{R}^4)$.
Furthermore we can equip the space $\mathbb{B}^t_{p, p, \ell}$ with a
different notion of convergence of sequences: we say that a sequence $h_n \in
\mathbb{B}^t_{p, p, \ell}$ converges to $h \in \mathbb{B}^t_{p, p, \ell}$ if
$h_n$ converges to $h$ strongly in $B^t_{p, p, \ell}$ and $\sup_n \| h_n
\|_{\infty} < + \infty$ (it is important to note that $\| h \|_{\infty}
\leqslant \liminf_{n \rightarrow + \infty} \| h_n \|_{\infty}$).

\begin{lemma}
  \label{lemma_interpolation}For any $t > 0$ and $1 < p \leqslant 2$, we have
  that $\mathbb{B}^t_{p, p, \ell} \subset B^{t (p - 1)
  \gamma}_{\frac{q}{\gamma}, \frac{q}{\gamma}, \ell}$ (where $1 / p + 1 / q =
  1$ and $0 < \gamma<1$) and the immersion is continuous with respect to the
  natural convergence in $\mathbb{B}^t_{p, p, \ell}$. Furthermore
  $\mathbb{B}^t_{p, p, \ell}$ is a Banach algebra.
\end{lemma}

\begin{proof}
  By Proposition \ref{proposition_interpolation}, the interpolation space
  $(B^0_{\infty, \infty}, B^t_{p, p, \ell})_{\frac{p \gamma}{q}, q}$ between
  $L^{\infty} \subset B^0_{\infty, \infty}$ and $B^t_{p, p, \ell}$ is exactly
  $B^{t \gamma \frac{p}{q}}_{\frac{q}{\gamma}, \frac{q}{\gamma}, \ell} = B^{t
  \gamma (p - 1)}_{\frac{q}{\gamma}, \frac{q}{\gamma}, \ell}$, where $q =
  \frac{p - 1}{p}$. This means that $\mathbb{B}^t_{p, p, \ell} = L^{\infty}
  \cap B^t_{p, p, \ell} \subset B^0_{\infty, \infty} \cap B^s_{p, p, \ell}$ is
  continuously embedded in $B^{t (p - 1) \gamma}_{\frac{q}{\gamma},
  \frac{q}{\gamma}, \ell}$. The fact that $\mathbb{B}_{p, p, \ell}^t$ is an
  algebra is proven in Corollary 2.86 of~{\cite{Bahouri2011}}.
\end{proof}

\begin{lemma}
  \label{lemma_gamma}Under the hypotheses of
  Theorem~\ref{theorem_exponentialmain} on $\alpha$ there are some $p, s,
  \delta$ such that $1 < p \leqslant 2$ and $p < \frac{2 (4
  \pi)^2}{\alpha^2}$, $- 1 < s < - \left( \frac{\alpha^2 (p - 1)}{(4 \pi)^2}
  \vee (p - 1) \right)$ and $0 < \delta < s + 1$ and for which there is a $0 <
  \gamma < 1$ such that
  \begin{eqnarray}
    \frac{(2 + s - \delta) (p - 1)}{p} + s & > & 0  \label{eq:gamma1}\\
    \frac{s + 1}{1 - \delta} < & \gamma & < 1  \label{eq:gamma2}\\
    \left( 1 - \frac{s + 1}{4} p \right) \gamma & < & p - 1. 
    \label{eq:gamma3}
  \end{eqnarray}
\end{lemma}

\begin{proof}
  Here we introduce the new variables $Z = 1 + s$ and $k = p - 1$. The
  statement of the lemma can be reformulated as follows: for any values of the
  quotient $0 < \frac{1 - Z}{k} < 2$ we can find $0 < Z < 1$ and $0 < k < 1$
  such that
  \begin{eqnarray}
    \frac{Z}{1 - \delta} & < & \frac{k}{\left( 1 - \frac{Z}{4} (k + 1)
    \right)}  \label{eq:Z1}\\
    \frac{(1 + Z - \delta) k}{k + 1} & > & 1 - Z  \label{eq:Z2}
  \end{eqnarray}
  From the inequality {\eqref{eq:Z2}} we get that
  \[ k > \frac{1 - Z}{2 Z - \delta} \]
  and inequality {\eqref{eq:Z1}} is equivalent to
  \[ k > \frac{Z (4 - Z)}{(Z^2 + 4 (1 - \delta))} . \]
  This means that the thesis of the lemma holds if and only if $k > G (Z,
  \delta)$ where
  \[ G (Z, \delta) = \max \left( \frac{Z (4 - Z)}{(Z^2 + 4 (1 - \delta))},
     \frac{1 - Z}{2 Z - \delta} \right) = \left\{ \begin{array}{l}
       \frac{1 - Z}{2 Z - \delta} \quad \tmop{for} Z \leqslant 4 - 2
       \sqrt{3}\\
       \frac{Z (4 - Z)}{(Z^2 + 4 (1 - \delta))} \quad \tmop{for} \quad Z > 4 -
       2 \sqrt{3}
     \end{array} \right. \]
  for $0 < \delta < 1$ and $\frac{\delta}{2} < Z < 1$. The bounds on the
  possible $\alpha$ can be obtained as follows
  \[ \frac{\alpha_{\max}^2}{(4 \pi)^2} = \sup_{s, p} \frac{- s}{(p - 1)} =
     \sup_{k, Z} \frac{1 - Z}{k} = \sup_{\frac{\delta}{2} < Z < 1, 0 < \delta
     < 1} \frac{1 - Z}{G (Z, \delta)} = \frac{1 - \left( 4 - 2 \sqrt{3}
     \right)}{G \left( 4 - 2 \sqrt{3}, 0 \right)} = 8 - 4 \sqrt{3} . \]
  
\end{proof}

We introduce the space $\mathfrak{M}^s_{p, p, \ell} \subset B^s_{p, p, \ell}$
(with $s$, $p$ and $\ell$ as in Theorem \ref{theorem_exponentialmain}) which
is the set of Radon $\sigma$-finite measures $\mu$ contained in $B^s_{p, p,
\ell}$ such that $| \mu | \in B^s_{p, p, \ell}$, where $| \mu | = \mu_+ +
\mu_-$, and where $\mu_+$ and $\mu_-$ are the unique positive measures such
that $\mu = \mu_+ - \mu_-$. We can define a notion of convergence on
$\mathfrak{M}^s_{p, p, \ell}$ in the following way: a sequence $\mu_n \in
\mathfrak{M}^s_{p, p, \ell}$ converges to $\mu \in \mathfrak{M}^s_{p, p,
\ell}$ if $\mu_n$ converges strongly in $B^s_{p, p, \ell}$ and $\sup_n \| |
\mu_n | \|_{B^s_{p, p, \ell}} < + \infty$.

We now consider the natural product $\cdot$ between smooth functions which
are in $\mathbb{B}^{s + 2 - \delta}_{r, r, \ell'}$ and and smooth measures in
$\mathfrak{M}^s_{p, p, \ell}$.

\begin{lemma}
  \label{lemma_multiplication}If $s, p, \delta, \gamma$ satisfy the thesis of
  Lemma~\ref{lemma_gamma}, then the product $\cdot$ can be extended in a
  unique and associative way from $\mathbb{B}^{s + 2 - \delta}_{r, r, \ell'}
  \times \mathfrak{M}^s_{p, p, \ell}$ into $\mathfrak{M}^s_{p, p, \ell}$,
  where
  \[ r = \frac{p}{\left( \left( 1 - \frac{s + 1}{4} p \right) \gamma + 1
     \right)} . \]
  This extension is (weakly) continuous with respect to the natural
  convergence in $\mathbb{B}^{s + 2 - \delta}_{r, r, \ell'}$ and
  $\mathfrak{M}^s_{p, p, \ell}$. Furthermore for any $h \in \mathbb{B}^{s + 2
  - \delta}_{r, r, \ell'}$ and $\mu \in \mathfrak{M}^s_{p, p, \ell}$ we have
  \begin{equation}
    \| | h \cdot \mu | \|_{B^s_{p, p, \ell}} \lesssim \| h \|_{\infty} \cdot
    \| | \mu | \|_{B^s_{p, p, \ell}} . \label{eq:exponentialmain1}
  \end{equation}
\end{lemma}

\begin{proof}
  First of all we note that, by Proposition \ref{proposition_product}, the
  product $\cdot$ is well defined and (strongly) continuous as bilinear
  functional from $B^{(s + 2 - \delta) \frac{(p - 1) r}{p}}_{\frac{p}{p - 1},
  \frac{p}{p - 1}, \ell'} \times B^s_{p, p, \ell}$ into $B^s_{1, 1, \ell +
  \ell'}$, where $r = \frac{p}{\left( \left( 1 - \frac{s + 1}{4} p \right)
  \gamma + 1 \right)}$ for some $\gamma$ such that $\frac{s + 1}{1 - \delta} <
  \gamma < 1$ and $\left( 1 - \frac{s + 1}{4} p \right) \gamma < p - 1$ (whose
  existence is proven in Lemma \ref{lemma_gamma}).
  
  Indeed if $\gamma$ satisfies the previous conditions, we have
  \[ (s + 2 - \delta) \frac{(p - 1) r}{p} + s > (s + 2 - \delta) \frac{(p -
     1)}{p} + s > 0 \]
  since by hypotheses $r = \frac{p}{\left( \left( 1 - \frac{s + 1}{4} p
  \right) \gamma + 1 \right)} > 1$ and $\frac{(s + 2 - \delta) (p - 1)}{p} + s
  > 0$.
  
  The only thing that remains to prove is
  inequality~{\eqref{eq:exponentialmain1}}, since the other statements of the
  lemma easily follow from it. First of all, using the equivalent norm on
  Besov spaces of Proposition \ref{proposition_equivalentnorm}, we note that
  if $\mu$ is a Radon measure we have
  \[ \| | \mu | \|_{B^s_{p, p, \ell}} \sim \left( \int_{\mathbb{R}^4 \times
     (0, 1]} \left( \frac{| \mu | (B_{\lambda} (\hat{z}))}{\lambda^{s + n}}
     \right)^p \bar{r}_{\ell} (\hat{z})^p \mathd \hat{z} \frac{\mathd
     \lambda}{\lambda} \right)^{1 / p} . \]
  Using again Proposition \ref{proposition_equivalentnorm}, if $h$ is a
  continuous bounded function we have that
  \[ \| | h \cdot \mu | \|_{B^s_{p, p, \ell}} \lesssim \left(
     \int_{\mathbb{R}^4 \times (0, 1]} \left( \frac{| h \cdot \mu |
     (B_{\lambda} (\hat{z}))}{\lambda^{s + n}} \right)^p \bar{r}_{\ell}
     (\hat{z})^p \mathd \hat{z} \frac{\mathd \lambda}{\lambda} \right)^{1 / p}
     \lesssim \| h \|_{\infty} \cdot \| | \mu | \|_{B^s_{p, p, \ell}} . \]
  If $h$ is a generic function on $\mathbb{B}^{s + 2 - \delta}_{r, r, \ell'}$
  there exists a sequence of smooth functions such that, as $n \rightarrow
  \infty$, $h_n \rightarrow h$ in $B^{s + 2 - \delta}_{r, r, \ell'}$ and $\|
  h_n \|_{\infty} \leqslant \| h \|_{\infty}$. By the first part of the lemma
  we have that $| h_n \cdot \mu |$ converges weakly in $\mathcal{S}'
  (\mathbb{R}^4)$ to $| h \cdot \mu |$ from which we obtain
  \[ \| | h \cdot \mu | \|_{B^s_{p, p, \ell}} \leqslant \liminf \| | h_n \cdot
     \mu | \|_{B^s_{p, p, \ell}} \lesssim \| | \mu | \|_{B^s_{p, p, \ell}}
     \liminf \| h_n \|_{\infty} \lesssim \| h \|_{\infty} \| | \mu |
     \|_{B^s_{p, p, \ell}} . \]
\end{proof}

We introduce the following map
\[ \mathcal{K}_g (\mu, \bar{\varphi}) : = - \alpha (- \Delta_{\hat{z}} +
   m^2)^{- 1} (g (z) G (\alpha \bar{\varphi} (\hat{z})) \cdot \mu (\hat{z})),
\]
for any $| \alpha | <\alpha_{\text{max}}$, where $G : \mathbb{R}
\rightarrow \mathbb{R}_+$ is an increasing smooth bounded function with all
bounded derivatives such that $G (x) = \exp (x)$ for $x \leqslant 0$. It is
clear that if $\alpha \bar{\phi} \leqslant 0$ and if $\bar{\phi}$ solves
equation~{\eqref{eq:exponential2}} then we have
\[ \bar{\phi} = \mathcal{K}_g (\eta, \bar{\phi}) . \]
Hereafter we denote by $\mathfrak{M}^s_{+, p, p, \ell}$ the space formed by
the positive distributions contained in $\mathfrak{M}^s_{p, p, \ell}$.

\begin{lemma}
  \label{lemma_G}Let $G$ be the function defined by $\varphi \longmapsto G
  (\varphi)$, and let $p, s, \delta, \gamma$ be as in the thesis of Lemma
  \ref{lemma_gamma} then, when $\ell > 0$ is big enough, we have that $G$ is a
  continuous map from $B^{s + 2 - \delta}_{p, p, \ell}$ into $\mathbb{B}^{s +
  2 - \delta}_{r, r, 2 \ell}$ where $r$ is chosen as in Lemma
  \ref{lemma_multiplication}.
\end{lemma}

\begin{proof}
  For any $0 \leq k < 1$, $q > 1$ and $\ell$ big enough we have that $G$ is a
  continuous (actually Lipschitz) map from $B^k_{q, q, \ell}$ into
  $\mathbb{B}^k_{q, q, \ell}$. We give a sketch of the proof for the weighted
  Besov spaces following the argument of {\cite{Runst1996}} Section 5.5.
  
  This fact can be proven using the equivalent definition of Besov norm using
  finite differences {\eqref{eq:equivalentnorm2}} for $N = 0$ and $M = 1$.
  Indeed we have that
  \begin{eqnarray}
    \| G (\varphi) \|_{B^k_{q, q, \ell}} & = & \| G (\varphi) \|_{L^p_{\ell}}
    + \left( \int_{| h | < 1} | h |^{- k q} \| G (\varphi (x + h)) - G
    (\varphi (x)) \|_{L^q_{\ell}}^q \frac{\mathd h}{| h |^n}
    \right)^{\frac{1}{q}} \nonumber\\
    & \lesssim & (\| G \|_{\infty} + \| G' \|_{\infty}) \left( 1 + \| \varphi
    \|_{L^p_{\ell}} + \left( \int_{| h | < 1} | h |^{- k} \| \Delta^1_h
    \varphi \|_{L^q_{\ell}}^q \frac{\mathd h}{| h |^n} \right)^{\frac{1}{q}}
    \right) \nonumber\\
    & \lesssim & (\| G \|_{\infty} + \| G' \|_{\infty}) (1 + \| \varphi
    \|_{B^k_{q, q, \ell}}) . \nonumber
  \end{eqnarray}
  This proves that the map $G$ takes values in $\mathbb{B}^k_{q, q, \ell}$. By
  using the real interpolation of Sobolev spaces
  \[ (L^p_{\ell}, W^{1, p}_{\ell})_{s, q} = B^s_{p, q, \ell}, \]
  which is a consequence of the interpolation of unweighted Sobolev spaces
  (see {\cite{Joran1976}} Theorem 6.2.4) and the isomorphism between weighted
  and unweighted Besov and Sobolev spaces (see {\cite{Triebel2006}} Theorem
  6.5), using the Lipschitz continuity of $G$ in $L^q_{\ell}$ and in $W^{1,
  q}_{\ell}$ (see {\cite{Bourdaud2011}} for the proof in the unweighted
  Sobolev spaces), and exploiting a standard interpolation argument (see
  {\cite{Maligranda}}), we obtain that $G$ is also continuous.
  
  Consider $\varphi \in B^{s + 2 - \delta}_{p, p, \ell}$, then $\nabla G
  (\varphi) = G' (\varphi) \nabla \varphi$. By a Besov embedding theorem we
  have that $\varphi \in B^{1 - \epsilon}_{p', p', \ell}$ where $p' = \left( 1
  - \frac{s + 1}{4} p \right)$ for $\varepsilon > \delta$ small enough. By our
  arguments above this means that $G' (\varphi) \in \mathbb{B}^{1 -
  \epsilon}_{p', p', \ell}$, and by Lemma \ref{lemma_interpolation}, $G'
  (\varphi) \in B^{(1 - \epsilon) \gamma}_{\frac{p'}{\gamma},
  \frac{p'}{\gamma}, \ell}$ This implies that, for $\gamma > \frac{s + 1}{1 -
  \varepsilon}$ and $\left( 1 - \frac{s + 1}{4} p \right) \gamma < p - 1$ (we
  must choose $\delta < \varepsilon < \frac{s + 1}{\gamma} - 1$ here), the
  product $G' (\varphi) \nabla \varphi$ is well defined and in $B^{s + 1 -
  \delta}_{r, r, \ell}$. By Remark \ref{remark_equivalentnorm} below, this
  implies that $G (\varphi) \in \mathbb{B}^{s + 2 - \delta}_{r, r, \ell}$.
  
  In a similar way, when $\varphi_n \in B^{s + 2 - \delta}_{p, p, \ell}$
  converges to $\varphi \in B^{s + 2 - \delta}_{p, p, \ell}$, it is possible
  to prove that $G (\varphi_n) \rightarrow G (\varphi)$ in $\mathbb{B}^{s + 2
  - \delta}_{r, r, \ell}$. Indeed, by Remark \ref{remark_equivalentnorm}, we
  have
  \begin{eqnarray}
    \| G (\varphi) - G (\varphi_n) \|_{B^{s + 2 - \delta}_{r, r, 2 \ell}} &
    \lesssim & \| G (\varphi) - G (\varphi_n) \|_{L^r_{2 \ell}} + \| G'
    (\varphi) \nabla \varphi - G' (\varphi_n) \nabla \varphi_n \|_{B^{s + 1 -
    \delta}_{r, r, 2 \ell}} \nonumber\\
    & \lesssim & \| G (\varphi) - G (\varphi_n) \|_{L^r_{2 \ell}} + \| G'
    (\varphi) \|_{B^{(1 - \epsilon) \gamma}_{\frac{p'}{\gamma},
    \frac{p'}{\gamma}, \ell}} \| \nabla \varphi - \nabla \varphi_n \|_{B^{s +
    1 - \delta}_{p, p, \ell}} + \nonumber\\
    &  & + \sup_n (\| \nabla \varphi_n \|_{B^{s + 1 - \delta}_{p, p, \ell}})
    \| G' (\varphi) - G' (\varphi_n) \|_{B^{(1 - \epsilon)
    \gamma}_{\frac{p'}{\gamma}, \frac{p'}{\gamma}, \ell}} \rightarrow 0
    \nonumber
  \end{eqnarray}
  since $G'$ \ is continuous from $B^{s + 2 - \delta}_{p, p, \ell}$ into
  $B^{(1 - \epsilon) \gamma}_{\frac{p'}{\gamma}, \frac{p'}{\gamma}, \ell}$ (as
  it was proven in the first part of the proof).
\end{proof}

\begin{lemma}
  \label{lemma_exponentialexistence}For $p, s, \delta, \gamma$ satisfying the
  thesis of Lemma \ref{lemma_gamma}, for any $\ell > 0$ big enough, for any
  fixed $\mu \in \mathfrak{M}^s_{+, p, p, \ell}$ and small enough $\delta' >
  0$, there exists at least a solution $\bar{\varphi} \in B^{s + 2 -
  \delta}_{p, p, \ell + \delta'}$ to the equation
  \begin{equation}
    \bar{\varphi} = \mathcal{K}_g (\mu, \bar{\varphi}) . \label{eq:K1}
  \end{equation}
\end{lemma}

\begin{proof}
  We want to use Schaefer's fixed-point theorem (see Theorem~4 Section~9.2
  Chapter 9 of~{\cite{Evans1998}}) to prove the lemma. In order to do this we
  have to prove that $\mathcal{K}_g$ is continuous in $\varphi$, that it maps
  any bounded set into a compact set and that the set of solutions to the
  equations
  \[ \bar{\varphi} = \lambda \mathcal{K}_g (\mu, \bar{\varphi}) \]
  is bounded uniformly for all $0 \leqslant \lambda \leqslant 1$.
  
  The continuity of $\mathcal{K}_g$ (in both $\mu$ and $\varphi$) is a
  consequence of Lemma~\ref{lemma_multiplication} and Lemma \ref{lemma_G}.
  Indeed in the cone $\mathfrak{M}^s_{+, p, p, \ell}$ the natural convergence
  in $\mathfrak{M}^s_{+, p, p, \ell} \subset \mathfrak{M}^s_{p, p, \ell}$ \
  coincides with the strong convergence in $B^s_{p, p, \ell}$, since if $\mu
  \in \mathfrak{M}^s_{+, p, p, \ell}$ then $| \mu | = \mu$.
  
  Furthermore, by Lemma \ref{lemma_G}, the map $\varphi \rightarrow G
  (\varphi)$ is continuous from $B^{s + 2 - \delta}_{p, p, \ell + \delta'}$
  into $\mathbb{B}^{s + 2 - \delta}_{r, r, \ell + \delta'}$ where $r =
  \frac{p}{\left( \left( 1 - \frac{s + 1}{4} p \right) \gamma + 1 \right)}$.
  Thus we can apply Lemma \ref{lemma_multiplication}, from which we obtain
  that the map $\varphi \longmapsto G (\varphi) \cdot \mu$ is continuous in
  $\varphi$. Finally the linear operator $(- \Delta_{\hat{z}} + m^2)^{- 1}$ is
  continuous from $B^s_{p, p, \ell}$ into $B^{s + 2}_{p, p, \ell}$ and thus
  compact from $B^s_{p, p, \ell}$ into $B^{s + 2 - \delta}_{p, p, \ell +
  \delta'}$ (since, by Proposition \ref{proposition_inclusion}, the immersion
  $B^{s + 2}_{p, p, \ell} \hookrightarrow B^{s + 2 - \delta}_{p, p, \ell +
  \delta'}$ is compact). Thus by Lemma~\ref{lemma_multiplication} and Lemma
  \ref{lemma_G}, the map $\mathcal{K}_g$ is weakly continuous as a map into
  $B^{s + 2}_{p, p, \ell}$ and strongly continuous as a map into $B^{s + 2 -
  \delta}_{p, p, \ell + \delta'}$.
  
  The map $\mathcal{K}_g$ is compact since the following inequality holds
  \begin{equation}
    \| \mathcal{K}_g (\mu, \varphi) \|_{B^{s + 2}_{p, p, \ell}} \lesssim \| G
    \|_{\infty} \| \mu \|_{B^s_{p, p, \ell}} \label{eq:inequalityboundeness}
  \end{equation}
  and, by Proposition \ref{proposition_inclusion}, we have the compact
  immersion $B^{s + 2}_{p, p, \ell} \hookrightarrow B^{s + 2 - \delta}_{p, p,
  \ell + \delta'}$.
  
  Finally the uniform boundedness in $\lambda$ follows from
  inequality~{\eqref{eq:inequalityboundeness}}. This proves the thesis of the
  lemma.
\end{proof}

\begin{lemma}
  \label{lemma_exponentialuniqueness}Under the hypotheses of
  Theorem~\ref{theorem_exponentialmain} the solution to
  equation~{\eqref{eq:K1}} is unique in $B^{s + 2 - \delta}_{p, p, \ell +
  \delta'}$ for $\delta, \delta' \geqslant 0$ small enough.
\end{lemma}

\begin{proof}
  Let $J : \mathbb{R} \rightarrow \mathbb{R}$ be a smooth, bounded, strictly
  increasing function such that $J (0) = 0$ and $J (- x) = - J (x)$, and let
  $\bar{\varphi}_1$ and $\bar{\varphi}_2$ be two solutions to
  equation~{\eqref{eq:K1}}. By Lemma~\ref{lemma_interpolation} and Lemma
  \ref{lemma_G}, $J (\bar{\varphi}_1 - \bar{\varphi}_2) \in B^{(s + 2 -
  \delta) \frac{p - 1}{r p}}_{q, q, \ell + \delta'}$ which implies that
  $\bar{r}_{\ell'} (\lambda \hat{z}) J (\bar{\varphi}_1 - \bar{\varphi}_2) \in
  (B^s_{p, p, \ell})^{\asterisk}$ for $\delta' \geqslant 0$ small enough,
  $\ell' > 0$ large enough and any $\lambda > 0$. This means that
  \[ \langle \bar{r}_{\ell'} (\lambda \hat{z}) J (\bar{\varphi}_1 -
     \bar{\varphi}_2), (- \Delta_{\hat{z}} + m^2) (\bar{\varphi}_1 -
     \bar{\varphi}_2 - \mathcal{K}_g (\mu, \bar{\varphi}_1) + \mathcal{K}_g
     (\mu, \bar{\varphi}_2)) \rangle = 0. \]
  We are going to see that the inequality
  \[ \langle \bar{r}_{\ell'} (\lambda \hat{z}) J (\bar{\varphi}_1 -
     \bar{\varphi}_2), (- \Delta_{\hat{z}} + m^2) (\bar{\varphi}_1 -
     \bar{\varphi}_2) \rangle \geqslant C \int \bar{r}_{\ell'} (\lambda
     \hat{z}) J (\bar{\varphi}_1 - \bar{\varphi}_2) (\bar{\varphi}_1 -
     \bar{\varphi}_2) \mathd \hat{z} \]
  holds for $\lambda > 0$ small enough and some constant $C > 0$. Indeed let
  $f_1, f_2$ be two smooth functions then, for $\lambda > 0$ small enough,
  \[ \begin{array}{l}
       \langle \bar{r}_{\ell'} (\lambda \hat{z}) J (f_1 - f_2), (-
       \Delta_{\hat{z}} + m^2) (f_1 - f_2) \rangle\\
       \qquad \qquad = \int \bar{r}_{\ell'} (\lambda \hat{z}) J (f_1 - f_2) (-
       \Delta_{\hat{z}} + m^2) (f_1 - f_2) \mathd \hat{z}\\
       \qquad \qquad = \int \bar{r}_{\ell'} (\lambda \hat{z}) J' (f_1 - f_2) |
       \nabla f_1 - \nabla f_2 |^2 \mathd \hat{z} +\\
       \qquad \hspace{4em} + \lambda \int \nabla \bar{r}_{\ell'} (\lambda
       \hat{z}) J (f_1 - f_2) \cdot (\nabla f_1 - \nabla f_2) \mathd \hat{z}\\
       \qquad \hspace{4em} + m^2 \int \bar{r}_{\ell'} (\lambda \hat{z}) J (f_1
       - f_2) (f_1 - f_2) \mathd \hat{z}\\
       \qquad \qquad \geqslant - \lambda^2 \int (\Delta \bar{r}_{\ell'}
       (\lambda \hat{z})) J^{- 1} (f_1 - f_2) \mathd \hat{z} + m^2 \int
       \bar{r}_{\ell'} (\lambda \hat{z}) J (f_1 - f_2) (f_1 - f_2) \mathd
       \hat{z}\\
       \qquad \qquad \geqslant \int \left( m^2 - \left| \frac{\lambda^2 \Delta
       \bar{r}_{\ell'}}{\bar{r}_{\ell'}} \right| \right) \rho_{\ell'} (\lambda
       \hat{z}) J (f_1 - f_2) (f_1 - f_2) \mathd \hat{z}\\
       \hspace{4em} \geqslant C \int \bar{r}_{\ell'} (\lambda \hat{z}) J (f_1
       - f_2) (f_1 - f_2) \mathd \hat{z}
     \end{array} \]
  where $J^{- 1} (t) = \int_0^t J (\tau) \mathd \tau$. For these deductions we
  used the fact that
  \[ \begin{array}{c}
       \int \nabla \bar{r}_{\ell'} (\lambda \hat{z}) J (f_1 - f_2) \cdot
       (\nabla f_1 - \nabla f_2) \mathd \hat{z} = \int \nabla \bar{r}_{\ell'}
       (\lambda \hat{z}) \nabla J^{- 1} (f_1 - f_2) \mathd \hat{z} =\\
       = - \lambda \int \Delta \bar{r}_{\ell'} (\lambda \hat{z}) J^{- 1} (f_1
       - f_2) \mathd \hat{z}
     \end{array} \]
  which is true since $J^{- 1}$ is a Lipschitz function such that $J^{- 1} (0)
  = 0$, and thus, by the first part of the proof of Lemma \ref{lemma_G}, $J^{-
  1} (f_1 - f_2) \in W^{1, p}_{\ell} (\mathbb{R}^4)$. We have, also, exploited
  the fact that $J^{- 1} (t) \leqslant J (t) t$ since $J$ is increasing and
  the fact that we can choose $\lambda$ small enough such that the last
  inequality holds. This proves that
  \begin{equation}
    \langle \bar{r}_{\ell'} (\lambda \hat{z}) J (f_1 - f_2), (-
    \Delta_{\hat{z}} + m^2) (f_1 - f_2) \rangle \geqslant C \int
    \bar{r}_{\ell'} (\lambda \hat{z}) J (f_1 - f_2) (f_1 - f_2) \mathd \hat{z}
    \label{eq:inequalityuniqueness},
  \end{equation}
  for $\lambda$ small enough, and some $C > 0$, for smooth functions $f_1, f_2
  \in B^{s + 2 - \delta}_{p, p, \ell + \delta'}$. Since the expressions
  $\langle \bar{r}_{\ell'} (\lambda \hat{z}) J (f_1 - f_2), (-
  \Delta_{\hat{z}} + m^2) (f_1 - f_2) \rangle$ and $\int \bar{r}_{\ell'}
  (\lambda \hat{z}) J (f_1 - f_2) (f_1 - f_2) \mathd \hat{z}$ are continuous
  for $f_1, f_2 \in B^{2 + 2 - \delta}_{p, p, \ell + \delta'}$ (with respect
  to the $B^{s + 2 - \delta}_{p, p, \ell + \delta}$-natural norm) we can
  extend inequality~{\eqref{eq:inequalityuniqueness}} to the case of general
  functions $f_1, f_2 \in B^{s + 2 - \delta}_{p, p, \ell + \delta'}$ and so it
  holds in particular for $f_1 = \bar{\varphi}_1$ and $f_2 = \bar{\varphi}_2$.
  Furthermore we have that $\langle \bar{r}_{\ell'} (\lambda \hat{z}) J
  (\bar{\varphi}_1 - \bar{\varphi}_2), (- \Delta_{\hat{z}} + m^2) (-
  \mathcal{K}_g (\mu, \bar{\varphi}_1) + \mathcal{K}_g (\mu, \bar{\varphi}_2))
  \rangle \geqslant 0$. Indeed by Lemma~\ref{lemma_multiplication} the product
  is associative and thus we obtain
  \begin{eqnarray}
    & \langle \bar{r}_{\ell'} (\lambda \hat{z}) J (\bar{\varphi}_1 -
    \bar{\varphi}_2), (- \Delta_{\hat{z}} + m^2) (- \mathcal{K}_g (\mu,
    \bar{\varphi}_1) + \mathcal{K}_g (\mu, \bar{\varphi}_2)) \rangle = & 
    \nonumber\\
    & = \int \bar{r}_{\ell'} (\lambda \hat{z}) g (z) (\alpha G (\alpha
    \bar{\varphi}_1) - \alpha G (\alpha \bar{\varphi}_2)) J (\bar{\varphi}_1 -
    \bar{\varphi}_2) \mathd \mu (\hat{z}) \geqslant 0, &  \nonumber
  \end{eqnarray}
  where we use that $\mathd \mu (\hat{z})$ is a positive measure and $(\alpha
  G (\alpha t_1) - \alpha G (\alpha t_2)) \cdot J (t_1 - t_2)$ is positive
  since both $G$ and $J$ are increasing functions and $J (0) = 0$. Using the
  previous inequalities we obtain that
  \[ \int \bar{r}_{\ell'} (\lambda \hat{z}) J (\bar{\varphi}_1 -
     \bar{\varphi}_2) (\bar{\varphi}_1 - \bar{\varphi}_2) \mathd \hat{z}
     \leqslant 0 \]
  which holds only if $\bar{\varphi}_1 - \bar{\varphi}_2 = 0$, since $J$ is a
  strictly increasing function. 
\end{proof}

\begin{remark}
  Combining Lemma~\ref{lemma_exponentialexistence} and
  Lemma~\ref{lemma_exponentialuniqueness} we deduce that the map $\mu
  \longmapsto \bar{\varphi}$, associating with the measure $\mu$ the unique
  solution $\bar{\varphi}$ to equation~{\eqref{eq:K1}}, is continuous with
  respect to $\mu$. Indeed suppose that $\mu_n \rightarrow \mu$ in
  $\mathfrak{M}^s_{+, p, p, \ell}$ (and so in $B^s_{p, p, \ell}$), then by
  Lemma~\ref{lemma_multiplication}, we have that $\| \bar{\varphi}_n \|_{B^{s
  + 2}_{p, p, \ell}} \lesssim \sup \| \mu_n \|_{B^s_{p, p, \ell}}$, and so
  there exists a converging subsequence $\bar{\varphi}_{k_n}$, as $n
  \rightarrow \infty$. On the other hand, since $\mathcal{K}_g$ is continuous
  in both $\mu$ and $\varphi$, we get that $\bar{\varphi}_{k_n}$ converges to
  the unique solution $\varphi$ to equation~{\eqref{eq:K1}} associated with
  $\mu .$ Since the limit does not depend on the subsequence we have that
  $\varphi_n \rightarrow \varphi$ strongly in $B^{s + 2 - \delta}_{p, p, \ell
  + \delta'}$, which proves the continuity of the solution map.
  
  This continuity of the solution map with respect to the measure $\mu$ is
  similar to the continuity result obtained for solutions of singular SPDEs
  defined by the methods of paracontrolled calculus or regularity structure
  theory. 
\end{remark}

\begin{proof*}{Proof of Theorem~\ref{theorem_exponentialmain}}
  The existence and uniqueness of the solution to
  equation~{\eqref{eq:exponential2}} are proved in
  Lemma~\ref{lemma_exponentialexistence} and
  Lemma~\ref{lemma_exponentialuniqueness} considering the equation $\bar{\phi}
  = \mathcal{K} (\eta, \bar{\phi})$. What remains to prove is the convergence
  of $\bar{\phi}_{\varepsilon_n}$ to $\bar{\phi}$, as $\varepsilon_n
  \rightarrow 0$. Let $\epsilon_n \rightarrow 0$ be a sequence of positive
  numbers such that $\eta_{\epsilon_n} \rightarrow \eta$ almost surely in
  $\mathfrak{M}^s_{+, p, p, \ell}$ and let $w \in \mathcal{W}$ be such that
  $\eta_{\epsilon_n} (w) \rightarrow \eta (w)$ in $\mathfrak{M}^s_{+, p, p,
  \ell}$, as $\varepsilon \rightarrow 0$. We note that
  \begin{equation}
    \bar{\phi}_{\epsilon_n} (w) = \mathfrak{a}^{\asterisk 2}_{\epsilon_n}
    \asterisk (\mathcal{K}_g (\eta_{\epsilon_n} (w),
    \bar{\phi}_{\varepsilon_n} (w))) . \label{eq:exponentialmain2}
  \end{equation}
  From the equality~{\eqref{eq:exponentialmain2}} and
  Lemma~\ref{lemma_multiplication} we obtain that
  \[ \| \bar{\phi}_{\epsilon_n} (w) \|_{B^{s + 2}_{p, p, \ell}} \lesssim
     \sup_{n \in \mathbb{N}} \| \eta_{\varepsilon_n} \|_{B^s_{p, p, \ell}}, \]
  uniformly in $n$. This means that that there exists a subsequence
  $\bar{\phi}_{\epsilon_{k_n}} (w)$ converging to some $\bar{\varphi}$ in
  $B^{s + 2 - \delta}_{p, p, \ell + \delta'}$. On the other hand we have that
  \[ \bar{\varphi} = \lim_{n \rightarrow + \infty}
     \bar{\phi}_{\epsilon_{k_n}} (w) = \lim_{n \rightarrow + \infty}
     \mathfrak{a}^{\asterisk 2}_{\epsilon_n} \asterisk \left( \mathcal{K}_g
     \left( \eta_{\epsilon_{k_n}} (w), \bar{\phi}_{\varepsilon_{k_n}} (w)
     \right) \right) = \mathcal{K}_g (\eta (w), \bar{\varphi}) . \]
  Since, by Lemma~\ref{lemma_exponentialuniqueness}, equation~{\eqref{eq:K1}}
  has a unique solution we have that all the subsequences
  $\bar{\phi}_{\epsilon_{k_n}} (w)$ converge to the same $\bar{\varphi}$ and
  so $\bar{\varphi} = \bar{\phi} (w)$ (which is the unique solution to
  equation~{\eqref{eq:exponential2}} evaluated in $w \in \mathcal{W}$). Since
  the previous reasoning holds for almost every $w \in \mathcal{W}$ and since
  $\eta_{\epsilon_n} \rightarrow \eta$ almost surely, we have that
  $\bar{\phi}_{\epsilon_n} \rightarrow \bar{\phi}$ in $B^{s + 2 - \delta}_{p,
  p, \ell + \delta'}$ almost surely, as $\varepsilon_n \rightarrow 0$.
\end{proof*}

\subsection{Dimensional reduction}

In this section we want to prove the reduction principle for
equation~{\eqref{eq:exponential1}}, i.e. that the random field $\phi (0, z) =
\mathcal{I} \xi (0, z) + \bar{\phi} (0, z)$ has the law $\kappa_g$ given by
expression~{\eqref{eq:exponentiallaw1}}.

Before proving the dimensional reduction for
equation~{\eqref{eq:exponential1}} we have to prove that the restriction of a
solution to a two dimensional hyperplane is a well defined operation.

\begin{lemma}
  \label{lemma_xi}There exists a version of the functional $x \longmapsto
  \mathcal{I} \xi (x, \cdot)$ which is continuous as a function from
  $\mathbb{R}^2$ into $\mathcal{C}^{0 -}_{\ell} (\mathbb{R}^2) \subset B^{0
  -}_{p, p, \ell + \delta'}$. Furthermore for any $h \in C^{\infty}_0
  (\mathbb{R}^2)$ the sequence of random variables $\int_{\mathbb{R}^2}
  \mathfrak{a}_{\epsilon} \asterisk (\mathcal{I} \xi) (z) h (z) \mathd z$
  converges to $\langle \mathcal{I} \xi (0, \cdot), h \rangle$ in $L^2 (\mathd
  \mu)$, as $\varepsilon \rightarrow 0$.
\end{lemma}

\begin{proof}
  In order to prove that $x \longmapsto \mathcal{I} \xi (x, \cdot)$ admits a
  continuous version we prove that for any smooth function $h : \mathbb{R}^2
  \rightarrow \mathbb{R}$ such that $(- \Delta_z + m^2)^{- 1 / 2 + \epsilon}
  (h) \in L^2 (\mathbb{R}^2)$, with $\varepsilon > 0$, we have
  \[ \mathbb{E} [| \langle \mathcal{I} \xi (x, \cdot) - \mathcal{I} \xi (y,
     \cdot), h \rangle |^2] \lesssim | x - y |^{\epsilon} \| (- \Delta_z +
     m^2)^{- 1 / 2 + \epsilon} (h) \|_{L^2 (\mathbb{R}^2)}, \]
  where the constants implied by $\lesssim$ do not depend on $h, x, y$. This
  result, exploiting hypercontractivity and a version of Kolmogorov continuity
  criterion for multidimensional random fields, implies the continuity of
  $\mathcal{I} \xi (x, \cdot)$ with respect to $x \in \mathbb{R}^2$. We note
  that, for any $x, y \in \mathbb{R}^2$:
  \[ \mathbb{E} [| \langle \mathcal{I} \xi (x, \cdot) - \mathcal{I} \xi (y,
     \cdot), h \rangle |^2] = 2 \int_{\mathbb{R}^4} \frac{1 - e^{i k_1 \cdot
     (x - y)}}{(k_1^2 + k_2^2 + m^2)^{1 + 2 \epsilon}} \frac{\hat{h}
     (k_2)^2}{(k_1^2 + k_2^2 + m^2)^{1 - 2 \epsilon}} \mathd k_1 \mathd k_2 .
  \]
  Now we observe that $| 1 - e^{i k_1 \cdot (x - y)} | \leqslant 2 | x - y
  |^{\frac{p - 1}{p}} | k_1 |^{\frac{p - 1}{p}}$ for any $p > 1$. If we choose
  $p = \frac{1}{1 - \epsilon}$, $0 < \epsilon < 1$, we obtain the claim.
  
  In order to prove that $\int_{\mathbb{R}^2} \mathfrak{a}_{\epsilon}
  \asterisk (\mathcal{I} \xi) (z) h (z) \mathd z$ converges, as $\varepsilon
  \rightarrow 0$, to $\langle \mathcal{I} \xi (0, \cdot), h \rangle$ in $L^2
  (\mathd \mu)$, we note that the distribution $\mathcal{I} (\delta_0 (x)
  \cdot h (z))$ belongs to $L^2 (\mathbb{R}^4)$, implying that
  \[ \int_{\mathbb{R}^2} \mathfrak{a}_{\epsilon} \asterisk (\mathcal{I} \xi)
     (z) h (z) \mathd z = \delta [\mathcal{I} (\delta_0 (x) \cdot
     \mathfrak{a}_{\epsilon} \asterisk h (z))] \qquad \langle \mathcal{I} \xi
     (0, \cdot), h \rangle = \delta [\mathcal{I} (\delta_0 (x) \cdot h (z))],
  \]
  where $\delta$ denote the Skorohod integral with respect to the white noise
  $\xi$. Since $\mathcal{I} (\delta_0 (x) \cdot \mathfrak{a}_{\epsilon}
  \asterisk h (z))$ converges to $\mathcal{I} (\delta_0 (x) \cdot h (z))$ in
  $L^2 (\mathbb{R}^4)$ also this claim follows. 
\end{proof}

\begin{lemma}
  \label{lemma_trace}The operator $T_0 : C^0_{\ell} (\mathbb{R}^4) \rightarrow
  C^0_{\ell} (\mathbb{R}^2)$ given by $T_0 (h) (\cdot) = h (0, \cdot)$, for $h
  \in C^0_{\ell} (\mathbb{R}^4)$, can be uniquely extended in a continuous way
  as an operator from $B^{s + 2 - \delta}_{p, p, \ell + \delta'}
  (\mathbb{R}^4)$ into $B^{s + 2 - 2 / p - \delta}_{p, p, \ell + \delta'}
  (\mathbb{R}^2)$, when $s, p, \delta$ satisfies the thesis of
  Lemma~\ref{lemma_gamma}, $\ell > 0$ big enough and $\delta' > 0$ small
  enough.
\end{lemma}

\begin{proof}
  If $s, p, \delta$ satisfies the thesis of Lemma~\ref{lemma_gamma} then $s +
  2 - \frac{2}{p} - \delta > 0$. Under this condition the proof can be found
  in Section~4.4.1 and Section~4.4.2 of~{\cite{Triebel1992}}, see also
  Section~18.1 of~{\cite{Triebel1997}}.
\end{proof}

\begin{theorem}
  \label{theorem_dimensionalreduction1}Let $\phi$ be the unique solution to
  equation~{\eqref{eq:exponential1}} then, if $| \alpha | < \alpha_{\text{max}}$ (see equation \eqref{eq:alphamax}), there exists $\ell > 0$ such that for any measurable and
  bounded function $F$ on $B^{0 -}_{p, p, \ell + \delta'} (\mathbb{R}^2)$, for
  some $\delta' > 0$ small enough, we have
  \begin{equation}
    \mathbb{E} [F (\phi (0, \cdot))] = \int_{B^{0 -}_{p, p, \ell + \delta'}
    (\mathbb{R}^2)} F (\omega) \mathd \kappa_g (\omega) . \label{eq:final}
  \end{equation}
\end{theorem}

\begin{proof}
  We prove the equality~{\eqref{eq:final}} for the case in which $F (\omega) =
  \tilde{F} (\langle \omega, f_1 \rangle, \ldots, \langle \omega, f_n
  \rangle)$ where $\tilde{F}$ is a bounded continuous function, and $f_1,
  \ldots, f_n \in C^{\infty}_0 (\mathbb{R}^2)$. Since the functions of the
  previous form generate all the $\sigma$-algebra of Borel measurable
  functions on $B^{0 -}_{p, p, \ell + \delta'}$, where $\delta' > 0$, proving
  the theorem for functions of the previous form is equivalent to prove the
  theorem in general.
  
  Since equation~{\eqref{eq:exponential5}} holds, we have only to prove that
  $\mathbb{E} [F (\phi_{\epsilon_n} (0, \cdot))] \rightarrow \mathbb{E} [F
  (\phi (0, \cdot))]$, as $\epsilon_n \rightarrow 0$ and where $\epsilon_n$ is
  any subsequence such that $\bar{\phi}_{\epsilon_n} \rightarrow \bar{\phi}$
  almost surely (whose existence is proved in
  Theorem~\ref{theorem_exponentialmain}) and $\int_{B^{0 -}_{p, p, \ell +
  \delta'} (\mathbb{R}^2)} F (\omega) \mathd \kappa_{\epsilon} (\omega)
  \rightarrow \int_{B^{0 -}_{p, p, \ell + \delta'} (\mathbb{R}^2)} F (\omega)
  \mathd \kappa_g (\omega)$, as $\epsilon_n \rightarrow 0$.
  
  The first convergence follows from the fact that $\phi_{\epsilon} =
  \bar{\phi}_{\epsilon} + \mathcal{\mathfrak{a_{\epsilon} (\mathcal{I}
  \xi)}}$. Indeed, by Theorem~\ref{theorem_exponentialmain}, as $\epsilon_n
  \rightarrow 0$,
  \[ \langle \bar{\phi}_{\epsilon_n} (0, \cdot), f_i \rangle = \langle
     \bar{\phi}_{\varepsilon_n}, T^{\asterisk}_0 (f_i) \rangle \rightarrow
     \langle \bar{\phi}, T^{\asterisk}_0 (f_i) \rangle = \langle \bar{\phi}
     (0, \cdot), f_i \rangle, \]
  for any $i = 1, \ldots, n$, almost surely and where the continuity of $T_0,$
  proved in Lemma~\ref{lemma_trace}, is used. On the other hand, by
  Lemma~\ref{lemma_xi}, we have $\langle \mathfrak{a}_{\epsilon} (\mathcal{I}
  \xi) (0, \cdot), f_i \rangle \rightarrow \langle \mathcal{I} \xi
  (0, \cdot), f_i \rangle$ in probability, as $\varepsilon_n \rightarrow 0$,
  and this implies that $\langle \phi_{\epsilon_n} (0, \cdot), f_i \rangle
  \rightarrow \langle \phi (0, \cdot), f_i \rangle$ in probability, for any $i
  = 1, \ldots, n$.
  
  Since the convergence in probability implies the one in distribution we get
  $\mathbb{E} [F (\phi_{\epsilon_n} (0, \cdot))] \rightarrow \mathbb{E} [F
  (\phi (0, \cdot))]$. Finally since $\kappa_{\epsilon_n}$ converges weakly to
  $\kappa_g$, as $\epsilon_n \rightarrow 0$, the thesis follows. 
\end{proof}

Now we want to discuss what happens if we remove the cut-off $g$ in
equation~{\eqref{eq:exponential1}} so we consider the equation
\begin{equation}
  (- \Delta_{\hat{z}} + m^2) (\phi) + \alpha \exp (\alpha \phi - \infty) =
  \xi, \label{eq:exponentialcutoff}
\end{equation}
for the content of this equation see the discussion ate the beginning of the
present section. In order to distinguish between the solution to
equation~{\eqref{eq:exponential1}} and equation~{\eqref{eq:exponentialcutoff}}
we denote by $\phi_g$ the solution of the former and by $\phi$ the solution of
the latter. We use also the symbols $\bar{\phi}_g = \phi_g - \mathcal{I} \xi$
and $\bar{\phi} = \phi - \mathcal{I} \xi$, with $\mathcal{I} = (- \Delta +
m^2)^{- 1}$.

\begin{proposition}
  \label{proposition_g1}For any $|\alpha|<\alpha_{\text{max}}$ there exists a unique solution $\phi$ to
  equation~{\eqref{eq:exponentialcutoff}}. \ Furthermore, for any affine
  transformation $\Phi : \mathbb{R}^4 \rightarrow \mathbb{R}^4$ in the
  Euclidean group of $\mathbb{R}^4$, the random field $\Phi_{\asterisk} (\phi)
  (\hat{z}) = \phi (\Phi (\hat{z}))$ has the same law as $\phi (\hat{z})$.
\end{proposition}

\begin{proof}
  The existence and uniqueness for the solution to
  equation~{\eqref{eq:exponentialcutoff}} can be proven as in
  Lemma~\ref{lemma_exponentialexistence} and
  Lemma~\ref{lemma_exponentialuniqueness}, since the estimates used to prove
  those lemmas do not depend on $g$.
  
  The invariance of the law of the solutions with respect to affine
  transformations in the Euclidean group of $\mathbb{R}^4$ follows from the
  invariance (in law) of the white noise $\xi$ and of the left-hand-side of
  equation~{\eqref{eq:exponentialcutoff}} with respect to translations and
  rotations and from the uniqueness of the solution to
  equation~{\eqref{eq:exponentialcutoff}}.
\end{proof}

\begin{theorem}
  \label{theorem_dimensionalreduction2}Consider $|\alpha| <\alpha_{\text{max}}$ (see equation \eqref{eq:alphamax}), if $g_n$ is an (increasing) sequence of
  cut-offs with compact support such that $g_n \uparrow 1$ then
  $\bar{\phi}_{g_n} \rightarrow \bar{\phi}$ in $B^{s + 2 - \delta}_{p, p, \ell
  + \delta'}$ (which means that $\phi_{g_n} \rightarrow \phi$ in $B^{0 -}_{p,
  p, \ell + \delta'} (\mathbb{R}^4)$). This means that the sequence of
  probability measures $\kappa_{g_n}$ on $B^{0 -}_{p, p, \ell + \delta'}
  (\mathbb{R}^2)$ converges weakly, as $n \rightarrow \infty$, to a
  probability measure $\kappa$ which is invariant with respect to the natural
  action of the Euclidean group of $\mathbb{R}^2$ on $B^{0 -}_{p, p, \ell +
  \delta'} (\mathbb{R}^2)$ and does not depend on the sequence of cut-offs
  $g_n$.
\end{theorem}

\begin{proof}
  The proof is essentially based on the fact that
  \[ \| \bar{\phi}_g \|_{B^{s + 2}_{p, p, \ell}} \lesssim \| \eta \|_{B^s_{p,
     p, \ell}} \]
  uniformly in $g$. From the previous inequality and some reasoning similar to
  the ones used in the proof of Theorem~\ref{theorem_exponentialmain} the
  convergence follows. The properties of the limit measure $\kappa$ follow
  from the same properties of the solution $\phi$ to
  equation~{\eqref{eq:exponentialcutoff}} proved in
  Proposition~\ref{proposition_g1}.
\end{proof}

\section{Elliptic quantization of the $P (\varphi)_2$
model}\label{section_power}

In this section we discuss the elliptic stochastic quantization of the $P
(\varphi)_2$ model where $P$ is a polynomial of even degree and satisfying
Hypothesis~QC. In order to avoid technical details we consider only the case
$P (\varphi) = \frac{\varphi^{2 n}}{2 n}$ for $n \in \mathbb{N}$, the general
case being then a straightforward generalization. We consider the equation
\begin{equation}
  (- \Delta_x - \Delta_z + m^2) (\phi) + f (x) \phi^{\diamond 2 n - 1} = \xi
  \label{eq:power1}
\end{equation}
where $z \in M =\mathbb{T}^2$, $\diamond$ stands for the Wick product and
$\xi$ is a $\mathbb{R}^2 \times \mathbb{T}^2$ white noise.
Equation~{\eqref{eq:power1}} can be better understood if we consider the
equation for $\bar{\phi} : = \phi - \mathcal{I} \xi$, the usual Da
Prato-Debussche trick (introduced in {\cite{da_prato_strong_2003}}), obtaining
the equation that is
\begin{equation}
  (- \Delta_x - \Delta_z + m^2) (\bar{\phi}) + \sum_{k = 0}^{2 n - 1} \binom{2
  n - 1}{k} f (x) \cdot \mathcal{I} \xi^{\diamond k} \cdot \bar{\phi}^{2 n - 1
  - k} = 0. \label{eq:power2}
\end{equation}
Equation~{\eqref{eq:power2}} is expected to be well defined, since we expect
that the solution $\bar{\phi}$ is in $H^1 (\mathbb{R}^2 \times \mathbb{T}^2) =
W^{1, 2} (\mathbb{R}^2 \times \mathbb{T}^2)$, in $L^{2 n}_{f^{1 / 2 n}}
(\mathbb{R}^2 \times \mathbb{T}^2)$ (where $L^{2 n}_{f^{1 / 2 n}}$ is the
weighted $L^{2 n}$ space with respect to the space weight $f^{\frac{1}{2 n}}$)
and in $B^{2 - \delta}_{\mathfrak{p}, \mathfrak{p}, \ell} (\mathbb{R}^2 \times
\mathbb{T}^2)$, where $\mathfrak{p =} \frac{2 n}{2 n - 1}$. Furthermore it is
well known that $\mathcal{I} \xi^{\diamond k} \in \mathcal{C}^{-
\delta}_{\ell} = B^{- \delta}_{\infty, \infty, \ell} (\mathbb{R}^2 \times
\mathbb{T}^2)$ for any $\ell, \delta > 0$ (see~{\cite{GH18}}). In general we
expect that equation~{\eqref{eq:power2}}, and so equation~{\eqref{eq:power1}},
for any realization of the noise $\xi$ admits multiple solutions. So we need a
notion of weak solution to equation~{\eqref{eq:power2}}.

First of all we consider a fixed probability space $\mathcal{W}^e =
(\mathcal{C}^{- \delta}_{\ell} (\mathbb{R}^2 \times \mathbb{T}^2))^{2 n}
\times \mathfrak{W}$, where
\[ \mathfrak{W} = H^{1 - \delta_1} (\mathbb{R}^2 \times \mathbb{T}^2) \cap
   L^{2 n - \delta_2}_{f^{1 / 2 n + \delta_3}} (\mathbb{R}^2 \times
   \mathbb{T}^2) \cap B^{2 - \delta_4}_{\mathfrak{p}, \mathfrak{p}}
   (\mathbb{R}^2 \times \mathbb{T}^2), \]
and where $\delta, \delta_1, \delta_2, \delta_3, \delta_4 > 0$ are small
enough (we give some more precise conditions in what follows, see inequalities
{\eqref{eq:gamma4}}, {\eqref{eq:beta}}, {\eqref{eq:delta1}},
{\eqref{eq:delta2}} and {\eqref{eq:delta3}}), that is the space where
$(\mathcal{I} \xi, \mathcal{I} \xi^{\diamond 2}, \ldots, \mathcal{I}
\xi^{\diamond 2 n}, \bar{\phi}) \in \mathcal{W}^e$ are defined. We consider a
distinguished subspace of $\mathfrak{W}$ namely
\[ \overline{\mathfrak{W}} = H^1 (\mathbb{R}^2 \times \mathbb{T}^2) \cap L^{2
   n}_{f^{1 / 2 n}} (\mathbb{R}^2 \times \mathbb{T}^2) \cap \mathfrak{W} . \]

\begin{remark}
  \label{remark_f} Hereafter, we will always use the following additional
  hypothesis on the cut-off $f$:
{\descriptionparagraphs{  
  \item[Hypothesis H$f$1.] The function $f$ satisfies Hypothesis H$f$
  and furthermore $f$ satisfies the following properties, for any $x, y \in
  \mathbb{R}^2$ and $\alpha \in \mathbb{N}^2$ with $| \alpha | \leq 2$
  \[ 0 < f (x)^{\pm 1} < c f (x - y)^{\pm 1} \exp (d | y |) \]
  \[ | D^{\alpha} f (x) | \leqslant e f (x), \]
  where $c, e > 0$ and $d \geqslant 0$ are some constants.
 }}
  Examples of functions $f$ satisfying Hypothesis H$f$1 can be found
  in~{\cite{Klein1984}}.
  
  \
  
  The importance of Hypothesis H$f$1 is that an equivalent norm for $B^s_{p,
  q, f^{\beta}} (\mathbb{R}^2 \times \mathbb{T}^2)$, i.e. the weighted Besov
  space with weight $(f (x))^{\beta}$ (see {\cite{Schott1,Schott2}} for the
  precise definitions), when $| s | < 2$, $- 1 \leqslant \beta \leqslant 1$
  and $p > 1$, is given by
  \[ \| u \|_{B^s_{p, q, f^{\beta}}} \sim \| f^{\beta} u \|_{B^s_{p, q}} \]
  (see {\cite{Schott2}} Theorem 4.4). This means that an analogous version of
  Proposition \ref{proposition_inclusion} and Proposition
  \ref{proposition_interpolation} holds also for the space $B^s_{p, q,
  f^{\beta}}$. \ 
\end{remark}

\begin{definition}
  \label{definition_we}A probability measure $\nu^e$ on $\mathcal{W}^e$ is a
  weak solution to equation~{\eqref{eq:power2}} if the projection of $\nu^e$
  on $[\mathcal{C}^{- \delta}_{\ell} (\mathbb{R}^2 \times \mathbb{T}^2)]^{2
  n}$, $\delta > 0$, gives the law of $(\mathcal{I} \xi, \mathcal{I}
  \xi^{\diamond 2}, \ldots, \mathcal{I} \xi^{\diamond 2 n - 1})$ (which is the
  law of a Gaussian noise with covariance $(- \Delta_x - \Delta_z + m^2)^{-
  2}$ and its Wick powers) and it is supported on the set of solutions to the
  equation
  \begin{equation}
    (- \Delta_x - \Delta_z + m^2) (\bar{\theta}) + \sum_{k = 0}^{2 n - 1}
    \binom{2 n - 1}{k} f (x) \cdot \sigma_k \cdot \bar{\theta}^{2 n - 1 - k} =
    0, \label{eq:power3}
  \end{equation}
  where $x \in \mathbb{R}^2$, $z \in \mathbb{T}^2$, $(\sigma_1, \ldots,
  \sigma_{2 n}, \bar{\theta}) \in \mathcal{W}^e$ and $\sigma_0 = 1$, defined
  on $\mathcal{W}^e$.
  
  The weak solution $\nu$ to equation~{\eqref{eq:power1}} associated with the
  weak solution $\nu^e$ to equation~{\eqref{eq:power2}}, is the probability
  law on $\mathcal{C}^{- \delta}_{\ell} (\mathbb{R}^2 \times \mathbb{T}^2) +
  \mathfrak{W}$, $\delta > 0$, given by the push-forward of $\nu^e$ with
  respect to the map $\sigma_1 + \theta$ defined on $\mathcal{W}^e$.
\end{definition}

We introduce a modified equation on $\mathfrak{W}$ given by
\begin{equation}
  (- \Delta_x - \Delta_z + m^2) (\bar{\theta}) + \mathfrak{a}^{\asterisk
  2}_{\epsilon} \asterisk \left[ \sum_{k = 0}^{2 n - 1} \binom{2 n - 1}{k} f
  (x) \cdot \sigma_k \cdot \bar{\theta}^{2 n - 1 - k} \right] = 0
  \label{eq:power4}
\end{equation}
where $\mathfrak{a}_{\epsilon}$ is some regular enough mollifier on
$\mathbb{T}^2$ such that the operator $\mathcal{A}_{\epsilon} =
\mathfrak{a}_{\epsilon} \asterisk$ satisfies Hypothesis~H$\mathcal{A}$ (for
example we can take $\mathfrak{a}_{\epsilon}$ as the Green function associated
with the operator $(- \epsilon \Delta_z + 1)^{- k}$ for $k$ large enough),
$\varepsilon > 0$. Also in this case we consider a special subspace of
$\mathfrak{W}$ given by
\[ \overline{\mathfrak{W}}_{\epsilon} = \mathcal{A}_{\epsilon} (H^1
   (\mathbb{R}^2 \times \mathbb{T}^2)) \cap L^{2 n}_{f^{1 / 2 n}}
   (\mathbb{R}^2 \times \mathbb{T}^2) \cap \mathfrak{W} . \]
It is important to note that $\overline{\mathfrak{W}}_{\epsilon} \subset
\overline{\mathfrak{W}}$.

\begin{lemma}
  \label{lemma_power1}Let $\bar{\theta} \in \overline{\mathfrak{W}}_{\epsilon}
  \subset \overline{\mathfrak{W}}$ be a solution to
  equation~{\eqref{eq:power4}} then $\bar{\theta} \in B^{2 -
  \delta''}_{\mathfrak{p}, \mathfrak{p}} (\mathbb{R}^2 \times \mathbb{T}^2)$
  (with $0 < \delta < \delta'' < \delta_4$ being $\delta_4$ as in the
  definition of $\mathcal{W}^e$) and
  \begin{eqnarray}
    \| \bar{\theta} \|_{H^1}^2 + \| \bar{\theta} \|_{L^{2 n}_{f^{1 / 2 n}}}^{2
    n} & \lesssim & \left[ \sum_{k = 1}^{2 n} \| \sigma_k \|_{\mathcal{C}^{-
    \delta}_{\ell}} \right]^{\beta_1}  \label{eq:P1}\\
    \| \bar{\theta} \|_{B^{2 - \delta''}_{\mathfrak{p}, \mathfrak{p}, f^{-
    \beta}}} \lesssim \| \mathcal{A}_{\epsilon}^{- 2} \bar{\theta} \|_{B^{2 -
    \delta''}_{\mathfrak{p}, \mathfrak{p}, f^{- \beta}}} & \lesssim & \left[
    \| \bar{\theta} \|_{H^1} + \| \bar{\theta} \|_{L^{2 n}_{f^{1 / 2 n}}} +
    \sum_{k = 1}^{2 n} \| \sigma_k \|_{\mathcal{C}^{- \delta}_{\ell}}
    \right]^{\beta_2}  \label{eq:P2}
  \end{eqnarray}
  for any $0 \leqslant \beta < \frac{1}{2 n} - (2 n - 1) \delta_3$, and where
  $\beta_1, \beta_2 \in \mathbb{R}_+$ and where $\beta_1, \beta_2$ and the
  constants implied by the symbol $\lesssim$ do not depend on $\epsilon$,
  $\sigma_k$ and $\beta$.
\end{lemma}

\begin{proof}
  Let $0 < \gamma < 1$ be such that
  \begin{equation}
    (2 n - 1) \left( \frac{\gamma}{2 n} + \frac{1 - \gamma}{2} \right) < 1.
    \label{eq:gamma4}
  \end{equation}
  Then by Remark \ref{remark_f}, Proposition \ref{proposition_inclusion} and
  Proposition \ref{proposition_interpolation} it is simple to see that
  $\bar{\theta} \in B^{(1 - \gamma)}_{p, 2, f^{1 / 2 n}} (\mathbb{R}^2 \times
  \mathbb{T}^2)$, where $p = \left( \frac{\gamma}{2 n} + \frac{1 - \gamma}{2}
  \right)^{- 1}$, and
  \[ \| \bar{\theta} \|_{B^{(1 - \gamma)}_{p, 2, f^{1 / 2 n}}} \lesssim \|
     \bar{\theta} \|_{L^{2 n}_{f^{1 / 2 n}}}^{\gamma} \cdot \| \bar{\theta}
     \|^{1 - \gamma}_{H^1} . \]
  Using the fact that $\bar{\theta} \in \mathcal{A}_{\epsilon} (H^1)$ we can
  multiply both sides of equation~{\eqref{eq:power4}} by
  $\mathcal{A}_{\epsilon}^{- 2} (\bar{\theta})$ and take the integral over
  $\mathbb{R}^2 \times \mathbb{T}^2$ obtaining
  \begin{eqnarray}
    \| \mathcal{A}_{\epsilon}^{- 1} (\bar{\theta}) \|_{H^1}^2 & \leqslant & -
    \| \bar{\theta} \|_{L^{2 n}_{f^{1 / 2 n}}}^{2 n} - \sum_{k = 1}^{2 n - 1}
    \binom{2 n - 1}{k} \int f (x) \cdot \sigma_k (x, z) \cdot \bar{\theta}^{2
    n - k} (x, z) \mathd x \mathd z \nonumber\\
    & \leqslant & - \| \bar{\theta} \|_{L^{2 n}_{f^{1 / 2 n}}}^{2 n} + C
    \sum_{k = 1}^{2 n - 1} \| f^{1 / 2 n} \sigma_k \|_{\mathcal{C}^{- \delta}}
    \left[ \| \bar{\theta} \|_{L^{2 n}_{f^{1 / 2 n}}}^{\gamma} \| \bar{\theta}
    \|^{1 - \gamma}_{H^1} \right]^{2 n - k} \nonumber
  \end{eqnarray}
  where we choose $\delta < 1 - \gamma$ and for some constant $C > 0$
  depending only on $n$. Using the condition {\eqref{eq:gamma4}}, the fact
  that $\| \bar{\theta} \|_{H^1} \lesssim \| \mathcal{A}^{- 1}_{\varepsilon}
  (\bar{\theta}) \|_{H^1}$ and Young's inequality for products we obtain
  \begin{equation}
    \| \bar{\theta} \|_{H^1}^2 + \| \bar{\theta} \|_{L^{2 n}_{f^{1
    / 2 n}}}^{2 n} \lesssim \sum_{k = 1}^{2 n - 1} \| f^{1 / 2 n} \sigma_k
    \|_{\mathcal{C}^{- \delta}}^{1 / \epsilon} . \label{eq:p1}
  \end{equation}
  for $\epsilon = \left( 1 - (2 n - 1) \left( \frac{\gamma}{2 n} + \frac{1 -
  \gamma}{2} \right) \right)$. Let
  \begin{equation}
    0 < \beta < \frac{1}{2 n} \quad \tmop{and} \quad \delta < \delta'' < (1 -
    \gamma), \label{eq:beta}
  \end{equation}
  we have
  \begin{equation}
    \begin{array}{lll}
      \| (\mathcal{A}_{\epsilon}^{- 2} \bar{\theta}) f^{- \beta} \|_{B^{2 -
      \delta''}_{\mathfrak{p}, \mathfrak{p}}} & \lesssim & \left\| (f (x))^{-
      \beta} \left[ \sum_{k = 0}^{2 n - 1} \binom{2 n - 1}{k} f (x) \cdot
      \sigma_k \cdot \bar{\theta}^{2 n - 1 - k} \right] \right\|_{B^{-
      \delta''}_{\mathfrak{p, \mathfrak{p}}}}\\
      & \lesssim & \sum_{k = 0}^{2 n - 1} \| (f (x))^{1 - \beta} \cdot
      \sigma_k \cdot \bar{\theta}^{2 n - 1 - k} \|_{B^{-
      \delta''}_{\mathfrak{p, \mathfrak{p}}}}\\
      & \lesssim & \sum_{k = 0}^{2 n - 1} \| f^{1 / 2 n - \beta} \cdot
      \sigma_k \|_{\mathcal{C}^{- \delta}} \times\\
      &  & \times \left[ \| \bar{\theta} \|_{L^{2 n}_{f^{1 / 2 n}}}^{\gamma}
      \cdot \| \bar{\theta} \|^{1 - \gamma}_{H^1} \right]^{2 n - 1 - k}\\
      & \lesssim & \left( \| \bar{\theta} \|_{H^1} + \| \bar{\theta} \|_{L^{2
      n}_{f^{1 / 2 n}}} + \sum_{k = 1}^{2 n - 1} \| f^{1 / 2 n - \beta}
      \sigma_k \|_{\mathcal{C}^{- \delta}} \right)^{\beta_2}
    \end{array} \label{eq:p2}
  \end{equation}
  where $\delta > 0$ such that $\delta > \delta''$. Inserting now inequality
  {\eqref{eq:p1}} into inequality {\eqref{eq:p2}}, and using the fact that $\|
  \bar{\theta} \|_{B^{2 - \delta''}_{\mathfrak{p}, \mathfrak{p}, f^{- \beta}}}
  \lesssim \| (\mathcal{A}_{\epsilon}^{- 2} \bar{\theta}) f^{- \beta} \|_{B^{2
  - \delta''}_{\mathfrak{p}, \mathfrak{p}}}$ we obtain the thesis.
\end{proof}

An easy consequence of the previous lemma is the following one.

\begin{lemma}
  \label{lemma_powercompact}Let $\mathcal{F}_{\epsilon} : (\mathcal{C}^{-
  \delta}_{\ell})^{2 n - 1} \rightarrow \mathcal{P}
  (\mathfrak{\bar{W}_{\varepsilon}})$, $\varepsilon > 0$, be the set-valued
  map associating $(\sigma_1, \ldots, \sigma_{2 n - 1})$ with the set of
  solutions to equation~{\eqref{eq:power4}} in
  $\overline{\mathfrak{W}}_{\varepsilon}$. If $K \subset (\mathcal{C}^{-
  \delta}_{\ell})^{2 n - 1}$ is a compact set then
  $\overline{\bigcup_{\epsilon < 1} \mathcal{F}_{\epsilon} (K)}$ is a subset
  of $\mathfrak{\bar{W}}$ and it is compact in $\mathfrak{W}$.
\end{lemma}

\begin{proof}
  The proof of the lemma consists only in noticing that a consequence of
  Proposition \ref{proposition_inclusion} is that the inclusion map $i : B^{2
  - \delta''}_{\mathfrak{p}, \mathfrak{p,} f^{- \beta}} \cap L^{2 n}_{f^{1 / 2
  n}} \rightarrow \mathfrak{W}$ is compact, when $\beta$ satisfies condition
  {\eqref{eq:beta}} and it is large enough and $\delta'' > \delta > 0$ is
  small enough. Indeed using Proposition \ref{proposition_inclusion} and
  Proposition \ref{proposition_interpolation}, it is simple to prove that
  $B^{2 - \delta''}_{\mathfrak{p}, \mathfrak{p,} f^{- \beta}}$ is compactly
  imbedded in $B^{2 - \delta_4}_{\mathfrak{p}, \mathfrak{p}}$ when
  \begin{equation}
    \delta < \delta'' < \delta_4 \label{eq:delta1}
  \end{equation}
  since $f^{- \beta} \rightarrow + \infty$ as $x \rightarrow + \infty$.
  
  Recalling that, by Remark \ref{remark_weight} and Proposition
  \ref{proposition_inclusion}, $L^p_{f^{1 / 2 n}} \subset B^0_{p, p, f^{1 / 2
  n}}$ when $p \geqslant 2$, we obtain, using Proposition
  \ref{proposition_interpolation}, that $B^{2 - \delta''}_{\mathfrak{p},
  \mathfrak{p,} f^{- \beta}} \cap L^{2 n}_{f^{1 / 2 n}}$ is compactly embedded
  in $B^{(2 - \delta'') \theta - \epsilon}_{p_{\theta} - \epsilon, p_{\theta}
  - \epsilon, f^{\beta_{\theta} + \epsilon}}$ where
  \begin{eqnarray}
    p_{\theta} & = & \frac{2 n}{(2 n - 2) \theta + 1} \nonumber\\
    \beta_{\theta} & = & - \theta \left( \frac{1}{2 n} + \beta \right) +
    \frac{1}{2 n} \nonumber
  \end{eqnarray}
  for any $0 \leqslant \theta \leqslant 1$. Taking $0 < \theta <
  \frac{\delta''}{(2 n - \delta'') (2 n - 2)}$, $\beta = 0$ and $\epsilon$
  small enough we have that $B^{2 - \delta''}_{\mathfrak{p}, \mathfrak{p,}
  f^{- \beta}} \cap L^{2 n}_{f^{1 / 2 n}}$ is compactly embedded in $L^{2 n -
  \delta_2}_{f^{1 / 2 n + \delta_3}}$ for any $\delta_2 < 1$ and $\delta_3 >
  0$.
  
  Furthermore if we take
  \[ \frac{1 - \delta_1}{2 - \delta''} < \theta < \frac{n - 1}{2 n - 2} \]
  \[ \left( \frac{2 - \delta''}{1 - \delta_1} \right) \frac{1}{2 n} -
     \frac{1}{2 n} < \beta < \frac{1}{2 n}, \]
  which are a non-void conditions whenever
  \begin{equation}
    \delta < \delta'' < 1 \label{eq:delta2}
  \end{equation}
  \begin{equation}
    0 < \delta < \delta'' < \delta_1 \frac{2 n - 2}{n - 1} \label{eq:delta3},
  \end{equation}
  we obtain that $B^{(2 - \delta'') \theta - \epsilon}_{p_{\theta} - \epsilon,
  p_{\theta} - \epsilon, f^{\beta_{\theta} + \epsilon}}$ compactly embeds in
  $H^{1 - \delta_1}$ for $\epsilon$ small enough.
  
  \
  
  Using inequalities {\eqref{eq:P1}} and fact that the balls of $(B^{2 -
  \delta''}_{\mathfrak{p}, \mathfrak{p,} f^{- \beta}} \cap L^{2 n}_{f^{1 / 2
  n}})$ are closed sets of $\mathfrak{W}$, since the embedding map $i : B^{2 -
  \delta''}_{\mathfrak{p}, \mathfrak{p,} f^{- \beta}} \cap L^{2 n}_{f^{1 / 2
  n}} \rightarrow \mathfrak{W}$ is injective and continuous,
  $\overline{\bigcup_{\epsilon < 1} \mathcal{F}_{\epsilon} (K)}$ is contained
  in $\overline{\mathfrak{W}}$.
\end{proof}

Hereafter we denote by $\mathfrak{P :} \mathcal{W}^e \rightarrow
(\mathcal{C}^{- \delta}_{\ell} (\mathbb{R}^2 \times \mathbb{T}^2))^{2 n}$ the
natural projection of the Cartesian product $\mathcal{W}^e = (\mathcal{C}^{-
\delta}_{\ell} (\mathbb{R}^2 \times \mathbb{T}^2))^{2 n} \times \mathfrak{W}$
on the first component.

\begin{lemma}
  \label{lemma_tightness}Let $\nu_{\epsilon}^e$ be a sequence of weak
  solutions to equation~{\eqref{eq:power4}} such that $\nu^e_{\epsilon}
  ((\mathcal{C}^{- \delta}_{\ell} (\mathbb{R}^2 \times \mathbb{T}^2))^{2 n}
  \times \mathfrak{\bar{W}}_{\epsilon}) = 1$, and suppose that
  $\mathfrak{P}_{\asterisk} (\nu_{\epsilon}^e)$ is tight. Then
  $\nu_{\epsilon}^e$ is tight. Furthermore suppose that
  $\mathfrak{P}_{\asterisk} (\nu_{\epsilon}^e)$ weakly converges to the law of
  $(\mathcal{I} \xi, \mathcal{I} \xi^{\diamond 2}, \ldots, \mathcal{I}
  \xi^{\diamond 2 n - 1})$ as $\epsilon \rightarrow 0$, then any convergent
  subsequence $\nu_{\epsilon_n}^e$ converges, as $\varepsilon_n \rightarrow
  0$, to a weak solution $\nu^e$ to equation~{\eqref{eq:power2}} such that
  $\nu^e ((\mathcal{C}^{- \delta}_{\ell} (\mathbb{R}^2 \times
  \mathbb{T}^2))^{2 n} \times \overline{\mathfrak{W}}) = 1$.
\end{lemma}

\begin{proof}
  The proof of the tightness of $\nu_{\epsilon}^e$ is similar to the one of
  Lemma~\ref{lemma_reduced3}, where Lemma~\ref{lemma_reduced2} is replaced by
  Lemma~\ref{lemma_powercompact}. The hypothesis that $\nu^e_{\epsilon}
  ((\mathcal{C}^{- \delta}_{\ell} (\mathbb{R}^2 \times \mathbb{T}^2))^{2 n}
  \times \mathfrak{\bar{W}}_{\epsilon}) = 1$ enters in the proof in the
  following way: let $\widetilde{\mathcal{F}}_{\epsilon} : (\mathcal{C}^{-
  \delta}_{\ell})^{2 n - 1} \rightarrow \mathcal{P} (\mathfrak{W})$,
  $\varepsilon > 0$, be the set-valued map associating $(\sigma_1, \ldots,
  \sigma_{2 n - 1})$ with the set of solutions to equation~{\eqref{eq:power4}}
  in $\mathfrak{W}$, let $\tilde{K} \subset (\mathcal{C}^{- \delta}_{\ell})^{2
  n}$ be a compact such that $\mu_{\varepsilon} (\tilde{K}) > 1 -
  \tilde{\epsilon}$, where $\mu_{\varepsilon}$ are the probability
  distributions of $(\mathcal{I} \xi_{\varepsilon}, \mathcal{I}
  \xi_{\varepsilon}^{\diamond 2}, \ldots, \mathcal{I}
  \xi_{\varepsilon}^{\diamond 2 n})$ which are tight since $\mu_{\varepsilon}$
  converges to $\mu$ the probability distribution of $(\mathcal{I} \xi,
  \mathcal{I} \xi^{\diamond 2}, \ldots, \mathcal{I} \xi^{\diamond 2 n})$, as
  $\varepsilon \rightarrow 0$. Furthermore denote by $\mathfrak{K} \assign
  \overline{\cup_{\epsilon} \mathcal{F}_{\epsilon} (\tilde{K})}$, which is
  compact by Lemma~\ref{lemma_powercompact}, then
  \[ \nu_{\varepsilon}^e (\tilde{K} \times \mathfrak{K}) \geqslant
     \nu_{\varepsilon}^e (\tilde{K} \times \cup_{\epsilon}
     \mathcal{F}_{\epsilon} (\tilde{K})) \geqslant \nu_{\varepsilon}^e
     (\tilde{K} \times \mathcal{F}_{\epsilon} (\tilde{K})) =
     \nu_{\varepsilon}^e (\tilde{K} \times \widetilde{\mathcal{F}_{\epsilon}}
     (\tilde{K})) = \mu_{\varepsilon} (\tilde{K}) \geqslant 1 -
     \tilde{\epsilon}, \]
  where in the fourth passage we used that $\nu_{\varepsilon} (\tilde{K}
  \times \mathcal{F}_{\epsilon} (\tilde{K})) = \nu_{\varepsilon} (\tilde{K}
  \times \widetilde{\mathcal{F}_{\epsilon}} (\tilde{K}))$ since
  $\nu^e_{\epsilon} ((\mathcal{C}^{- \delta}_{\ell} (\mathbb{R}^2 \times
  \mathbb{T}^2))^{2 n} \times \mathfrak{\bar{W}}_{\epsilon}) = 1$.
  
  Since $\mathfrak{P}_{\asterisk} (\nu_{\epsilon}^e)$ weakly converges, as
  $\varepsilon \rightarrow 0$, to the law of $(\mathcal{I} \xi, \mathcal{I}
  \xi^{\diamond 2}, \ldots, \mathcal{I} \xi^{\diamond 2 n - 1})$ and so
  $\mathfrak{P}_{\asterisk} (\nu^e)$ has the same law as $(\mathcal{I} \xi,
  \mathcal{I} \xi^{\diamond 2}, \ldots, \mathcal{I} \xi^{\diamond 2 n - 1})$,
  the proof of the fact that $\nu^e$ is a weak solution to
  equation~{\eqref{eq:power2}} is equivalent to prove that $\nu^e$ is such
  that for any $C^1$ bounded function $F$ \ from $B^{-
  \delta_4}_{\mathfrak{p}, \mathfrak{p}}$ into $\mathbb{R}$ with bounded
  derivative we have
  \begin{equation}
    \int F \left( (- \Delta_x - \Delta_z + m^2) (\bar{\theta}) + 
    \sum_{k = 0}^{2 n - 1} \binom{2 n - 1}{k} f (x) \sigma_k \bar{\theta}^{2 n
    - 1 - k} \right) \mathd \nu^e = \int F (E) \mathd \nu^e = F (0) .
    \label{eq:weak1}
  \end{equation}
  We know that $\int F (E_{\epsilon}) \mathd \nu_{\epsilon}^e = F (0)$ with
  \[ E_{\epsilon} = (- \Delta_x - \Delta_z + m^2) (\bar{\theta}) +
     \mathfrak{a}^{\asterisk 2}_{\epsilon} \asterisk \left( \sum_{k = 0}^{2 n
     - 1} \binom{2 n - 1}{k} f (x) \cdot \sigma_k \cdot \bar{\theta}^{2 n - 1
     - k} \right) . \]
  On the other hand $E_{\epsilon}$ converges to $E$, as $\varepsilon
  \rightarrow 0$, uniformly on compact sets (since
  $\mathfrak{a}_{\epsilon}^{\asterisk 2} \asterisk$ strongly converges as an
  operator to the identity on $B^{- \delta_4}_{\mathfrak{p}, \mathfrak{p}}$),
  which implies that $F \circ E_{\epsilon}$ converges to $F \circ E$ uniformly
  on compact sets, as $\varepsilon \rightarrow 0$.
  
  \
  
  Finally since the sets of the form $\overline{\bigcup_{\epsilon < 1}
  \mathcal{F}_{\epsilon} (K)}$ are subsets of $\mathfrak{\bar{W}}$, they are
  closed and they have an arbitrary measure we have that any limit $\nu^e$ is
  such that $\nu^e ((\mathcal{C}^{- \delta}_{\ell} (\mathbb{R}^2 \times
  \mathbb{T}^2))^{2 n} \times \overline{\mathfrak{W}}) = 1$.
\end{proof}

This fact and the boundedness of $F$ and of its derivatives implies (using a
reasoning similar to the one of Lemma~\ref{lemma_reduced4}, see also Lemma~2
in~{\cite{Albeverio2018elliptic}}) that $F \circ E_{\epsilon} \mathd
\nu_{\epsilon}^e$ weakly converges to $F \circ E \mathd \nu^e$, as
$\varepsilon \rightarrow 0$. This implies equation~{\eqref{eq:weak1}} and thus
the thesis of the lemma.

\begin{theorem}
  \label{theorem_polynomialreduction}There exists at least one weak solution
  $\nu^e$ to equation~{\eqref{eq:power2}} such that, for any continuous
  bounded function $F : B^{- \delta}_{\mathfrak{p}, \mathfrak{p}, \ell}
  (\mathbb{T}^2) \rightarrow \mathbb{R}$ (where $\delta > 0$ and $\ell > 0$ as
  in the definition of $\mathcal{W}^e$ and \ $\mathfrak{p} = \frac{2 n}{2 n -
  1}$),
  \begin{equation}
    \int_{\mathcal{W}^e} F (\phi (0, \cdot)) \tilde{\Upsilon}_f (\sigma,
    \bar{\theta}) \mathd \nu^e (\sigma, \bar{\theta}) = Z_f \int_{B^{-
    \delta}_{\mathfrak{p}, \mathfrak{p}, \ell} (\mathbb{T}^2)} F (\omega)
    \mathd \kappa (\omega) \label{eq:powermain}
  \end{equation}
  where $\phi = \sigma_1 + \bar{\theta}$ (where $\sigma_1$ and $\bar{\theta}$
  are as in Definition \ref{definition_we}),
  \[ \tilde{\Upsilon}_f (\sigma, \bar{\theta}) = \exp \left( \sum_{k = 0}^{2
     n} \binom{2 n}{k} \langle \sigma_k \bar{\theta}^{2 m - k}, f' \rangle
     \right), \]
  \[ Z_f = \int_{\mathcal{W}^e} \tilde{\Upsilon}_f (\sigma, \bar{\theta})
     \mathd \nu^e (\sigma, \bar{\theta}) \]
  and $\kappa$ is given by
  \[ \frac{\mathd \kappa}{\mathd \mu^{- \Delta_z}} = \frac{\exp \left(
     \int_{\mathbb{T}^2} \frac{\omega^{\diamond 2 m} (z)}{2 m} \mathd z
     \right)}{Z_{\kappa}}, \]
  where $\mu^{- \Delta_z}$ is the Gaussian with covariance given
  by~{\eqref{eq:covariance1}} with $\mathfrak{L} = - \Delta_z$ and
  $\mathcal{A} = \tmop{id}$.
\end{theorem}

\begin{lemma}
  \label{lemma_uniform}Let $\nu^e_{\epsilon}$ be a sequence of weak solutions
  to equation~{\eqref{eq:power4}} such that $\mathfrak{P}_{\asterisk}
  (\nu_{\epsilon}^e) \sim (\mathcal{\mathfrak{a}}_{\epsilon} \asterisk
  \mathcal{I} \xi, \ldots, (\mathcal{\mathfrak{a}}_{\epsilon} \asterisk
  \mathcal{I} \xi)^{\diamond 2 n - 1})$ (where $\sim$ means ``with the same
  law'') and $\nu^e_{\epsilon} ((\mathcal{C}^{- \delta}_{\ell} (\mathbb{R}^2
  \times \mathbb{T}^2))^{2 n} \times \mathfrak{\bar{W}}_{\epsilon}) = 1$ then
  \[ \int \exp \left( p \sum_{k = 0}^{2 n} \binom{2 n}{k} \langle \sigma_k,
     f'  \bar{\theta}^{2 n - k} \rangle \right) \mathd \nu_{\epsilon}^e < C_p
  \]
  form some constant $C_p$ uniform in $\epsilon$ small enough and depending on
  $p \geq 1$.
\end{lemma}

\begin{proof*}{Proof}
  We want to use Nelson's trick to prove the theorem (see Chapter~V
  of~{\cite{Simon1974}}). Then we put $E (\sigma_k, \bar{\theta}) = \sum_{k =
  0}^{2 n} \binom{2 n}{k} \langle \sigma_k, f'  \bar{\theta}^{2 n - k}
  \rangle$ and introduce the following expression
  \[ E_{\epsilon, N} (\sigma_1, \bar{\theta}) = \left.
     \left\{\begin{array}{lll}
       \int_{\mathbb{R}^2 \times \mathbb{T}^2} \sum_{k = 0}^{2 n} \binom{2
       n}{k} \left( \mathfrak{a}_{\frac{1}{N} - \epsilon} (\sigma_1) (\hat{z})
       \right)^{\diamond k} f' (x)  \bar{\theta}^{2 n - k} (\hat{z}) \mathd
       \hat{z}, & \quad & \text{for $\epsilon < \frac{1}{N}$,}\\
       \int_{\mathbb{R}^2 \times \mathbb{T}^2} \sum_{k = 0}^{2 n} \binom{2
       n}{k} (\sigma_1 (\hat{z}))^{\diamond k} f' (x)  \bar{\theta}^{2 n - k}
       (\hat{z}) \mathd \hat{z}, &  & \text{for $\epsilon \geqslant
       \frac{1}{N}$.}
     \end{array}\right. \right. \]
  We have to prove that, for any $p \geq 1$,
  \begin{equation}
    \left( \int | E - E_{\epsilon, N} |^p \mathd \nu_{\epsilon}^e \right)^{1 /
    p} \lesssim (p - 1)^u N^{- \alpha} \label{eq:IP1}
  \end{equation}
  for some $u \in \mathbb{N}$ and $\alpha > 0$ and
  \begin{equation}
    E_{\epsilon, N} \lesssim - f' (x) (\log (N))^{\alpha'} \label{eq:IP2},
  \end{equation}
  for some $\alpha' > 0$, and that both inequalities are uniform in
  $\epsilon$. Inequality~{\eqref{eq:IP1}} is obvious when $\epsilon \geqslant
  \frac{1}{N}$, since then $E - E_{\epsilon, N} = 0$ $\nu^e_{\epsilon}$-almost
  surely. Consider the case $\epsilon < \frac{1}{N}$ then
  \begin{eqnarray}
    | E - E_{\epsilon, N} | & \lesssim & \left| \sum_{k = 0}^{2 n}
    \int_{\mathbb{R}^2 \times \mathbb{T}^2} f' (x) \left( \sigma_k (z) -
    \left( \mathfrak{a}_{\frac{1}{N} - \epsilon} (\sigma_1) (\hat{z})
    \right)^{\diamond k} \right) \cdot \bar{\theta}^{2 n - k} (\hat{z}) \mathd
    \hat{z} \right| \nonumber\\
    & \lesssim & \sum_{k = 1}^{2 n} \left\| | f' (x) |^{1 / 2 n} \left(
    \sigma_k (z) - \left( \mathfrak{a}_{\frac{1}{N} - \epsilon} (\sigma_1)
    (\hat{z}) \right)^{\diamond k} \right) \right\|_{\mathcal{C}^{- \delta}}
    \cdot \| | f' (x) |^{1 - 1 / 2 n} \bar{\theta}
    \|_{B^{\delta}_{\mathfrak{p, \mathfrak{p}}}}^{2 n - k} \nonumber
  \end{eqnarray}
  where $\mathfrak{p} = \frac{2 n - 1}{2 n}$. On the other hand, by
  Lemma~\ref{lemma_power1}, we have:
  \[ \| | f' (x) |^{1 - 1 / 2 n} \bar{\theta} \|_{B^{\delta}_{\mathfrak{p,
     \mathfrak{p}}}}^{2 n - k} \lesssim \left( \| \bar{\theta} \|_{L^{2
     n}_{f^{1 / 2 n}}} + \| \bar{\theta} \|_{H^1} \right)^{2 n - k} \lesssim
     \left( \sum_{k = 1}^{2 n} \| \sigma_k \|_{\mathcal{C}^{- \delta}_{\ell}}
     \right)^{\beta_1 (2 n - 1)}, \]
  $\nu^e_{\epsilon}$-almost surely and uniformly in $\epsilon$ (since
  $\nu^e_{\epsilon} ((\mathcal{C}^{- \delta}_{\ell} (\mathbb{R}^2 \times
  \mathbb{T}^2))^{2 n} \times \mathfrak{\bar{W}}_{\epsilon}) = 1$). Using
  H{\"o}lder's inequality we obtain, for any $p \geq 1$,
  \begin{eqnarray}
    \left( \int | E - E_{\epsilon, N} |^p \mathd \nu_{\epsilon}^e
    \right)^{\frac{1}{p}} & \lesssim & \sum_{k = 1}^{2 n} \left[ \int \left\|
    | f' (x) |^{1 / 2 n} \left[ \sigma_k (z) - \left(
    \mathfrak{a}_{\frac{1}{N} - \epsilon} (\sigma_1) (\hat{z})
    \right)^{\diamond k} \right] \right\|_{\mathcal{C}^{- \delta}}^{2 p}
    \mathd \nu^e_{\epsilon} \right]^{\frac{1}{2 p}} \times \nonumber\\
    &  & \times \left( \int \left( \sum_{k = 1}^{2 n} \| \sigma_k
    \|_{\mathcal{C}^{- \delta}_{\ell}} \right)^{2 \beta_1 (2 n - 1) p} \mathd
    \nu_{\epsilon}^e \right)^{\frac{1}{2 p}} \nonumber
  \end{eqnarray}
  Since the law of $\sigma_k$ is a multilinear functional of a Gaussian random
  field, by hypercontractivity we have that there exists $m_1, m_2 \in
  \mathbb{N}$ and $\alpha > 0$ such that
  \begin{eqnarray}
    \int \left\| | f' (x) |^{1 / 2 n} \left( \sigma_k (z) - \left(
    \mathfrak{a}_{\frac{1}{N} - \epsilon} (\sigma_1) (\hat{z})
    \right)^{\diamond k} \right) \right\|_{\mathcal{C}^{- \delta}}^{2 p}
    \mathd \nu^e_{\epsilon} & \lesssim & (p - 1)^{p u_1} N^{- \alpha}
    \nonumber\\
    \int \left( \sum_{k = 1}^{2 n} \| \sigma_k \|_{\mathcal{C}^{-
    \delta}_{\ell}} \right)^{2 \beta_1 (2 N - 1) p} \mathd \nu_{\epsilon}^e &
    \lesssim & (p - 1)^{p u_2} \nonumber
  \end{eqnarray}
  uniformly in $\epsilon$. This proves inequality~{\eqref{eq:IP1}}. We note
  that
  \[ \sum_{k = 0}^{2 n} \binom{2 n}{k} \left( \mathfrak{a}_{\frac{1}{N} -
     \epsilon} (\sigma_1) (\hat{z}) \right)^{\diamond k} f' (x) 
     \bar{\theta}^{2 n - k} (\hat{z}) = f' (x) H_{2 n} \left(
     \mathfrak{a}_{\frac{1}{N} - \epsilon} (\sigma_1) (\hat{z}) + \bar{\theta}
     (\hat{z}) ; c_{N, \epsilon} \right) \]
  where $H_{2 n} (x ; \lambda)$ is the Hermite polynomial of degree $2 n$ of a
  Gaussian random variable with variance $\lambda$ and $c_{N, \epsilon}
  =\mathbb{E} \left[ \left( \mathfrak{a}_{\frac{1}{N}} \asterisk (\mathcal{I}
  \xi) \right)^2 \right] \sim \log (N)$, as $N \rightarrow + \infty$ for any
  $\epsilon > 0$. It is simple to see that, for any $x \in \mathbb{R}^2$ and
  $\hat{z} = (x, z) \in \mathbb{R}^2 \times \mathbb{T}^2$,
  \begin{eqnarray}
    f' (x) H_{2 n} \left( \mathfrak{a}_{\frac{1}{N} - \epsilon} (\sigma_1)
    (\hat{z}) + \bar{\theta} (\hat{z}) ; c_{N, \epsilon} \right) & \lesssim &
    a_n f' (x) \left( \mathfrak{a}_{\frac{1}{N} - \epsilon} (\sigma_1)
    (\hat{z}) + \bar{\theta} (\hat{z}) \right)^{2 n} \nonumber\\
    &  & + f' (x) \sum_{k = 1}^m a_k \left( \mathfrak{a}_{\frac{1}{N} -
    \epsilon} (\sigma_1) (\hat{z}) + \bar{\theta} (\hat{z}) \right)^{2 N - 2
    k} c_{N, \epsilon}^k \nonumber\\
    & \lesssim & \tilde{a}_n f' (x) \left( \mathfrak{a}_{\frac{1}{N} -
    \epsilon} (\sigma_1) (\hat{z}) + \bar{\theta} (\hat{z}) \right)^{2 N} - f'
    (x) c_{N, \epsilon}^n \nonumber\\
    & \lesssim & - f' (x) c_{N, \epsilon}^n \lesssim - f' (x) (\log (N))^n,
    \nonumber
  \end{eqnarray}
  where we used that $f' \leqslant 0$ and $a_N, \tilde{a}_N' > 0$. This proves
  inequality~{\eqref{eq:IP2}}.
  
  Now we use Lemma~V.5 of~{\cite{Simon1974}} and the fact that
  inequality~{\eqref{eq:IP1}} and~{\eqref{eq:IP2}} are uniform in $\epsilon$,
  to prove that there exist two constants $b, \alpha'' > 0$ depending only on
  $f'$ and $m$ but not on $\epsilon > 0$ such that
  \begin{equation}
    \nu_{\epsilon}^e (E (\sigma, \bar{\theta}) > b \log (K)) \leqslant e^{-
    K^{\alpha''}} . \label{eq:measure}
  \end{equation}
  Inequality~{\eqref{eq:measure}} implies that $\exp (E) \in L^p
  (\nu^e_{\epsilon})$ for any $p \geq 1$, and we have also a uniform bound on
  the $L^p$ norm of $\exp (E)$ with respect to $\epsilon > 0$. This concludes
  the proof of the lemma.
\end{proof*}

\begin{proof*}{Proof of Theorem~\ref{theorem_polynomialreduction}}
  First we introduce the equation
  \begin{equation}
    (- \Delta_x - \Delta_z + m^2) (\phi) +
    \mathfrak{a}_{\varepsilon}^{\asterisk 2} \asterisk (H_{2 n} (\phi ;
    c_{\epsilon})) = \mathfrak{a}_{\epsilon} \asterisk \xi
    \label{eq:powerproof1},
  \end{equation}
  where $c_{\epsilon} =\mathbb{E} [(\mathfrak{a}_{\epsilon} \asterisk
  \xi)^2]$. Equation~{\eqref{eq:powerproof1}} is of the
  form~{\eqref{eq:main1}}, and the potential $V_{\epsilon} (y) = H_{2 n} (y ;
  c_{\epsilon})$, where $y \in \mathbb{R}$, satisfies Hypothesis~QC since it
  is an even polynomial with positive leading coefficient. This means that,
  for any $\epsilon > 0$, there exists a weak solution $\nu_{\epsilon}$ to
  equation~{\eqref{eq:powerproof1}}, such that
  \[ \int_{\mathcal{W}} F (\phi (0, \cdot)) \Upsilon_f^{\epsilon} (\phi)
     \mathd \nu_{\epsilon} = \int_{B^{- \delta}_{\mathfrak{p}, \mathfrak{p},
     \ell} (\mathbb{T}^2)} F (\omega) \mathd \kappa_{\epsilon} (\omega) \]
  where $\Upsilon_f^{\epsilon} (\phi) = \int_{\mathbb{R}^2 \times
  \mathbb{T}^2} f' (x) H_{2 n} (\phi (\hat{z}) ; c_{\epsilon}) \mathd
  \hat{z}$. On the other hand if by $\nu^e_{\epsilon}$ we denote the law of
  $(\mathcal{\mathfrak{a}}_{\epsilon} \asterisk \mathcal{I} \xi, \ldots,
  (\mathcal{\mathfrak{a}}_{\epsilon} \asterisk \mathcal{I} \xi)^{\diamond 2
  n}, \phi - \mathcal{\mathfrak{a}}_{\epsilon} \asterisk \mathcal{I} \xi)$
  (where $\phi$ has law $\nu$ and $\mathcal{\mathfrak{a}}_{\epsilon} \asterisk
  \mathcal{I} \xi$ is given by equation~{\eqref{eq:powerproof1}}), we have
  that $\nu^e_{\epsilon}$ is a solution to equation~{\eqref{eq:power4}} and
  \[ \int_{\mathcal{W}} F (\phi (0, \cdot)) \Upsilon_f^{\epsilon} (\phi)
     \mathd \nu_{\epsilon} = \int_{\mathcal{W}^e} F (\phi (0, \cdot))
     \tilde{\Upsilon}_f (\sigma, \bar{\theta}) \mathd \nu^e_{\epsilon}
     (\sigma, \bar{\theta}) . \]
  Using the Galerkin approximation of Section \ref{section_discrete1} and a
  method similar to the one of Lemma \ref{lemma_tightness}, it is possible to
  prove that $\nu^e_{\epsilon} ((\mathcal{C}^{- \delta}_{\ell} (\mathbb{R}^2
  \times \mathbb{T}^2))^{2 n} \times \overline{\mathfrak{W}}_{\varepsilon}) =
  1$.
  
  By Lemma~\ref{lemma_xi} and Lemma~\ref{lemma_trace} we have that the measure
  $F (\phi (0, \cdot)) \mathd \nu_{\epsilon}^e$ weakly converges to $F (\phi
  (0, \cdot)) \mathd \nu^e$, as $\varepsilon \rightarrow 0$. Let $G_r :
  \mathcal{W}^e \rightarrow \mathbbm{R}$, $r \in \mathbb{N}$, be a sequence of
  continuous functions such that $0 \leqslant G_r \leqslant 1$, $G_r = 1$ when
  $\tilde{\Upsilon}_f (\sigma, \bar{\theta}) \leqslant r$ and $G_r = 0$ when
  $\tilde{\Upsilon}_f (\sigma, \bar{\theta}) \geqslant 2 r$ (the existence of
  this kind of functions follows from the fact that $\tilde{\Upsilon}_f
  (\sigma, \bar{\theta})$ is continuous on $\mathcal{W}^e$). Since $F (\phi
  (0, \cdot)) \mathd \nu_{\epsilon}^e$ weakly converges to $F (\phi (0,
  \cdot)) \mathd \nu^e$ then $\int_{\mathcal{W}^e} G_r F \tilde{\Upsilon}_f
  \mathd \nu^e_{\epsilon} \rightarrow \int_{\mathcal{W}^e} G_r F
  \tilde{\Upsilon}_f \mathd \nu^e$, as $\varepsilon \rightarrow 0$. On the
  other hand, by Lemma~\ref{lemma_uniform}, for all $\varepsilon > 0$ and any
  $r > 0$,
  \[ \left| \int_{\mathcal{W}^e} (G_r - 1) F \tilde{\Upsilon}_f \mathd
     \nu^e_{\epsilon} \right| \leqslant \int_{\tilde{\Upsilon}_f \geqslant r}
     | F | \tilde{\Upsilon}_f \mathd \nu^e_{\epsilon} \leqslant \frac{C_p \| F
     \|_{\infty}}{r^{p - 1}} \]
  for any $p \geq 1$ (and a similar inequality is true for $\epsilon = 0$).
  Taking $r \rightarrow + \infty$ the thesis follows.
\end{proof*}

\appendix\section{Besov spaces}\label{appendix_besov}

In this appendix we recall some results on weighted Besov spaces used in this
paper. We consider only the case of Besov spaces defined on $\mathbb{R}^n$ but
all what follows holds also for Besov spaces on $\mathbb{T}^n$ or on
$\mathbb{R}^{n_1} \times \mathbb{T}^{n_2}$ (like the ones used in Section
\ref{section_power}).

\

First we recall the definition of Littlewood-Paley block: let $\chi, \varphi$
be smooth non-negative functions from $\mathbb{R}^n$ into $\mathbb{R}$ such
that
\begin{itemize}
  \item $\tmop{supp} (\chi) \subset B_{\frac{4}{3}} (0)$ and $\tmop{supp}
  (\varphi) \subset B_{\frac{8}{3}} (0) \setminus B_{\frac{3}{4}} (0)$,
  
  \item $\chi, \varphi \leq 1$ and $\chi (y) + \sum_{j \geq 0} \varphi (2^{-
  j} y) = 1$ for any $y \in \mathbb{R}^n$,
  
  \item $\tmop{supp} (\chi) \cap \tmop{supp} (\varphi (2^{- i} \cdot)) =
  \emptyset$ for $i \geqslant 1$,
  
  \item $\tmop{supp} (\varphi (2^{- j} \cdot)) \cap \tmop{supp} (\varphi (2^{-
  i} \cdot)) = \emptyset$ if $| i - j | > 1$,
\end{itemize}
where by $B_r (x)$ we denote the ball of center $x \in \mathbb{R}^n$ and
radius $r > 0$.

We introduce the following notations: $\varphi_{- 1} = \chi$, $\varphi_j
(\cdot) = \varphi (2^{- j} \cdot)$, $D_j = \hat{\varphi}_j$ and for any $f \in
\mathcal{S}' (\mathbb{R}^n)$ we put $\Delta_j (f) = D_j \asterisk f$.
Furthermore we write, for any $\ell > 0$, \ $r_{\ell} (y) = (1 + | y |^2)^{-
\ell / 2}$, $L^p_{\ell} (\mathbb{R}^d)$ is the $L^p$ space with respect to the
norm
\[ \| f \|_{p, \ell} = \left( \int_{\mathbb{R}^n} (f (y) r_{\ell} (y))^p
   \mathd y \right)^{1 / p}, \]
where $p \in [1, + \infty]$.

\begin{definition}
  \label{definition_Besov}Consider $s \in \mathbb{R}, p, q \in [1, + \infty]$
  and $\ell \in \mathbb{R}$. If $f \in \mathcal{S} (\mathbb{R}^n)$ we define
  the norm
  \begin{equation}
    \| f \|_{B^s_{p, q, \ell}} = \left( \sum_{j \geq - 1} 2^{s q} \| \Delta_j
    (f) \|_{p, \ell}^q \right)^{1 / q} . \label{eq:norm11}
  \end{equation}
  The space $B^s_{p, q, \ell}$ is the subset of $\mathcal{S}' (\mathbb{R}^n)$
  such that the norm {\eqref{eq:norm11}} is finite. 
\end{definition}

\begin{remark}
  It is important to note that $B^s_{2, 2, \ell}$ is equal to the weighted
  Sobolev space $W^{s, 2}_{\ell} = H^s_{\ell}$ of $L^2$ distributions, and
  that for $s > 0$, $s \not{\in} \mathbb{N}$, $B^s_{\infty, \infty}$ is equal
  to $\mathcal{C}^s$, the space of H{\"o}lder functions of regularity $s$.
  Moreover we remark that $B^s_{p, q} = B^s_{p, q, 0}$.
\end{remark}

\begin{remark}
  \label{remark_decay}It is always possible to choose $\chi$ and $\varphi$ in
  such a way that there exists some constants $\gamma_{- 1}, \gamma > 0$ and
  $0 < \theta_{- 1}, \theta < 1$ such that
  \[ | D_{- 1} (y) | \lesssim \exp (- \gamma_{- 1} | y |^{\theta_{- 1}})
     \quad \tmop{and} \quad | D_0 (y) | \lesssim \exp (- \gamma | y
     |^{\theta}), \]
  (see, i.e., {\cite{SchmeisserTriebel}} Section 1.2.2 Proposition 1 or
  {\cite{MW17}}).
\end{remark}

\begin{remark}
  \label{remark_weight}In this appendix we treat only the case of polynomial
  weight. More generally all the statements here presented can be extended to
  the case of regular enough weights \ $\omega$ satisfying the following
  conditions
  \begin{equation}
    0 < \omega (y)^{\pm 1} < c \omega (w - y)^{\pm 1} \exp (d | w |)
    \label{eq:weight1}
  \end{equation}
  \begin{equation}
    | D^{\alpha} \omega (y) | \leqslant e \omega (y) \label{eq:weight2}
  \end{equation}
  where $c, d, e > 0$ and $y, w \in \mathbb{R}^n$. Here we have decided to
  present the case with polynomial weights because for finite distributions
  with respect to weights satisfying only the conditions {\eqref{eq:weight1}}
  and {\eqref{eq:weight2}} the Fourier transform is not defined. This means
  that we cannot use directly Definition \ref{definition_Besov} and a more
  technical discussion is needed. We refer to {\cite{Schott1,Schott2}} for the
  study of weighted Besov spaces with a weight satisfying the conditions
  {\eqref{eq:weight1}} and {\eqref{eq:weight2}}.
\end{remark}

We denote by $\mathfrak{F}_r$, with $r \in \mathbb{N}$, the space of
functions$f \in C^r (\mathbb{R}^n)$ with support in $B_1 (0)$ and norm $\| f
\|_{ C^r (\mathbb{R}^n)} \leq 1$. If $f \in \mathfrak{F}_r$ we
write $f_{y, \lambda} (\cdot) = \lambda^{- n} f \left( \frac{\cdot -
y}{\lambda} \right)$, $y \in \mathbb{R}^n$ and $\lambda > 0$. \

\begin{proposition}
  \label{proposition_equivalentnorm}For any $s < 0, p, q \in [1, + \infty]$
  and $\ell \in \mathbb{R}$ an equivalent norm in the space $B^s_{p, q, \ell}$
  is given by the following expression
  \begin{equation}
    \| f \|_{B^s_{p, q, \ell}} \sim \left( \int_0^1 \frac{\| \sup_{g \in
    \mathfrak{F}_r} | \langle f, g_{\cdot, \lambda} \rangle | \|_{p,
    \ell}^q}{\lambda^{s q}} \frac{\mathd \lambda}{\lambda} \right)^{1 / q} .
    \label{eq:equivalentnorm}
  \end{equation}
  where $r \in \mathbb{N}$ is the first integer bigger than $- s$ when $s <
  0$.
\end{proposition}

\begin{proof}
  Theorem 6.15, in {\cite{Triebel2006}}, proves the equivalence between norm
  {\eqref{eq:norm11}} and the norm of $B^s_{p, q, \ell}$ built using wavelets,
  while Proposition 2.4 in {\cite{HaiLabb2017}} proves the equivalence between
  the norm of $B^s_{p, q, \ell}$ built using wavelets and the norm
  {\eqref{eq:equivalentnorm}}. Combining the two results we obtain the thesis.
\end{proof}

If $f$ is a measurable function and $N \in \mathbb{N}$ we define
\[ \Delta^N_h (f) (y) = \sum_{i = 0}^N (- 1)^{N - i} \left( \begin{array}{c}
     N\\
     i
   \end{array} \right) f (y + h i) \]
for $x, h \in \mathbb{R}^n$.

\begin{proposition}
  \label{proposition_difference}For any $s > 0$, $p, q \geqslant 1$, $\ell \in
  \mathbb{R}$, and $N, M \in \mathbb{N}_0$ such that $N < s < M$ an equivalent
  norm in the space $B^s_{p, q, \ell}$ is given by the following expression
  \begin{equation}
    \| f \|_{B^s_{p, q, \ell}} \sim \| f \|_{L^p_{\ell}} + \sum_{\alpha \in
    \mathbb{N}^n, | \alpha | = N} \left( \int_{| h | < 1} | h |^{- (s - N) q}
    \| \Delta_h^M D^{\alpha} f \|_{L^p_{\ell}}^q \frac{\mathd h}{| h |^n}
    \right)^{\frac{1}{q}} . \label{eq:equivalentnorm2}
  \end{equation}
\end{proposition}

\begin{proof}
  The proof is given in {\cite{Triebel2006}} Theorem 6.9 for the case $N = 0$.
  The generic case $N > 0$ can be proved applying the techniques of Corollary
  2 Section 2.6.1 of {\cite{Triebel1992}} (where formula
  {\eqref{eq:equivalentnorm2}} is proven for Besov spaces without weight).
\end{proof}

\begin{remark}
  \label{remark_equivalentnorm}An important consequence of Proposition
  \ref{proposition_difference} is that for any $N \in \mathbb{N}$ such that $N
  < s$ we have that an equivalent norm of $B^s_{p, q, \ell}$ is given by
  \begin{equation}
    \| f \|_{B^s_{p, q, \ell}} \sim \| f \|_{L^p_{\ell}} + \sum_{\alpha \in
    \mathbb{N}^n, | \alpha | = N} \| D^{\alpha} f \|_{B^{s - N}_{p, q, \ell}}
    . \label{eq:equivalentenorm3}
  \end{equation}
\end{remark}

\begin{proposition}
  \label{proposition_inclusion}Consider $p_1, p_2, q_1, q_2 \in [1, \infty]$,
  $\ell_{1,} \ell_2 \in \mathbb{R}$ and $s_1, s_2 \in \mathbb{R}$ such that
  $s_1 - \frac{n}{p_1} > s_2 - \frac{n}{p_2}$ and $\ell_1 > \ell_2$ then
  $B^{s_2}_{p_2, q_2, \ell_2} \subset B^{s_1}_{p_1, q_1, \ell_1}$ and the
  immersion is compact.
  
  Furthermore for $1 \leq p \leq 2$
  \[ B^s_{p, p, \ell} \subset W^{s, p}_{\ell} \subset B^s_{p, 2, \ell} ; \]
  for $2 \leq p < \infty$
  \[ B^s_{p, 2, \ell} \subset W^{s, p}_{\ell} \subset B^s_{p, p, \ell} \]
  and for $p = \infty$ $W^{s, \infty}_{\ell} \subset B^s_{\infty, \infty}$.
  Each of above immersions is continuous. 
\end{proposition}

\begin{proof}
  The proof of the first part of the proposition can be found in Theorem 6.7
  in {\cite{Triebel2006}}.
  
  The second part of the proposition is proved in Theorem 6.4.4 and Theorem
  6.2.4 of {\cite{Joran1976}} for unweighted spaces. The result for weighted
  spaces follows from the fact that $f \in B^s_{p, q, \ell}$ and $g \in W^{s,
  p}_{\ell}$ if and only if $f \cdot r_{\ell} \in B^s_{p, q}$ and $g \cdot
  r_{\ell} \in W^{s, p}$ (see {\cite{Triebel2006}} Theorem 6.5 point ii) ). 
\end{proof}

\begin{proposition}
  \label{proposition_product}Consider $p_1, p_2, p_3, q_1, q_2, q_3 \in [1,
  \infty]$ such that $\frac{1}{p_1} + \frac{1}{p_2} = \frac{1}{p_3}$ and $q_1
  = q_3$ and $q_2 = \infty$. Moreover consider $\ell_{1,} \ell_2, \ell_3 \in
  \mathbb{R}$ with $\ell_1 + \ell_2 = \ell_3$ and consider $s_1 < 0$ $s_2 \geq
  0$ and $s_3 = s_1 + s_2 > 0$. Finally, consider the bilinear functional \
  $\Pi (f, g) = f \cdot g$ defined on $\mathcal{S} (\mathbb{R}^n)$ taking
  values in $\mathcal{S} (\mathbb{R}^n)$. Then there exists a unique
  continuous extension of $\Pi$ as the map
  \[ \Pi : B^{s_1}_{p_1, q_1, \ell_1} \times B^{s_2}_{p_2, q_2, \ell_2}
     \rightarrow B^{s_3}_{p_3, q_3, \ell_3} \]
  and we have, for any $f, g$ for which the norms are defined:
  \[ \| \Pi (f, g) \|_{B^{s_3}_{p_3, q_3, \ell_3}} \lesssim \| f
     \|_{B^{s_1}_{p_1, q_1, \ell_1}} \| g \|_{B^{s_2}_{p_2, q_2, \ell_2}} . \]
\end{proposition}

\begin{proof}
  The proof can be found in {\cite{MW17}} in Section 3.3 for Besov spaces with
  exponential weights. The proof for polynomial weights is similar.
\end{proof}

\begin{proposition}
  \label{proposition_interpolation}Consider $p_1, p_2, q_1, q_2 \in [1,
  \infty]$, $\ell_1, \ell_2 \in \mathbb{R}$ and $s_1, s_2 \in \mathbb{R}$. For
  any $\theta \in [0, 1]$ we write $\frac{\theta}{p_1} + \frac{1 -
  \theta}{p_2} = \frac{1}{p_{\theta}}$, $\frac{\theta}{q_1} + \frac{1 -
  \theta}{q_2} = \frac{1}{q_{\theta}}$, $\theta \ell_1 + (1 - \theta) \ell_2 =
  \ell_{\theta}$ and $\theta s_1 + (1 - \theta) s_2 = s_{\theta}$. If $f \in
  B^{s_1}_{p_1, q_1, \ell_1} \cap B^{s_2}_{p_2, q_2, \ell_2}$ then $f \in
  B^{s_{\theta}}_{p_{\theta}, q_{\theta}, \ell_{\theta}}$, and furthermore
  \[ \| f \|_{B^{s_{\theta}}_{p_{\theta}, q_{\theta}, \ell_{\theta}}} \leq \|
     f \|_{B^{s_1}_{p_1, q_1, \ell_1}}^{\theta} \| f \|_{B^{s_2}_{p_2, q_2,
     \ell_2}}^{1 - \theta} . \]
  Furthermore if $f \in W^{s_1, p_1}_{\ell_1} \cap W^{s_2, p_2}_{\ell_2}$ then
  \ $f \in W^{s_{\theta}, p_{\theta}}_{\ell_{\theta}}$, and furthermore
  \[ \| f \|_{W^{s_{\theta}, p_{\theta}}_{\ell_{\theta}}} \leq \| f
     \|_{W^{s_1, p_1}_{\ell_1}}^{\theta} \| f \|_{W^{s_2, p_2}_{\ell_2}}^{1 -
     \theta} . \]
\end{proposition}

\begin{proof}
  The proof is based on the fact that the complex interpolation
  $(B^{s_1}_{p_1, q_1, \ell_1}, B^{s_2}_{p_2, q_2, \ell_2})_{\theta}$ of the
  two spaces $B^{s_1}_{p_1, q_1, \ell_1}$, $B^{s_2}_{p_2, q_2, \ell_2}$ is
  given by $B^{s_{\theta}}_{p_{\theta}, q_{\theta}, \ell_{\theta}}$. This
  interpolation is proved in~{\cite{Joran1976}} Theorem 6.4.5 point (6) for
  unweighted space. For weighted space it follows from the fact that $f \in
  B^{s_i}_{p_i, q_i, \ell_i}$ if and only if $f \cdot r_{\ell_i} \in
  B^{s_i}_{p_i, q_i}$.
  
  A similar reasoning, using the interpolation proved in~{\cite{Joran1976}}
  Theorem 6.4.5 point (7), holds also for the second inequality.
\end{proof}

\bibliographystyle{plain}

\bibliography{multi-elliptic_rev}

\end{document}